\documentclass[letter,11pt,reqno]{amsart}

\usepackage[utf8]{inputenc}

\usepackage{etex}

\usepackage{xcolor}

\definecolor{verydarkblue}{rgb}{0,0,0.5}

\usepackage[
    breaklinks,
    colorlinks,
    citecolor=verydarkblue,
    linkcolor=verydarkblue,
    urlcolor=verydarkblue,
    pagebackref=true,
    hyperindex
]{hyperref}

\backrefenglish

\usepackage{fancyhdr}

\usepackage[
    hscale=0.70,
    vscale=0.775,
    headheight=13pt,
    centering,
]{geometry}

\usepackage{amsmath}
\usepackage{amsthm}
\usepackage{amssymb}
\usepackage{bm}
\usepackage{mathtools}
\usepackage{stmaryrd}
\usepackage{mathdots}
\usepackage{framed}
\usepackage[capitalize]{cleveref}
\usepackage{array}
\usepackage[alphabetic]{amsrefs}
\usepackage[all,cmtip]{xy}

\theoremstyle{plain}

\crefname{introtheorem}{Theorem}{Theorems}

\newtheorem{theorem}{Theorem}[section]
\newtheorem{proposition}[theorem]{Proposition}
\newtheorem{lemma}[theorem]{Lemma}

\newtheorem{definition-theorem}[theorem]{Definition-Theorem}

\theoremstyle{definition}

\newtheorem{definition}[theorem]{Definition}

\newtheorem{case}{Case}

\theoremstyle{remark}

\newtheorem{remark}[theorem]{Remark}
\newtheorem{example}[theorem]{Example}

\numberwithin{figure}{section}

\numberwithin{equation}{section}

\IfFileExists{./article-style.tex}{\input{article-style.tex}}{}

\usepackage{marginnote}

\def\Z{{\mathbb Z}}
\def\Q{{\mathbb Q}}
\def\R{{\mathbb R}}

\def\S{{\mathbb S}}

\def\L{{\mathbb L}}
\def\S{{\mathbb S}}

\def\A{{\mathbb A}}
\def\P{{\mathbb P}}

\def\cB{\mathcal{B}}
\def\cC{\mathcal{C}}
\def\cD{\mathcal{D}}

\def\cF{\mathcal{F}}

\def\cM{\mathcal{M}}

\def\cM{\mathcal{M}}
\def\cX{\mathcal{X}}

\def\cZ{\mathcal{Z}}

\def\rI{\mathrm{I}}

\def\I{\mathcal{I}}

\def\O{\mathcal{O}}

\def\fa{\mathfrak{a}}

\def\fn{\mathfrak{n}}

\def\bL{{\bf L}}

\def\a{\alpha}
\def\b{\beta}
\def\g{\gamma}

\def\f{\phi}
\def\ff{\psi}

\def\n{\nu}
\def\m{\mu}
\def\om{\omega}
\def\p{\pi}
\def\r{\rho}
\def\s{\sigma}
\def\t{\tau}
\def\x{\xi}

\def\D{\Delta}
\def\G{\Gamma}
\def\S{\Sigma}
\def\Om{\Omega}

\def\.{\cdot}
\def\^{\widehat}
\def\~{\widetilde}
\def\o{\circ}
\def\ov{\overline}

\def\rat{\dashrightarrow}
\def\surj{\twoheadrightarrow}
\def\inj{\hookrightarrow}

\def\({\left(}
\def\){\right)}

\def\limdir{\underrightarrow{\lim}}
\def\bla{\hskip5pt}

\makeatletter
\newcommand*{\da@rightarrow}{\mathchar"0\hexnumber@\symAMSa 4B }
\newcommand*{\da@leftarrow}{\mathchar"0\hexnumber@\symAMSa 4C }
\newcommand*{\xdashrightarrow}[2][]{%
  \mathrel{%
    \mathpalette{\da@xarrow{#1}{#2}{}\da@rightarrow{\,}{}}{}%
  }%
}
\newcommand{\xdashleftarrow}[2][]{%
  \mathrel{%
    \mathpalette{\da@xarrow{#1}{#2}\da@leftarrow{}{}{\,}}{}%
  }%
}
\newcommand*{\da@xarrow}[7]{%
  \sbox0{$\ifx#7\scriptstyle\scriptscriptstyle\else\scriptstyle\fi#5#1#6\m@th$}%
  \sbox2{$\ifx#7\scriptstyle\scriptscriptstyle\else\scriptstyle\fi#5#2#6\m@th$}%
  \sbox4{$#7\dabar@\m@th$}%
  \dimen@=\wd0 %
  \ifdim\wd2 >\dimen@
    \dimen@=\wd2 %
  \fi
  \count@=2 %
  \def\da@bars{\dabar@\dabar@}%
  \@whiledim\count@\wd4<\dimen@\do{%
    \advance\count@\@ne
    \expandafter\def\expandafter\da@bars\expandafter{%
      \da@bars
      \dabar@ 
    }%
  }%
  \mathrel{#3}%
  \mathrel{%
    \mathop{\da@bars}\limits
    \ifx\\#1\\%
    \else
      _{\copy0}%
    \fi
    \ifx\\#2\\%
    \else
      ^{\copy2}%
    \fi
  }%
  \mathrel{#4}%
}
\makeatother

\renewcommand{\and}{ \ \ \text{ and } \ \ }

\def\red{\mathrm{red}}
\def\an{\mathrm{an}}

\def\bir{\mathrm{bir}}
\def\div{\mathrm{div}}
\def\qm{\mathrm{qm}}

\def\Jac{\mathrm{Jac}}

\def\PR{\mathrm{PR}}

\DeclareMathOperator{\codim} {codim}

\DeclareMathOperator{\Spec} {Spec}

\DeclareMathOperator{\ord} {ord}

\DeclareMathOperator{\Supp} {Supp}
\DeclareMathOperator{\dirlim} {\varinjlim}
\DeclareMathOperator{\invlim} {\varprojlim}

\DeclareMathOperator{\val} {val}

\DeclareMathOperator{\lct} {lct}
\DeclareMathOperator{\id} {id}
\DeclareMathOperator{\Fitt} {Fitt}

\DeclareMathOperator{\Func} {\cF}

\DeclareMathOperator{\Var}{Var}

\DeclareMathOperator{\wt}{wt}
\DeclareMathOperator{\Sk} {Sk}

\DeclareMathOperator\FF{F}
\DeclareMathOperator\IF{IF}
\DeclareMathOperator\MM{M\R}

\DeclareMathOperator\MF{MF}
\DeclareMathOperator\IMF{IMF}

\DeclareMathOperator\GG{G}
\DeclareMathOperator\IG{IG}
\DeclareMathOperator\MG{MG}
\DeclareMathOperator\IMG{IMG}

\DeclareMathOperator\PP{P}
\DeclareMathOperator\IP{IP}
\DeclareMathOperator\MP{MP}
\DeclareMathOperator\MQ{MQ}
\DeclareMathOperator\IMP{IMP}
\DeclareMathOperator\IMQ{IMQ}
\DeclareMathOperator\MA{M\A}

\DeclareMathOperator\Snc{Lsnc}

\DeclareMathOperator{\QM}{QM}

\newcommand{\relIMG}[1]{\rI_{#1}\!\MG}
\newcommand{\relIG}[1]{\rI_{#1}\!\GG}
\newcommand{\relIP}[1]{\rI_{#1}\!\PP}
\newcommand{\relIMP}[1]{\rI_{#1}\!\MP}
\newcommand{\relIMQ}[1]{\rI_{#1}\!\MQ}

\def\ae{\mathrm{ae}}
\def\aeeq{\overset{\ae}=}

\mathchardef\mhyphen="2D

\makeatletter
\@namedef{subjclassname@2020}{%
  \textup{2020} Mathematics Subject Classification}
\makeatother

\title{Motivic integration on Berkovich spaces}

\author{Tommaso de Fernex}
\address{Department of Mathematics, University of Utah, Salt Lake City, UT 48112, USA}
\email{{\tt defernex@math.utah.edu}}

\author{Chung Ching Lau}
\email{{\tt malccad@gmail.com}}


\thanks{%
The research of the first author was partially supported by NSF Grants DMS-1700769 and DMS-2001254.
The research of the second author was partially supported by a Croucher Foundation Fellowship. 
}

\begin{document}

\begin{abstract}
We define a motivic measure on the Berkovich analytification of an algebraic variety
defined over a trivially valued field, 
and introduce motivic integration in this setting. 
The construction is geometric with a similar spirit as Kontsevich's original definition,
and leads to the formulation of a functorial theory
which mirrors, in this aspect, the approach of
Cluckers and Loeser via constructible motivic functions. 
A version of the integral
over nontrivially valued fields and its relation to Hrushovski and Kazhdan's integration
are also discussed. 
\end{abstract}

\maketitle

\section{Introduction}

Motivic integration was introduced by Kontsevich in 1995 \cite{Kon95} and later
developed by Denef and Loeser in \cite{DL99};
see also \cite{Bat99,DL02,Loo02,Rei02} for further developments
and applications to stringy invariants. 
Following these pioneer works, more advanced theories of integration were later discovered
using model theory. 
Among them, we recall 
the theory of constructible motivic functions due to Cluckers and Loeser \cite{CL08,CL10}.
According to this theory, performing motivic integration corresponds to taking push-forward to a point, 
and pushing forward along other morphisms can be interpreted as integration along the fibers. 
A key feature of Cluckers--Loeser's theory is the functoriality of push-forward, 
which in turns leads to a Fubini type theorem.

In this paper, we propose a new approach to motivic integration 
which aims to maintain the same geometric flavor of the original construction 
while reproducing some of the features of \cite{CL08}.
The goal is to offer a concrete approach which
does not rely on model theory but rather on geometric properties such as 
resolution of singularities and the weak factorization theorem.
Berkovich spaces are used in place of arc spaces, 
and the resulting theory can be viewed as a combination of motivic and Lebesgue measures, 
reflecting a unique feature of Berkovich analytifications
whose topology is a blend of Zariski and Euclidean topologies.
We also present an atomic version of integration, where the measure is concentrated on integral valuations
(viewed as the integral points on the Berkovich space) and the integral takes values in the usual
motivic ring, fully recovering the classical motivic integral. 
Finally, we discuss a relative version of motivic integration over nontrivially 
valued fields which is closely related to, but different from, 
Hrushovski--Kazhdan's integration \cite{HK06}. 

The following paragraphs provide a more detailed overview of the main contents of the paper.

\subsubsection*{Measure and integration}

The original definition of motivic integration is purely geometric.
Working on a smooth variety $X$, one assigns a motivic volume to 
cylinders in the arc space $X_\infty$ of $X$,
and integrate functions of the form $\L^{-\ord(Z)}$ where $Z \subset X$ is a closed subscheme.
Here, $\L := [\A^1_X]$ is the Lefschetz motive,
and the function $\ord(Z) \colon X_\infty \to \R \cup \{\infty\}$
is defined by $\a \mapsto \ord_\a(\I_Z)$ where $\I_Z \subset \O_X$ is the ideal sheaf of $Z$. 
The motivic integral $\int_{X_\infty} \L^{-\ord(Z)} d\m^{X_\infty}$
is expressed as an infinite
sum by stratifying the arc space (away from a set of measure zero)
as a union of cylinders where the function takes constant values.
It follows by resolution of singularities that the integral can be realized 
in the localization of the Grothendieck ring $K_0(\Var_X)$ at the classes $[\P^a_X]$
of all projective spaces. 

Here we propose a similar approach where the arc space $X_\infty$ of the variety $X$
is replaced by its analytification $X^\an$ in the sense of Berkovich \cite{Ber90}.
There are some advantages in performing integration over $X^\an$ that are discussed below. 
The main focus of this paper is on the case where the variety $X$ is defined over 
a field with trivial norm, a setting that is in line with motivic integration
on arc spaces, but a similar theory can be developed 
when $X$ is defined over a nontrivially valued field
(the function field of a complete discrete valuation ring, to be precise), 
which relates more closely, for instance, to \cite{LS03,HK06};
this setting is discussed in the last section of the paper. 
We will mostly restrict to
the space of real valuations $X^{\val} \subset X^\an$;
we will discuss a way to extend the measure in a meaningful way to the
space $X^\beth \subset X^\an$ of semi-valuations centered in $X$,
and explain the obstruction to extend it nontrivially to the whole Berkovich space $X^\an$
when $X$ is not proper. 

The approach relies on the approximation of $X^{\val}$ via certain skeleta, which are fans
of quasi-monomial valuations \cite{Ber99,Thu07}.
While this should be reminiscent of the original approach to motivic integration based on the approximation 
of arc spaces via jet schemes, the two approximations are actually quite different in nature, 
with the approximation of $X^{\val}$ by skeleta more in line, in fact, with the construction of 
the Riemann--Zariski space. 

To get into some details, every log resolution $\p \colon X_\p \to X$, paired with
a simple normal crossing divisor $D_\p = \sum_{i=1}^r D_i$ containing the exceptional locus,  
determines a set $\Sk_\p \subset X^{\val}$ of quasi-monomial valuations, called a `skeleton', 
via the toroidal structure given by the embedding of $X_\p \setminus D_\p$. 
This set $\Sk_\p$ has a natural fan structure indexed by the subsets $I \subset \{1,\dots,r\}$, 
and can be identified with the cone
over the dual complex $\cD(D_\p)$ associated to $D_\p$.
The valuation space 
$X^{\val}$ is the inverse limit of retraction maps $r_\p \colon X^{\val} \to \Sk_\p$. 
Log discrepancies (more precisely, {Mather log discrepancies}, if $X$ is singular)
determine a function $\^A_X$ on $X^{\val}$ that is linear homogeneous on the
faces $\Sk_{\p,I}$ of $\Sk_\p$.
This function is modeled over the log discrepancy function $A_X$
constructed in \cite{JM12,BdFFU15}.
A technical requirement here is that all resolutions $\p$ must factor through the Nash blow-up of $X$.

The Lebesgue measure naturally defined on each face $\Sk_{\p,I}$, rescaled by $e^{-\^A_X}$
and weighted by the motive $[D_I^\o]$ of the stratum 
$D_I^\o := \bigcap_{i\in I} D_i \setminus \bigcup_{j \not\in I} D_j$, 
defines a measure on $\Sk_\p$.
These measures, defined for all $\p$, glue together 
to a measure $\m_X$ on $X^{\val}$ which gives a motivic volume on `cylinders' $r_\p^{-1}(S)$
where $S \subset \Sk_\p$ is any measurable set. Explicitly, for any such set we have 
\[
\m_X(r_\p^{-1}(S)) = \sum_I [D_I^\o] \int_{S \cap \Sk_{\p,I}} e^{-\sum a_i x_i} dx_1\dots dx_{|I|}
\]
where the $x_i$ are suitable Euclidean coordinates on $\Sk_{\p,I}$ and
$a_i = \val_{D_i}(\Jac_\p)$. 
One can interpret this measure as a motivically weighted pull-back of 
suitably normalized Lebesgue measures on all tropicalizations of the underlying variety $X$.
The measure $\m_X$ takes values in the ring
$\MM_X := (K_0(\Var_X)/(\L - 1)) \otimes_\Z\R$.
Taking Euler characteristics, one gets
a measure with values in $\R$.

A similar approach leads to the definitions of integrable motivic functions and motivic integrals. Real valued functions on $X^{\val}$
provide an already interesting example of motivic functions.
For instance, every closed subscheme $Z \subset X$ determines a function $|\I_Z| \colon X^{\val} \to \R$
given by $x \mapsto |\I_Z|_x$. 
This function is integrable and so is any positive
real power of it. If $\p \colon X_\p \to X$ is a
log resolution of $(X,Z)$ and $\I_Z\.\O_{X_\p} = \O_{X_\p}(-b_i D_i)$, then for every $t > 0$
we have 
\[
\int_{X^{\val}} |\I_Z|^t \.d\m_X 
= \sum_I [D_I^\o] \int_{\Sk_{\p,I}} e^{-\sum (a_i+ t b_i) x_i} dx_1\dots dx_{|I|}
\]
where, as before, $a_i = \val_{D_i}(\Jac_\p)$. 
Thinking of elements of $X^{\val}$ as valuations, the integral
can be equivalently written in the form $\int_{X^{\val}} e^{-t \ord(Z)} d\m_X$, 
a notation that makes the connection to the usual motivic integration more evident. 

Showing that the measure and the integral are well defined requires a proof, but once this is
done, properties like the change-of-variables formula follow immediately.
The latter states that if $h \colon Y \to X$ is a resolution of singularities and
$h^{\val} \colon Y^{\val} \to X^{\val}$ is the induced function (a bijection, in this case),
then for any integrable motivic function $f$ on $X^{\val}$ we have
\[
\int_{X^{\val}} f\,d\m_X = \int_{Y^{\val}} (f\o h^{\val})\,|\Jac_h| \,d\m_Y
\]
via the natural push-forward $h_*\colon \MM_Y \to \MM_X$.

\subsubsection*{Functoriality}

The class of integrable motivic functions considered above, 
which we denote by $\IMF(X^{\val})$,
is defined via a direct limit construction using pull-back maps 
along the natural retraction maps between skeleta.
In order to provide a functorial theory, we need enlarge this class 
using an inverse limit construction via push-forward maps. 
The resulting class of functions is denoted by $\IMG(X^{\val})$. 
This requires some additional bookkeeping.
At the finite level, when tracing the functions
on the skeleton $\Sk_\p$ determined by a log resolution $\p \colon X_\p \to X$, 
we need to keep track of information coming from higher resolutions. 
For this reason, we work with certain vector functions, and
refer to elements in $\IMG(X^{\val})$
as integrable motivic Functions, with a capital \emph{F}. 

More generally, we define a class of {integrable motivic Functions} $\relIMG{B}(X^{\val})$ on $X^{\val}$, 
where the symbol $B$ represents a `base' which can be a variety dominated by $X$, 
the analytification of such a variety, or something in between. 
Integrability is intended relatively to such a base. 
The notation $\IMG(X^{\val})$ is reserved for the case where $B = X$. 

The class of integrable motivic functions is stable under push-forward. 
Given a dominant morphism of varieties $b \colon X \to Y$, and working over a common base $B$, 
we define a push-forward map $b_! \colon \relIMG{B}(X^{\val}) \to \relIMG{B}(Y^{\val})$
and prove functoriality.
As in \cite{CL08}, which has served as a guide and inspiration for this part of the paper, 
push-forwards are interpreted as integrations
along fibers, and functoriality as a Fubini theorem. 

Taking $B = Y^{\val}$, we have the projection formula
\[
b_!(b^*(f) \. \bm{g}) = f \. b_!(\bm{g})
\]
where $\bm{g} \in \relIMG{Y^{\val}}(X^{\val})$ and $f$ is any motivic function on $Y^{\val}$. 
Motivic integration over $X$ is realized by taking $B = X$ and setting
\[
\int_{X^{\val}} \bm{g}\, d\m_X := t_X(\bm{g})
\]
for any $\bm{g} \in \IMG(X^{\val})$, where
$t_X \colon \IMG(X^{\val}) \to \MM_X$ is the corresponding push-forward map.
The two approaches are related via a naturally defined inclusion
$\IMF(X^{\val}) \inj \IMG(X^{\val})$ that is compatible with
the respective definitions of integrals.

\subsubsection*{Atomic approach}

Neither the original definition of motivic integration
nor the theory of Cluckers and Loeser require working modulo $\L-1$.
This is however necessary in the approach followed in this paper 
where certain motivic denominators (classes of projective spaces)
that appear in the usual motivic integral are absorbed as numerical denominators
(the Euler characteristic of those projective spaces) that are now coming 
out of the Lebesgue
component of the measure. This allows us to enlarge the algebra of measurable sets and
work with real valued functions such as $e^{-\ord(Z)}$.
One can push this point of view even further by taking 
Euler characteristics and working exclusively with real valued functions. 
Another characteristic of this approach is that it does not require taking any localization
of the Grothendieck ring.

Using an atomic approach that is more line with with
\cite{CL08}, we also develop a parallel theory of motivic integration
on Berkovich spaces where the measure takes 
values in the same motivic ring as the original theory of motivic integration
without needing to work modulo $\L-1$.
This comes at the cost of reducing the algebra of measurable sets 
and requires integrating
functions that looks more like $\L^{-\ord(Z)}$, just like in the usual motivic integration.
However, it has the advantage to fully recover the original motivic integration.

\subsubsection*{Nontrivially valued fields}

The same approach used to define motivic integration on the analytification of a variety 
over a field $k$ with trivial norm can be adapted to the case where $X$ is a proper variety
over a valued field $(K,v)$. Here we consider the case where $K$ is the fraction field of 
a complete discrete valuation ring $R$ with residue field $k$ of characteristic zero. 
Up to a non-canonical isomorphism, we identify $R$ with $k[[t]]$ and $K$ with $k(\hskip-1pt(t)\hskip-1pt)$.
We set $\D := \Spec R$ and assume that there exist a smooth curve $C$ and a variety $Y$ over $k$
such that $X$ is obtained by base-change via a non-constant map $\D \to C$
from a proper flat morphism $Y \to C$.
Under suitable assumptions on existence of resolution of singularities
and weak factorization, much of what we do can be extended to the case where $R$ is 
any complete discrete valuation ring.

The analytification $X^\an$ of $X$ over $K$ implicitly assumes having fixed a normalization of the valuation $v$
(i.e., an embedding of the value group $\G_v$ in $\R$). 
While changing the normalization does not change $X^\an$ as an analytic space, 
it may affect the measure we wish to define on this space. 
It turns out, in fact, that the right thing to do is to work
with all normalizations at once, packaged together as a fibered topological space
$X^\an_{(0,\infty)} \to (0,\infty)$
with fiber over $b \in (0,\infty)$ equal to the analytification corresponding 
to the normalization given by $v(t) = b$.
We define a motivic measure on $X^\an_{(0,\infty)}$ in a similar way as it was done over constant fields, 
and interpret it as an `average measure' on the fibers, which are all homeomorphic to $X^\an$, 
by assigning $(0,\infty)$ measure 1 using the volume form $e^{-u}du$. 
The same approach leads to the definition of motivic integral. 
The Mather log discrepancy function is replaced in this context 
by the weight function $\wt_\om$, which depends on the choice of a 
nonzero rational canonical form $\om$ on $X$, see \cite{MN15}.
The resulting measure is again a mash of motivic and Lebesgue measures. Alternatively, one can 
develop an atomic version in this setting similarly to what we did in the trivially valued case.

The motivic contribution to the integral comes from a specialization map
(a relative version of the center map) just like in Hrushovski--Kazhdan's definition \cite{HK06}.
The fan structure of the skeleta, however,
is encoded differently compared to Hrushovski--Kazhdan's integration. 
The space of integration itself is in fact different, as in our definition 
we work on a bigger space where the valuation of the ground field is
subject to rescaling. One advantage of working on this enlarged space
becomes manifest in the atomic approach, as the support of the 
measure (the integral valuations) lie over different normalizations
of the valuation on the ground field. 

Our space of integration sits naturally in a hybrid space.
We exploit this feature to prove a comparison theorem relating different integrals. 
Working over a valued field $K$ as above, motivic integrals on $X^\an$ can be computed from
pairs of integrals on the valuation space $X^{\val}_0 \subset X^\an_0$ of $X$ 
viewed as a variety over $K$ with the trivial norm,
and on the valuation space $Y^{\val}$ of $Y$ as a variety over $k$. 
Starting with a Cartier divisor $B$ on $Y$ and setting 
$\cB := B \times_C\D$ and $B_K := B \times_X\Spec K$, for every $s \in \R$ we have
\[
\int_{X^\an} e^{-\ord(s\cB)}\,d\ov\m_\om = 
\int_{Y^{\val}} e^{-\ord(K_Y+sB)} \,d\m_Y - 
\int_{X^{\val}_0} e^{-\ord(K_X+sB_K)} \,d\m_X
\]
where $K_X$ and $K_Y$ are canonical divisors defined by the canonical form on $X$ and $Y$,
$\ov\m_\om$ is the `average' motivic measure on $X^\an$ mentioned above, 
and all integrals are viewed as taking values in $\MM(\D)$ via the obvious maps.

\subsubsection*{Acknowledgements}
We thank the referee for useful comments and suggestions.

\section{Preliminaries on Berkovich analytification}
\label{s:preliminaries}

Throughout this section, we work over an algebraically closed field $k$ of characteristic zero.
We regard $k$ as being equipped with the trivial norm.

\subsection{Berkovich analytification}
\label{ss:Ber-analytification}

The Berkovich analytification $X^\an$ of a scheme $X$ over $k$ is a topological space 
parameterizing multiplicative seminorms on $X$ that are trivial on $k$
(this space carries an analytic structure, but we will not use this).
For the convenience of the reader, we briefly recall the construction.
We refer to \cite{Ber90,Thu07} for further details.

For any affine scheme $X = \Spec A$ over $k$, $X^\an$ is defined to
be the set of multiplicative seminorms $|\bla|_x$ on $A$ that restrict to the trivial norm on $k$.
It is endowed with the weakest topology such that for any given $f\in A$, 
the map $|\bla|_x \mapsto |f|_x$ is continuous.
We denote by $\p_X \colon X^\an \to X$ the continuous map sending a seminorm to its kernel. 
If $X$ is affine and $U \subset X$ is any open subscheme, then we define $U^\an := \pi_X^{-1}(U)$ with the induced topology. Note that this is compatible with the above definition if $U$ is affine. 
For an arbitrary scheme $X$ over $k$, we choose an open affine cover $\{U_i\}$ and
define $X^\an$ by gluing the spaces $U_i^\an$ along $(U_i\cap U_j)^\an$.
The projections $\p_{U_i}$ also glue together to give a continuous map $\p_X \colon X^\an \to X$. 
We write $x$ when we think of a multiplicative seminorm $|\bla|_x$ as a point of $X^\an$, 
and denote by $\x_x \in X$ its image via $\p_X$. The point $x$ 
can be equivalently viewed as a real valuation $v_x := -\log |\bla|_x$ of the residue field $k(\x_x)$ of $\x_x$. 
We can therefore think of a point $x$ of the Berkovich analytification of $X$ as a pair $(\x_x,v_x)$.
The space $X^\an$ is equipped with an $\R_{> 0}$-action given by rescaling the values of a valuation $v_x$.

We denote by $X^\beth \subset X^\an$ the compact analytic subspace consisting of the points
$x = (\x_x,v_x)$ such that the valuation $v_x$ of $k(\x_x)$ has center in the closure of $\x_x$ in $X$;
we say for short that $v_x$ has center in $X$, and denote the center by $c_X(v_x)$. 
Note that $X^\beth = X^\an$ whenever $X$ is proper over $k$. 
When $X$ is a variety, we denote by $X^\bir \subset X^\an$ the inverse image of the generic point of $X$
under the map $\p_X$, and let $X^{\val} := X^\bir \cap X^\beth$.
This is the set of real valuations on the function field of $X$ with center in $X$; 
we will sometimes denote its points simply by $v_x$. 
The construction of these spaces is functorial. In particular, any morphism 
of schemes $p \colon X \to Y$ over $k$ induces maps $p^\an \colon X^\an \to Y^\an$
and $p^\beth \colon X^\beth \to Y^\beth$, 
and if $p$ is a dominant morphism of varieties
then these maps restrict to a map $p^{\val} \colon X^{\val} \to Y^{\val}$.

\subsection{Skeleta and quasi-monomial valuations}
\label{ss:skeleta-QM}

Let $X$ be an algebraic variety over $k$.
We start by fixing some terminology.

\begin{definition}
A \emph{simple normal crossing (snc) model} $(X_{\p},D_\p)$ over $X$ is given by a log resolution
$\p \colon X_{\p} \to X$ and a reduced simple normal crossing divisor $D_\p$ on $X_{\p}$
such that
$\p \colon X_{\p} \to X$ factors through the Nash blow-up of $X$,
$D_\p$ contains the exceptional locus of $\p$,
and every stratum of $D_\p$ is irreducible.
We refer to $D_\p$ as the \emph{boundary divisor} of the snc model.
A snc model $(X_{\p},D_\p)$ is said to be a {\it log resolution} of a non-zero ideal
sheaf $\fa \subset \O_X$ if $\fa\.\O_{X_\p}$ is a locally principal ideal
cosupported within the support of $D_\p$.
\end{definition}

\begin{remark}
The condition that $\p$ factors through the Nash blow-up is to ensure that the
Jacobian ideal $\Jac_\p := \Fitt^0(\Om_{X_\p/X}) \subset \O_{X_\p}$ is locally principal.
This will become useful in order to compare Jacobian ideals on
different snc models in terms of combinatorial information 
attached to their respective boundary divisors. 
\end{remark}

It will be convenient to enlarge the class of snc models to allow singularities
away from the divisor. 

\begin{definition}
\label{d:snc-model}
A \emph{local snc model} $(X_{\p},D_\p)$ over $X$ is given by a 
proper birational morphism $\p \colon X_{\p} \to X$ and a 
divisor $D_\p$ on $X_{\p}$
such that $X$ is covered by two open sets $V$ and $W$ with the following properties:
\begin{enumerate}
\item
$\Supp(D_\p) \subset \p^{-1}(V)$ and $(\p^{-1}(V), D_\p)$ is a snc model over $V$, and
\item
the restriction of $\p$ to $\p^{-1}(W)$ gives an isomorphism $\p^{-1}(W) \cong W$. 
\end{enumerate}
\end{definition}

The main reason for enlarging the class from snc models to local snc models
is to include the variety $X$ itself, with the identity map $\id_X \colon X \to X$
and empty boundary divisor. While for most of the paper this is not needed
and one can restrict to working with snc models, allowing the identity map $\id_X \colon X \to X$
to be counted as a model
will come useful in \cref{s:pushforward} in order to realize
motivic integrals as a push-forwards. 

We use the symbol $\p$ to denote a local
snc model and not just the underlying morphism,
and hence write $\p \colon (X_{\p},D_\p) \to X$. 
To avoid cluttering the notation, we will typically drop the label $\p$ from the divisor 
and just write $\p \colon (X_\p,D) \to X$ for the model. We stress, however, that even 
if the label $\p$ has been removed from the notation, 
fixing $\p$ means that we are fixing the divisor $D$ as well.

\begin{definition}
Given a morphism of varieties $p \colon X \to Y$, a \emph{morphism between two local snc models}
$\p \colon (X_{\p},D) \to X$ and $\s \colon (Y_\s,E) \to Y$
is a morphism $q \colon X_\p \to Y_\s$ such that $\s\o q = p \o \p$
and $\Supp(q^*E) \subset \Supp(D)$. 
We write $q \colon (X_\p,D) \to (Y_\s,E)$.
\end{definition}

\begin{definition}
Given two local snc models $\p \colon (X_{\p},D) \to X$ and $\p' \colon (X_{\p'},D') \to X$, 
we say that $(X_{\p'},D')$ \emph{dominates}
$(X_{\p},D)$, and write $(X_{\p'},D') \ge (X_{\p},D)$ or $\p' \ge \p$, if there is a morphism 
of local snc models $\a \colon (X_{\p'},D') \to (X_{\p},D)$, that is, a proper birational morphism
$\a \colon X_{\p'} \to X_{\p}$ such that $\p' = \p \o \a$ and $\Supp(\a^*D) \subset \Supp(D')$.  
We say that such a map $\a$ is a \emph{smooth transversal blow-up} if $\a$
is the blow-up of a smooth center contained in the smooth locus of $X_\p$ that is
transversal to $D$, and $D'$ is the sum of the proper transform of $D$ and the exceptional divisor. 
(We recall that a subvariety $W$ of a smooth variety $V$ is said to be \emph{transversal} to a simple
normal crossing divisor $E$ if in the formal neighborhood at any closed point,
$W$ and $E$ are locally defined by monomial equations in the same system of coordinates.)  
\end{definition}

If $(X_{\p},D)$ is a local snc model over $X$ and $D = \sum_{i=1}^r D_i$ is the decomposition into irreducible components, 
then for every $I \subset \{1,2,\dots,r\}$ we denote
\[
D_I := \bigcap_{i \in I} D_i 
\quad\text{and}\quad
D_I^\o := \bigcap_{i \in I} D_i \setminus \bigcup_{j \not\in I} D_j,
\]
with the convention that $D_\emptyset = X_{\p}$ and $D_\emptyset^\o = X_{\p} \setminus \Supp(D)$.
Note that, according to our definition, $D_I$ is irreducible.

Following \cite{Thu07,JM12}, 
every local snc model $\p \colon (X_{\p},D) \to X$ determines 
a set $\Sk_\p \subset X^{\val}$, called the
\emph{skeleton} of $X^{\val}$ associated to $(X_\p,D)$, 
whose elements are monomial valuations on $X_\p$ with respect to its toroidal structure
(this set is denoted $\QM(X_\p,D)$ in \cite{JM12}).
If $\p = \id_X$ (with empty divisor), then $\Sk_{\id_X}$
only consists of one point, the trivial valuation of $k(X)$. 
For any index set $I$, we denote by $\Sk_{\p,I} \subset \Sk_{\p}$
the set of quasi-monomial valuations centered at the generic point of $D_I$.
There are natural and compatible deformation retraction maps
$r_\p \colon X^{\val} \to \Sk_\p$
and $r_{\p'\p} \colon \Sk_{\p'} \to \Sk_\p$,
and $X^{\val} = \invlim_\p \Sk_\p$. 
For every $\p' \ge \p$ there is an inclusion $\Sk_\p \subset \Sk_{\p'}$
of subsets of $X^{\val}$. The elements of $X^\qm := \bigcup_\p \Sk_\p$
are by definition the \emph{quasi-monomial valuations} on $X$.

\begin{definition}
For every $I$, we call $\Sk_{\p,I}$ a \emph{face} of $\Sk_\p$.
A subset $R \subset \Sk_\p$ is said to be a \emph{potential face}
of $\Sk_\p$ if there is a local snc model $\p' \ge \p$ such that
$R$ is a face $\Sk_{\p',I'}$ of $\Sk_{\p'}$ under the 
natural inclusion $\Sk_\p \subset \Sk_{\p'}$. 
The \emph{dimension} of a face is its dimension over $\R$. 
The \emph{relative codimension} of a potential face $R \subset \Sk_\p$
is the codimension of $R$ inside the face $\Sk_{\p,I} \subset \Sk_\p$
containing $R$. 
\end{definition}

\begin{definition}
We say that $\p'$ is a \emph{refinement} of $\p$ if $\p' \ge \p$ and $\Sk_{\p'} = \Sk_\p$
as subsets of $X^{\val}$. 
\end{definition}

\begin{remark}
Every potential face of $\Sk_\p$ is contained in a face $\Sk_{\p,I}$, and 
any potential face of $\Sk_\p$ 
can be realized as a face of $\Sk_{\p'}$ for some refinement $\p' \ge \p$. 
Note also that $\p'$ is a refinement of $\p$ if and only if the induced map $\a \colon X_{\p'} \to X_\p$
is the composition of a sequence of blow-ups of strata. 
\end{remark}

Given any $\p' \ge \p$, we write $I' \succeq I$ whenever the generic point of a stratum $D'^\o_{I'}$ maps
to a point of a stratum $D_I^\o$. Note that this is equivalent to having $r_{\p'\p}(D'^\o_{I'}) \subset D^\o_I$, 
and also to the condition that the map 
$\Sk_{\p',I'} \to \Sk_\p$ induced by $r_{\p'\p}$ factors through $\Sk_{\p,I}$.

If $g \colon X' \to X$ is a proper birational morphism of varieties and $(X_{\p},D)$ is
a local snc model over both $X$ and $X'$, then $\Sk_{\p}$ can be viewed both as a subset of $X^{\val}$
and a subset of $(X')^{\val}$ via the natural identification given by
the induced map $g^{\val} \colon (X')^{\val} \to X^{\val}$.

The closure of $\Sk_{\p}$ in $X^\beth$ is denoted by $\ov\Sk_\p$ and is called
the \emph{skeleton} of $X^\beth$ associated to $(X_{\p},D)$. Note that,
differently from $\Sk_\p$, the closure $\ov\Sk_\p \subset X^\beth$ does depend on the model $X$.
It is proven in \cite{Thu07} that the retractions $r_\p$ extend to deformation retractions 
$\ov r_\p \colon X^\beth \to \ov\Sk_\p$ and $X^\beth = \invlim_\p \ov\Sk_\p$, 
but we will not use this.

\subsection{Mather log discrepancy function}
\label{ss:^A_X}

The condition that $X_{\p}$ dominates the Nash blow-up of $X$ 
means that the Jacobian ideal $\Jac_\p \subset \O_{X_\p}$ is locally principal.
If $X_{\p'}$ is any other local snc model
dominating $X_{\p}$ via a morphism $\a$, then $\Jac_{\p'} = \Jac_\p\.\, \Jac_\a$
(this formula relies on the fact that the exceptional locus of $\p$
is contained in the smooth locus of $X_\p$). 
For every divisor $E$ on $X_{\p}$, the \emph{Mather discrepancy} of $X$ along $E$ is defined 
to be the integer $\^k_E := \val_E(\Jac_\p)$. 
The following property is a variant of similar results from \cite{JM12,BdFFU15}.

\begin{proposition}
There is a unique lower-semicontinuous function $\^A_X \colon X^{\val} \to \R \cup \{\infty\}$
satisfying the following properties:
\begin{enumerate}
\item
\label{eq:^A(valE)}
$\^A_X(\val_E) = \^k_E(X)+1$ for every divisorial valuation $\val_E$ on $X$;
\item
$\^A_X$ is continuous on $\Sk_\p$ for all $\p$;
\item
$\^A_X = \sup_\p \^A_X \o \ov r_\p$. 
\end{enumerate}
This function is homogeneous with respect to rescaling of valuations
(i.e., $\^A_X(t\.v) = t\.\^A_X(v)$ for every $v \in X^{\val}$ and every $t > 0$), 
is finite on $\Sk_\p$ for every local snc model $\p$, and is positive away from the trivial valuation.
\end{proposition}

\begin{proof}
Let $\p_0$ be a local snc model, and let $\Jac_{\p_0} \subset \O_{X_\p}$
be the Jacobian ideal. Using the natural identification $X_{\p_0}^{\val} \simeq X^{\val}$, we define 
$\^A_X := A_{X_{\p_0}} + \ord(\Jac_{\p_0})$ where $A_{X_{\p_0}}$ is the \emph{log discrepancy function} 
defined in \cite{JM12}. 
For every divisorial valuation $\val_E$, if $E$ is a divisor on a local snc model $\p \ge \p_0$
and $\a \colon X_\p \to X_{\p_0}$ is the induced map, then $A_{X_{\p_0}}(\val_E) = \val_E(\Jac_\a) + 1$
by \cite[Proposition~5.1]{JM12}, hence
the equality $\Jac_\p = \Jac_{\p_0}\.\Jac_\a$ implies that
$\^A_X(\val_E) = \val_E(\Jac_\p) + 1$, which gives \eqref{eq:^A(valE)}. 
The function $\ord(\Jac_{\p_0})$ is homogeneous and 
continuous by the definition of the topology on $X_{\p_0}^{\val}$, 
and the proposition follows easily from the properties of the function $A_{X_{\p_0}}$ listed in 
\cite[Proposition~5.1 and Lemma~5.7]{JM12} and the fact that this function is homogeneous and strictly 
positive on divisorial valuations. 
The properties just established for the function $\^A_X$, in conjunction with
the property in \eqref{eq:^A(valE)}, imply that the definition of $\^A_X$
does not depend on the choice of model $\p_0$. 
\end{proof}

\begin{definition}
We call $\^A_X$ the \emph{Mather log discrepancy function}.
\end{definition}

\section{Motivic  measure}
\label{s:measure}

The purpose of this section is to define a motivic measure on 
$X^{\val}$ where $X$ is an algebraic variety 
over an algebraically closed field $k$ of characteristic zero.

We start by defining the motivic ring in which the measure will take values.

\subsection{Motivic ring}

Let $S$ be a scheme. 
We denote by $K_0(\Var_S)$ the Grothendieck ring generated by isomorphism classes $[V]_S$ of
separated schemes $V$ of finite type over $S$ modulo the relation $[V]_S = [V \setminus Z]_S + [Z]_S$
whenever $Z \subset V$ is a closed subscheme, with product defined by
$[V]_S\.[W]_S := [V \times_S W]_S$. In this ring, we have $0 = [\emptyset]_S$ and $1 = [S]_S$. 
We denote by $\L_S$ the class of $\A^1_S$.

In this paper, we work with the following version of the motivic ring:
\[
\MM_S := \big(K_0(\Var_S)/(\L_S - 1)\big)\otimes_\Z\R.
\]
The element of $\MM_S$ determined by a scheme $V$ over $S$ is still be denoted by $[V]_S$, 
or simply by $[V]$ if the base $S$ is clear from the context. 
When $S = \Spec k$, we just write $\MM_k$. 

Any morphism of varieties $p \colon S \to T$ induces a push-forward group homomorphism
$p_* \colon K_0(\Var_S) \to K_0(\Var_T)$ defined by mapping a class $[V]_S$ of a scheme $V$ over $S$
to the class $[V]_T$ of $V$ viewed as a scheme over $T$. Since 
$[\A^1_S]_T - [S]_T = ([\A^1_T]_T - [T]_T)\.[S]_T$ in $K_0(\Var_T)$, 
this induces a push-forward group homomorphism 
$p_* \colon \MM_S \to \MM_T$. The morphism $p$ also induces a 
pull-back group homomorphism $p^* \colon \MM_T \to \MM_S$ given by
$[V]_T \mapsto [V\times_T S]_S$.

\begin{remark}
The motivic ring used in the usual (geometric) motivic integration is different
from $\MM_S$ but, as we shall discuss later, the two are closely related. More precisely, 
$\MM_S$ is closely related to the ring $K_0(\Var_S)[[\P_S^a]^{-1}]_{a \ge 1}$ where usual 
motivic integrals can be shown to take values. Intuitively, 
by setting $\L_S = 1$ and allowing rational
coefficients, the denominators $[\P_S^a]$ appearing
in the above localization get encoded in $\MM_S$ as rational coefficients. 
\end{remark}

\subsection{Measure}
\label{ss:motivic-L-measure}

In \cref{ss:^A_X}, we introduced the Mather discrepancy function $\^A_X$ on $X^{\val}$. 
This function is piecewise linear on each $\Sk_\p$, in the following sense. 
Let $I = \{i_1,\dots,i_s\}$ be such that $D_I \ne \emptyset$.
Assuming $I \ne \emptyset$, once the indices $i_1,\dots,i_s$ are ordered, there is a canonical isomorphism
$\ff_{\p,I} \colon \Sk_{\p,I} \xrightarrow{\sim}  \R^s_{>0}$
given by $v \mapsto (v(D_{i_1}),\dots,v(D_{i_s}))$,
and $\^A_X$ is the linear function on $\Sk_{\p,I}$ 
determined by the conditions $\^A_X(\val_{D_i}) = \^a_{D_i}(X)$.
Under this isomorphism, the Lebesgue measure $\n$ on $\R^s_{>0}$
induces a measure $\ff_{\p,I}^*(\n)$ on $\Sk_{\p,I}$. After rescaling by $e^{-\^A_X}$, we define
the measure 
\[
\n_{\p,I} := e^{-\^A_X}\.\ff_{\p,I}^*(\n)
\]
on $\Sk_{\p,I}$. Note that $\Sk_{\p,I}$ is a $\sigma$-finite measurable space under this measure. 
If $x_1,\dots,x_s$ are the coordinates of $\R^s$, then we also define the form
\[
\om_{\p,I} := e^{-\^A_X}\.\ff_{\p,I}^*(dx_1 \wedge \dots \wedge dx_s)
\]
on $\Sk_{\p,I}$. We have $|\om_{\p,I}| = d\n_{\p,I}$. 
If $I = \emptyset$ then $\Sk_{\p,I}$ is just a point; 
in this case we extend the above notation using the convention that $\R_{>0}^0 = \{0\}$. 
We set $\Sk_{\p,I} = \emptyset$ if $D_I = \emptyset$.
The decomposition $\Sk_{\p} = \bigsqcup_I\Sk_{\p,I}$ determines a fan structure on $\Sk_{\p}$. 

Let now $R \subset \Sk_\p$ be any potential face. Pick a local snc model $\p' \ge \p$
such that $R = \Sk_{\p',I'}$ for some $I'$, and consider the form
$\om_R := \om_{\p',I'}$ and  the measure $\n_R := \n_{\p',I'}$. 
It is immediate to verify that these definitions 
are independent of the choice of model $\p'$.
The form $\om_R$ is only well-defined up to sign, 
which depends on the order of the indices in $I'$, and we have $|\om_R| = d\n_R$.
Stating the obvious, we stress that whenever $R \subset \Sk_\p$ is an actual face $\Sk_{\p,I}$, 
we can take $\om_R = \om_{\p,I}$, and we have $\n_R = \n_{\p,I}$.

Using these volume forms, we define a $\s$-algebra
of measurable sets on $\Sk_{\p,I}$, as follows.

\begin{definition}
A subset $S \subset \Sk_\p$ is \emph{measurable}
if for every potential face $R \subset \Sk_\p$
the set $S\cap R$ is measurable with respect to the measure $\n_R$. 
We denote by $\S(\Sk_\p)$ the collection of measurable subsets $S \subset \Sk_\p$.
Then we define
\[
\S_\p(X^{\val}) := \{ r_\p^{-1}(S) \subset X^{\val} \mid S \in \S(\Sk_\p) \}.
\]
Finally, we define the \emph{measure} $\m_\p$ on $\S_\p(X^{\val})$ with values in $\MM_X$ by setting
\[
\m_\p(r_\p^{-1}(S)) := 
\sum_I [D_I^\o]_X \, \n_{\p,I}(S\cap \Sk_{\p,I})
\]
for every $S \in \S(\Sk_\p)$.
\end{definition}

\begin{remark}
Imposing measurability conditions on all potential faces and not just on the faces of $\Sk_\p$
is natural once one realizes that any potential face is a face on $\Sk_{\p'}$
for some refinement $\p'$ of $\p$, and $\Sk_{\p'} = \Sk_\p$ as subsets of $X^{\val}$. 
\end{remark}

Setting $a_i := \^A_X(\val_{D_i})$, we have 
\[
\m_\p(r_\p^{-1}(S)) = 
\sum_I [D_I^\o]_X \int_{\ff_{\p,I}(S\cap \Sk_{\p,I})} e^{-\sum_{i\in I} a_ix_i} dx_1\dots dx_s.
\]

\begin{theorem}
\label{t:existence-measure}
Denoting $\S(X^{\val}) := \bigcup_\p \S_\p(X^{\val})$
where the union is taken over all local snc models $\p \colon (X_{\p},D) \to X$,
the measures $\m_\p$ glue together to give a function
\[
\m_X \colon \S(X^{\val}) \to \MM_X.
\]
\end{theorem}

We postpone the proof of this theorem to \cref{ss:proof}.

\begin{definition}
We call $\m_X$ the \emph{motivic  measure} on $X^{\val}$, and $\S(X^{\val})$
the \emph{algebra of measurable subsets} of $X^{\val}$. 
\end{definition}

\begin{remark}
$\S(X^{\val})$ is a Boleean algebra under the operations of sets, 
and $\m_X$ is additive on finite disjoint sums.
Note, however, that $\S(X^{\val})$ is not a $\s$-algebra. 
\end{remark}

\subsection{Extended motivic measure}
\label{ss:extended-measure}

It is natural to extend the measure $\m_X$ to a measure on $X^\beth$ by setting
the measure equal to the trivial measure on the complement $X^\beth \setminus X^{\val}$.
For all purposes, one can pretend that the theory of integration introduced in this paper
is in fact developed over $X^\beth$.

Setting the measure equal to zero on $X^\beth \setminus X^{\val}$
is consistent with the fact that 
for every $\p$ the function $\^A_X$ admits a continuous extension
from any skeleton $\Sk_\p$ to its closure $\ov\Sk_\p$ 
where the value on $\ov\Sk_\p \setminus \Sk_\p$ is set equal to 
$\infty$, hence $e^{-\^A_X}$ vanishes over there.
Extending the measure in this way, we have that
for every proper closed subscheme $Y \subsetneq X$, 
the set $Y^\beth \subset X^\beth$ is a measurable set of measure zero. 
This is in analogy with the fact that $Y_\infty \subset X_\infty$,  
being the intersection of a sequence of measurable sets whose measure converges to
zero in the motivic ring $\^\cM_X$, 
can be regarded as a set of measure zero from the point of view of motivic integration.

There is however a different way of extending the measure to $X^\beth$ 
so that it gives a nontrivial measure on the complement
$X^\beth \setminus X^{\val}$. This can be done by supplementing
the rescaling factor $e^{-\^A_X}$, which vanishes over $X^\beth \setminus X^{\val}$, 
with the functions $e^{-\^A_V}$ where $V$ ranges among the
subvarieties of $X$. The point is that, set theoretically, 
$X^\beth = \bigsqcup_{V \subset X} V^{\val}$ 
where the union is taken over all closed subvarieties $V \subset X$.

\begin{definition}
\label{d:extended-measure}
We define the \emph{extended motivic ring} to be 
$\~\MM_X := \prod_{V \subset X} \MM_V$
where the union is taken over all closed 
subvarieties $V \subset X$.
Let $\~\S(X^\beth)$ be the collection of all sets $T \subset X^\beth$
such that $T \cap V^{\val} \in \S(V^{\val})$ for every closed subvariety $V \subset X$. We define
the \emph{extended motivic  measure} $\~\m_X$ on $X^\beth$ by setting
$\~\m_X(T) := \big(\m_V(T \cap V^{\val})\big)_{V \subset X}$
for every $T \in \~\S(X)$. 
\end{definition}

Entries $\m_V(T \cap V^{\val})$, for $V \subsetneq X$, 
can be thought as the `infinitesimal' contributions to the measure, 
with the partial order given by inclusion yielding a partial order between 
these `infinitesimal' quantities.
There is no natural way, however, of extending 
nontrivially the measure from $X^\beth$ to the whole Berkovich space $X^\an$.

\subsection{Proof of \cref{t:existence-measure}}
\label{ss:proof}

We start with the following lemma. 

\begin{lemma}
\label{l:existence-measure-one-blowup}
Let $\a \colon (X_{\p'},D') \to (X_{\p},D)$ be a smooth transversal blow-up. 
Let $C \subset X_{\p}$ be the center of blow-up, and write $D' = \sum_{i=0}^rD_i'$
where $D_0'$ is the exceptional divisor of $\a$ and $D_i' = f^{-1}_*D_i$ for $i > 0$. 
Then for every $S \subset \Sk_\p$ we have $S \in \S(\Sk_\p)$ if and only if
$r_{\p'\p}^{-1}(S) \in \S(\Sk_{\p'})$, and the two measures $\m_\p$ and $\m_{\p'}$
agree on $r_\p^{-1}(S)$.
\end{lemma}

\begin{proof}
Throughout the proof, we abuse notation and write $\m_\p(S)$ instead of $\m_\p(r_\p^{-1}(S)).$
By additivity, it suffices to consider the case where $S \subset \Sk_{\p,I}$ for some $I\subset \{1,\dots,r\}$. 
After reindexing, we can assume that $I = \{1,\dots,s\}$ where $s = |I|$.
For every subset $J\subset I$, we associate the set $\tilde{J}:=\{0\}\cup J\subset \{0,1,\ldots, s\}$.
Finally, let $c = \codim(C,X_{\p})$.

Let $C_I^\o := C \cap D_I^\o$. 
If $C_I^\o = \emptyset$, then $r_{\p'\p}^{-1}(\Sk_{\p,I}) = \Sk_{\p',I}$, the map
$\Sk_{\p',I} \to \Sk_{\p,I}$ is an isomorphism preserving the measure given by the isomorphisms $\ff_{\p,I}$ and $\ff_{\p',I}$ and $\alpha$ identifies $D'^\o_I$ with $D^o_I$.
Thus, the result follows. 

Assume then that $C_I^\o \ne\emptyset$.
After reindexing $I$, we can assume that there is an integer $t \in \{0,\dots,s\}$,
such that $C \subset D_i$ if and only if $1 \le i \le t$.
Note that $\codim(C_I^\o,D_I^\o) = c - t$, 
and $t = 0$ if and only if $C \not\subset \Supp(D)$. 
Recall that we write $I' \succeq I$ if $\alpha(D'^\o_{I'})\subset D_I^\o$.
We have $I' \succeq I$ if and only if either $I' = I$ or
$I' =\~J$ where $J \subset I$ is a subset containing $\{t+1,\dots,s\}$. 

\begin{enumerate}
\item\label{stra1}
First, observe that $\alpha$ restricts to an isomorphism $D'^\o_I\cong D_I^\o \setminus C_I^\o$. Thus,
\[
[D'^\o_I] = [D_I^\o \setminus C_I^\o]
\]
in $K_0(\Var_X)/(\L_X - 1)$.
\end{enumerate}

Let $F$ be a fiber of $D_0' \to C$ over any closed point of $C_I^\o$. Note that
$F \cong \P^{c-1}_k$. 
The divisors $D'_i$, for $t<i\leq s$, contain $F$.
On the other hand, the divisor $\sum_{i=1}^t D'_i$ restricts to a simple normal crossing divisor $E$ on $F$, 
and writing $E = \sum_{i=1}^t E_i$ where $E_i = D_i'|_F$, 
the pair $(F,E)$ has a unique minimal stratum $E_I \cong \P^{c-1-t}_k$.
For every subset $J\subset I$ that contains $\{t+1,\dots,s\}$, the map $D'^\o_{\~J} \to C_I^\o$ is a locally trivial fibration with fiber $E_J^\o$.
For these $J$'s, we have:
\begin{enumerate}
\setcounter{enumi}{1}
\item\label{stra2}
If $|J| = s$ (i.e., $J = I$), then $E_J^\o\cong\P^{c-t-1}_k$, hence
\[
[D'^\o_{\~J}] = [\P^{c-t-1}_k\times_k C_I^\o]=[\P^{c-t-1}_X\times_X C_I^\o]=(c-t)[C_I^\o]
\]
in $K_0(\Var_X)/(\L_X - 1)$. In the last equality we used the fact that $\L_X=1$.
\item\label{stra3}
If $|J| = s-1$, then $E^\o_J\cong \A^{c-t}_k$, hence
\[
[D'^\o_{\~J}] = [\A^{c-t}_k\times_k C_I^\o]=[C_I^\o]
\]
in $K_0(\Var_X)/(\L_X - 1)$. 
\item\label{stra4}
If $|J| \le s-2$, then $E^\o_J\cong \mathbb{G}_{m,k}^{s-1-|J|}\times_k \A^{c-t}_k$, hence
\[
[D'^\o_{\~J}] = [\mathbb{G}_{m,k}^{s-1-|J|}\times_k \A^{c-t}_k\times_k C_I^\o]=0
\]
in $K_0(\Var_X)/(\L_X - 1)$.
\end{enumerate}

Denoting $S' = r_{\p'\p}^{-1}(S)$
and setting $S'_{I'} := S' \cap \Sk_{\p',I'}$ for every $I'$, we have the decomposition
$S' = \bigsqcup_{I' \succeq I} S'_{I'}.$
The computation below will show that $S$ is a measurable subset
of $\Sk_{\p,I}$ if and only if $S'_{I'}$ is a measurable subset of $\Sk_{\p',I'}$ for all $I' \succeq I$. 

Recall that, by definition, 
$\m_{\p'}(S'_{I'}) = [D'^\o_{I'}] \,\n_{\p,I}(S'_{I'})$ for every $I'$. 
It follows from \eqref{stra4} that $\m_{\p'}(S'_{\~J}) = 0$ for all $J \subsetneq I$ with $|J| \le s-2$, 
hence we have
\[
\m_{\p'}(S') = \m_{\p'}(S_I') + \m_{\p'}(S_{\~I}') + \sum_{j=1}^t \m_{\p'}(S_{\~I \setminus \{j\}}').
\]
We now analyze the terms in the right-hand-side of this equation.
Regarding the first term, we have $\ff_{\p',I}(S'_I) = \ff_{\p,I}(S)$ in $\R^s_{>0}$, 
and hence
\[
\m_{\p'}(S_I') = [D_I'^\o] \int_{S'_I}  d\n_{\p',I} = [D_I^\o \setminus C_I^\o] \int_S d\n_{\p,I}
= [D_I^\o \setminus C_I^\o] \,\m_{\p}(S_I).
\]

The analysis of the other terms require a closer look at the integrals and 
the linear maps
$\lambda_{\~J} \colon \R^{|\~J|}_{>0} \to \R^{|I|}_{>0}$
defined by the retraction $r_{\p'\p}$ via the isomorphisms
$\ff_{\p,I}\colon\Sk_{\p,I} \to \R_{>0}^{|I|}$ and
$\ff_{\p',\~J} \colon \Sk_{\p',\~J} \to \R_{>0}^{|\~J|}$ 
(that is, $\lambda_{\~J}:=\ff_{\p,I}\circ r_{\p'\p}\circ \ff_{\p',\~J}^{-1}$).
First observe that $\alpha^*D_i = D_i'+D_0'$ for $1\leq i\leq t$ and 
$\alpha^*D_i = D_i'$ for $t+1< i\leq s$. 
Taking $J = I$, we see from these equations that the map 
$\lambda_{\~I}\colon \R^{s+1}_{>0}\to \R^s_{>0}$ is given by
\begin{equation}
\label{eq:lambda}
\begin{cases}
y_i=x_i+x_0 & \text{for $1\leq i\leq t$},  \\
y_i=x_i & \text{for $t+1<i\leq s$}. 
\end{cases}
\end{equation}
Here the coordinates $y_i$ and $x_j$ are the coordinates of the target and domain respectively, indexed by the corresponding indexing sets, $I$ and $\~I$.
When $J=I\setminus\{j\}$ for $1 \le j \le t$, then the corresponding map
$\lambda_{\~I\setminus\{j\}}\colon \R^s_{>0}\to \R^s_{>0}$ is 
given by the same equations in \eqref{eq:lambda} except for $y_j$ which is set to $0$
(note that the variable $x_j$ is now omitted).

For every $0 \le i \le s$, we set $a_i := \^A_X(\val_{D'_i})$.
It follows from the equations for $\a^*D_i$ 
and the fact that $\val_{D'_0}(\Jac_{\alpha})=c-1$ and $\val_{D'_i}(\Jac_{\alpha})=0$ for $1\leq i\leq s$
that $a_i = \^A_X(\val_{D_i})$ for all $1 \le i \le s $ and $a_0 = c - t + \sum_{j=1}^t a_j$.

\begin{case}
\label{case:t=0}
Assume that $C \not\subset \Supp(D)$, i.e., $t = 0$. 
\end{case}

We start with the case $J = I$. 
We have $\^A_X(x_0,\dots,x_s) = \sum_{i=0}^s a_ix_i$ and
$\^A_X(y_1,\dots,y_s) = \sum_{i=1}^s a_iy_i$ under the isomorphisms $\ff_{\p',\~I}$
and $\ff_{\p,I}$, and hence 
\begin{align*}
\m_{\p'}(S_{\~I}') 
&= [D'^\o_{\~I}] \int_{\ff_{\p',\~I}(S_{\~I}')} e^{- \sum_{i=0}^s a_ix_i} \,dx_0\,dx_1\,\dots\,dx_s \\
&= c\,[C_I^\o] \int_{\ff_{\p,I}(S)} \Big( \int_0^\infty e^{- cx_0 - \sum_{i=1}^s a_iy_i} \,dx_0 \Big) dy_1\,\dots\,dy_s \\
&= [C_I^\o] \int_{\ff_{\p,I}(S)} e^{- \sum_{i=1}^s a_iy_i} \,dy_1\,\dots\,dy_s \\
&= [C_I^\o] \, \n_{\p,I}(S).
\end{align*}
Combining with the computation done for the first term, this gives the required identity
\[
\m_{\p'}(S') = ([D_I^\o \setminus C_I^\o] + [C_I^\o]) \int_S  d\n_{\p,I} 
= \m_\p(S).
\]

\begin{case}\label{case:t-neq-0}
Assume that $C \subset \Supp(D)$, i.e., $t \ne 0$. 
\end{case}

We start by looking at the term with $J = I$. 
For every $i = 1,\dots,t$, we let
\[
\ff_{\p,I}(S)_i := \ff_{\p,I}(S) \cap \{(y_1,\dots,y_s) \mid y_i = \min\{y_1,\dots,y_t\} \}.
\]
Note that $\ff_{\p,I}(S) = \bigcup_{i=1}^t\ff_{\p,I}(S)_i$ 
and the union is a disjoint union away from a set of measure zero.
By \eqref{eq:lambda}, we may apply the change of variables $x_i = y_i - x_0$ for $1\leq i\leq t$
and $x_i = y_i$ for $t+1\leq i\leq s$, and compute
\begin{align*}
\m_{\p'}(S_{\~I}') 
&= [D'^\o_{\~I}] \int_{\ff_{\p',\~I}(S_{\~I}')} 
e^{- \sum_{i=0}^s a_ix_i} \,dx_0\,dx_1\,\dots\,dx_s \\
&= (c-t)[C_I^\o] \int_{\ff_{\p',\~I}(S_{\~I}')}  
e^{- (c-t)x_0 - \sum_{i=1}^s a_iy_i} \,dx_0\,dy_1\,\dots\,dy_s \\
&= (c-t)[C_I^\o] \sum_{j=1}^t \int_{\ff_{\p,I}(S)_j} 
\Big( \int_0^{y_j} e^{- (c-t)x_0 - \sum_{i=1}^s a_iy_i} \,dx_0 \Big) dy_1\,\dots\,dy_s \\
&= [C_I^\o] \Big(\int_{\ff_{\p,I}(S)} e^{ - \sum_{i=1}^s a_iy_i}dy_1\,\dots\,dy_s  
- \sum_{j=1}^t \int_{\ff_{\p,I}(S)_j} e^{- (c-t)y_j - \sum_{i=1}^s a_iy_i} \, dy_1\,\dots\,dy_s \Big) \\
&= [C_I^\o] \Big( \n_{\p,I}(S)
- \sum_{j=1}^t \int_{\ff_{\p,I}(S)_j} e^{- (c-t)y_j - \sum_{i=1}^s a_iy_i} \, dy_1\,\dots\,dy_s \Big). 
\end{align*}
We look now at the terms where $J = I \setminus \{j\}$ for some $1 \le j \le t$.
Note that, modulo a set of measure zero,
$\lambda_{\~I\setminus \{j\}}$ sends $\ff_{\p',\~I\setminus\{j\}}(S_{\~I\setminus\{j\}}')$ to $\ff_{\p,I}(S)_j$ bijectively.
Using the change of variable $x_0=y_j$, $x_i=y_i-y_j$ for $1\leq i\leq t$, $i \neq j$, 
and $x_i=y_i$ for $t<i\leq s$ as given in \eqref{eq:lambda}, we compute
\begin{align*}
\m_{\p'}(S_{\~I \setminus \{j\}}')
&=  [D'^\o_{\~I \setminus \{j\}}]\int_{\ff_{\p',\~I\setminus\{j\}}(S_{\~I\setminus\{j\}}')}  e^{- a_0x_0 - \sum_{i=1}^{j-1} a_ix_i-\sum_{i=j+1}^{s} a_ix_i} \, dx_0dx_1\,\dots \widehat{dx_j}\dots\,dx_s\\
&=[C_I^\o] \int_{\ff_{\p,I}(S)_j} e^{- (c-t)y_j - \sum_{i=1}^s a_iy_i} \, dy_1\,\dots\,dy_s.
\end{align*}
This completes the proof of the lemma.
\end{proof}

\begin{proof}[Proof of \cref{t:existence-measure}]
We need to show that, given two local snc models 
$\p  \colon (X_{\p},D) \to X$ and $\p' \colon (X_{\p'},D') \to X$, 
the measures $\m_\p$ and $\m_{\p'}$ agree on $\S_\p(X^{\val}) \cap \S_{\p'}(X^{\val})$. 
Let $T \in \S_\p(X^{\val}) \cap \S_{\p'}(X^{\val})$ be any element, and write 
$T = r_\p^{-1}(S) = r_{\p'}^{-1}(S')$
where $S \in \S(\Sk_\p)$ and $S' \in \S(\Sk_{\p'})$. 

The singular loci of $X_\p$ and $X_{\p'}$ are disjoint from the 
supports of the divisors $D$ and $D'$, and we can resolve singularities
without touching the divisors, and hence without changing the sets $S$ and $S'$. 
Therefore, we can assume without loss of generality that 
both models are smooth.

The composition $\p^{-1} \o \p' \colon X_{\p'} \rat X_{\p}$ is a birational map. 
Let $U \subset X_{\p} \setminus \Supp(D)$ and $U' \subset X_{\p'} \setminus \Supp(D')$ 
be nonempty open sets such that this map restricts to an isomorphism $U' \cong U$. 
Using existence of resolutions of marked ideals \cite{Wlo05}, 
where we mark the defining ideals of the complements $X_{\p} \setminus U$ and $X_{\p'} \setminus U'$, 
we can replace the models $(X_{\p},D)$ and $(X_{\p'},D')$ with higher models
so that $\Supp(D) = X_{\p} \setminus U$ and $\Supp(D') = X_{\p'} \setminus U'$. 
This reduction step is done by only taking of smooth transversal blow-ups.
By applying again resolution of singularities, this time to the indeterminacy subscheme
of $\p^{-1} \o \p'$, we can furthermore assume that $\p^{-1} \o \p'$
is a morphism, so that $(X_{\p'},D') \ge (X_{\p},D)$. 
After we replace $S$ and $S'$ with the respective inverse images under the induced
retraction maps corresponding to these sequences of blow-ups, we still have 
$T = r_\p^{-1}(S) = r_{\p'}^{-1}(S')$, 
and \cref{l:existence-measure-one-blowup} ensures that the corresponding measures 
$\m_\p(T)$ and $\m_{\p'}(T)$ have not changed. 

The next step is to apply the weak factorization theorem \cite{AKMW02,Wlo03} to the birational
morphism $\p^{-1} \o \p' \colon X_{\p'} \to X_{\p}$, keeping track of the open set $U$.
Using the formulation given in \cite[Theorem~0.3.1]{AKMW02}, we
obtain a decomposition of $\p^{-1} \o \p'$ into a chain of birational maps
\[
X_{\p'} = X_{\p_0} \xdashrightarrow{\text{$\ff_1$}} X_{\p_1} \xdashrightarrow{\text{$\ff_2$}} \dots
\xdashrightarrow{\text{$\ff_{n-1}$}} X_{\p_{n-1}} \xdashrightarrow{\text{$\ff_n$}} X_{\p_n} = X_{\p}
\]
where either $\ff_i$ or $\ff_i^{-1}$ is a blow-up of a smooth center that is disjoint with $U$
and has transversal intersections with its complement, and all models $X_{\p_i}$ dominate $X_\p$ via a well-defined morphism $X_{\p_i} \to X_{\p}$. 
Denoting $S_i = r_{\p_i,\p}^{-1}(S)$, we have 
$S_{i-1} = r_{\p_{i-1},\p_i}^{-1}(S_i)$ if $\ff_i$ is a morphism, 
and $S_i = r_{\p_{i},\p_{i-1}}^{-1}(S_{i-1})$ if $\ff_i^{-1}$ is a morphism. 
Going through these blow-ups and blow-downs, and applying \cref{l:existence-measure-one-blowup}
at each step, we conclude that $\m_{\p}(T) = \m_{\p'}(T)$. 
\end{proof}

\section{Motivic integration}
\label{s:functions}

As before, let $X$ be a variety over an algebraically closed field $k$ of characteristic zero.
In this section, we define motivic integration on $X^{\val}$.

\subsection{Motivic functions}
\label{ss:MC-functions}

Given a local snc model $\p \colon (X_{\p},D) \to X$, 
we denote by $\Func(\Sk_\p,\R)$ the ring of real valued functions on $\Sk_\p$, 
and let $\Func_0(\Sk_\p,\R) \subset \Func(\Sk_\p,\R)$ be the ideal
of functions whose restriction to every potential face $R$ of $\Sk_\p$
vanishes almost-everywhere with respect to the corresponding measure $\n_R$.

\begin{definition}
We call 
\[
\FF_\p := \Func(\Sk_\p,\R)/\Func_0(\Sk_\p,\R)
\]
the \emph{ring of functions of level $\p$}. 
We think of elements of $\FF_\p$ as almost-everywhere defined functions
(in the sense explained above), 
and simply refer to them as {functions}. 
We let $\FF_\p^\o \subset \FF_\p$ be the subring generated by 
the characteristic functions $1_{\Sk_{\p,I}}$ of the faces $\Sk_{\p,I}$ of $\Sk_{\p}$.
\end{definition}

Consider the motivic ring
$\MM_{X_\p} = \big(K_0(\Var_{X_\p})/(\L_{X_\p} - 1)\big)\otimes_\Z\R$.
For short, we will denote the element in this ring defined by a scheme $V$ over $X_\p$ by
the symbol $[V]_{\p}$, instead of $[V]_{X_\p}$; if $\p$ is the identity function $X \to X$ then
we stick with the notation $[V]_X$.

For every model $\p \colon (X_{\p},D) \to X$, there is an injective ring homomorphism
$\FF_\p^\o \inj \MM(X_{\p})$ defined by mapping $1_{\Sk_{\p,I}} \mapsto [D_I^\o]_{\p}$. 
We regard $\MM(X_\p)$ as an $\FF_\p^\o$-module via this map. 

\begin{definition}
We call 
\[
\MF_\p := \MM(X_{\p}) \otimes_{\FF_\p^\o} \FF_\p.
\]
the \emph{ring of motivic functions of level $\p$}.
\end{definition}

Let $\p \colon (X_{\p'},D') \to X$ be another local snc model dominating $(X_{\p},D)$ via a 
morphism $\a \colon X_{\p'} \to X_{\p}$. Recall that we have
retraction maps
$\ov r_{\p'\p} \colon \Sk_{\p'} \to \Sk_\p$ and $r_{\p'\p} \colon \Sk_{\p'} \to \Sk_{\p}$.
Composition with $r_{\p'\p}$ gives rise to an injective map $\FF_\p \inj \FF_{\p'}$, 
and the latter restricts to an injection $\FF_\p^\o \inj \FF_{\p'}^\o$ since
the inverse image of a face of $\Sk_\p$ is a union of faces of $\Sk_{\p'}$.

\begin{definition}
Set $\FF^\o(X^{\val}) := \dirlim_\p \FF_\p^\o$ and
$\FF(X^{\val}) := \dirlim_\p \FF_\p$. 
We think of elements of these rings as almost-everywhere defined functions on $X^{\val}$
and simply refer to them as \emph{functions}. 
\end{definition}

We have an inclusion of rings $\FF^\o(X^{\val}) \inj \FF(X^{\val})$.
For every $\p$, there is a natural injection $\FF_\p \inj \FF(X^{\val})$ whose image
consists of functions of the form $\f = \f_\p \o r_\p$ where $\f_\p \in \FF_\p$. 

\begin{example}
Given a nonzero ideal sheaf $\fa \subset \O_X$ and a real number $s \in \R$,
we consider the function $X^{\val} \to \R$ given by $x \mapsto e^{-s\,v_x(\fa)}$ 
(or, equivalently, $x \mapsto |\fa|_x^s$). We denote this function by $e^{-s\ord(\fa)}$
or $|\fa|^s$. 
We have $e^{-s\ord(\fa)} = e^{-s\ord(\fa)} \o r_\p$ for any local snc model $\p \colon (X_{\p},D) \to X$
such that $\fa\.\O_Y$ is a locally principal ideal sheaf cosupported on $D$, 
thus the function defines an element in $\FF_\p$ and hence in $\FF(X^{\val})$.
If $s > 0$, then the definition extends to the case $\fa = (0)$, giving a
function that is identically zero. 
\end{example}

At the level of motivic rings, for any morphism $\a \colon (X_{\p'},D') \to (X_{\p},D)$,
of local snc models over $X$ we have a pull-back ring homomorphism
$\a^* \colon \MM_{X_\p} \to \MM_{X_{\p'}}$ mapping $[V]_{\p} \mapsto [V \times_{X_{\p}} X_{\p'}]_{\p'}$
for every $X_{\p}$-scheme $V$. Note that
$\a^*([D_I^\o]_{\p}) = \sum_{I' \succeq I} [D'^\o_{I'}]_{\p'}$, 
which is compatible with the formula
$1_{\Sk_{\p,I}} \o r_{\p'\p} = \sum_{I' \succeq I} 1_{\Sk_{\p',I'}}$
It follows that the diagram
\[
\xymatrix{
\MM(X_{\p}) \ar[d]_{\a^*} & \FF_\p^\o \ar[l] \ar[d] \ar[r] & \FF_\p \ar[d] \\
\MM(X_{\p'}) & \FF_{\p'}^\o \ar[l] \ar[r] & \FF_{\p'}
}
\]
commutes.
This yields a natural pull-back ring homomorphism
\[
\a^* \colon \MF_\p \to \MF_{\p'}.
\]

\begin{definition}
The \emph{ring of motivic functions}
is the ring
\[
\MF(X^{\val}) := \limdir_\p \MF_\p,
\]
where the direct limit is taken over all (local) snc models $\p \colon (X_{\p},D) \to X$
using the pull-back map defined above.  
An element of $\MF(X^{\val})$ is called a \emph{motivic function}.
A motivic function $f$ is said to be \emph{determined at level $\p$}
if it is in the image of the natural map $\MF_\p \to \MF(X^{\val})$. 
Any element $f_\p \in \MF_\p$ mapping to $f$ is called a \emph{representative} of $f$.
\end{definition}

\begin{remark}
\label{r:bourbaki}
Since direct limit commutes with tensor product \cite[Proposition~7 at page~290]{Bou89}, there is a natural isomorphism
\[
\MF(X^{\val}) \cong  \big(\limdir_\p \MM_{X_\p} \big) \otimes_{\FF^\o(X^{\val})} \FF(X^{\val}).
\]
Note also that if $X = \Spec k$ then $\MF(X^{\val}) = \MM_k$. 
\end{remark}

\begin{remark}
If $f_\p = [V]_{\p} \otimes \f_\p$ where $V$ is a scheme over $X_{\p}$ and $\f_\p \in \FF_\p$, then 
$f_{\p'} = [V\times_YX_{\p'}]_{\p'} \otimes (\f_\p \o r_{\p'\p})$ for every $\p' \ge \p$. 
If $V$ is a scheme over $X_{\p}$ whose image in $X_{\p}$ is supported on a stratum $D^\o_I$ of $D$, 
then for every $\f_\p \in \FF_\p$ we have
$[V]_{\p} \otimes \f_\p 
= ([V]_{\p}\.[D_I^\o]_\p) \otimes \f_\p = [V]_{\p} \otimes (\f_\p\. 1_{\Sk_{\p,I}})$ in $\MF_\p$. 
Similarly, if $V$ is any scheme over $X_{\p}$ and $\f_\p$ is
supported on $\Sk_{\p,I}$, then $[V]_{\p} \otimes \f_\p = [V \times_{X_\p}D_I^\o]_{\p} \otimes \f_\p$. 
\end{remark}

\subsection{Integrability}
\label{ss:integrable-functions}

We start by defining integrability for functions of finite levels. 

\begin{definition}
Let $\p \colon (X_\p,D) \to X$ be a local snc model.
A function $\f \in \FF_\p$ is \emph{measurable} if for every potential face $R \subset \Sk_\p$, 
the restriction of $\f$ to $R$ is measurable with respect to the measure $\n_R$. 
A measurable function $\f \in \FF_\p$ is \emph{integrable} if furthermore
$\int_R |\f| \,d\n_R < \infty$ for all potential faces $R \subset \Sk_\p$.
We denote by $\IF_\p \subset \FF_\p$ the subspace of integrable functions. 
\end{definition}

\begin{example}
Any continuous function on $\Sk_\p$ or, more generally, 
any function that is continuous on each face $\Sk_{\p,I}$
of $\Sk_\p$ is integrable. 
\end{example}

Given local snc models $\p' \ge \p$, we have $\IF_\p \inj \IF_{\p'}$
via the map $\FF_\p \inj \FF_{\p'}$.
This follows by a similar computation as 
in the proof of \cref{t:existence-measure}, by
adding at each step of the computation the function $\f$ as a factor in the integrals.
We can therefore define
\[
\IF(X^{\val}) := \dirlim_{\p} \IF_\p,
\]
We denote by $\IF_\p(X^{\val})$ the image of $\IF_\p \inj \IF(X^{\val})$.

\begin{definition}
A motivic function $f \in \MF(X^{\val})$ is \emph{integrable} if it belongs to 
\[
\IMF(X^{\val}) := \dirlim_\p \big( \MM_{X_\p} \otimes_{\FF_\p^\o} \IF_\p \big).
\]
We call $\IMF(X^{\val})$ the module of \emph{integrable motivic functions}
on $X^{\val}$. The module of \emph{measurable motivic functions}
on $X^{\val}$ can be defined in a similar fashion.
\end{definition}

\begin{remark}
\label{r:f-integrable-f_pi}
There is a natural isomorphism
\[
\IMF(X^{\val}) \cong \big(\dirlim_\p \MM_{X_\p}\big) \otimes_{\FF^\o(X^{\val})} \IF(X^{\val}).
\]
This shows that $\IMF(X^{\val})$ has a natural structure of module over $\FF^\o(X^{\val})$. 
\end{remark}

\begin{definition}
If $f \in\IMF(X^{\val})$ is represented by an element $f_\p = \sum_j [V_j]_\p \otimes \f_j$ 
where $[V_j] \in \MM_{X_\p}$ and $\f_j \in \IF_\p$, 
then the \emph{integral} of $f$ over $X^{\val}$ is the element in $\MM_X$ given by
\[
\int_{X^{\val}} f\,d\m_X := 
\sum_j \sum_I [V_j\times_{X_\p} D_I^\o]_X \int_{\Sk_{\p,I}} \f_j\, d\n_{\p,I}.
\]
\end{definition}

\begin{theorem}
\label{t:integral-well-defined}
The integral is well-defined.
\end{theorem}

\begin{proof}
The definition is clearly independent from the way the element $f_\p$ 
is written in the tensor product.
The independence from the choice of the model $\p$
can be verified following the argument in the proof of \cref{t:existence-measure}
by reducing to check that the definition of the integral is 
stable when we take a smooth transversal blow-up $\a \colon X_{\p'} \to X_\p$ of snc models.

This follows by similar computations as those carried out in \cref{l:existence-measure-one-blowup},
adding at each step the factors $V_j$ in the motivic coefficients and the 
functions $\f_j$ as factors in the integrals.
We use the same notation as in that \lcnamecref{l:existence-measure-one-blowup}.
By linearity, it suffices to consider the case where $f$ is represented by $f_\p =  [V]_\p \otimes \f_\p$
with $\f_\p$ supported on a face $\Sk_{\p,I}$.
Note that $f$ is also represented by 
$f_{\p'} = \sum_{I' \succeq I}  [V\times_{X_\p} X_{\p'}]_{\p'} \otimes \f_{\p'}$
where $\f_{\p'} := \f_\p \o r_{\p'\p}$. 
We need to show that
\begin{equation}\label{eq:invariance-integral}
[V\times_{X_\p} D_I^\o]_X \int_{\Sk_{\p,I}} \f_\p \, d\n_{\p,I}
= \sum_{I' \succeq I} [V\times_{X_\p} D'^\o_{I'}]_X \int_{\Sk_{\p',I'}} \f_{\p'} \, d\n_{\p',I'}.
\end{equation}
Analyzing the strata $D'^\o_{I'}$ as in the proof of \cref{l:existence-measure-one-blowup}, 
the only terms that matter are $I'$ is either $I$, $\~I$, or $\~I\setminus \{j\}$ for some $1\leq j\leq t$.
We shall analyze in detail the case where $I' = \~I$, which was 
discussed in \cref{case:t-neq-0} in the proof of \cref{l:existence-measure-one-blowup}.
Following the same argument used there, we compute
\begin{align*}
&[V\times_{X_\p} D'^\o_{\~I}]_X \int_{\Sk_{\p',\~I}} \f_{\p'}\, d\n_{\p',I'} = \\
&=(c-t)[V\times_{X_\p}C^\o_I]_X \int_{\Sk_{\p',\~I}} \f_{\p'}(x_0,\ldots,x_s)\, e^{- \sum_{i=0}^s a_ix_i} \,dx_0\,dx_1\,\dots\,dx_s\\
&=(c-t)[V\times_{X_\p}C^\o_I]_X\sum_{j=1}^t \int_{(\Sk_{\p,I})_j} \f_\p(y_1,\cdots y_s)\, \Big(\int_0^{y_j}e^{-(c-t)x_0- \sum_{i=1}^s a_iy_i} \,dx_0\Big)\,dy_1\,\dots\,dy_s\\
&=[V\times_{X_\p}C^\o_I]_X\Big( \int_{\Sk_{\p,I}}\f_\p\, d\n_{\p,I}
-\sum_{j=1}^{t}\int_{(\Sk_{\p,I})_j}\f_\p(y_1,\dots,y_s)\, e^{-(c-t)y_j- \sum_{i=1}^s a_iy_i}\,dy_1\,\dots\,dy_s
\Big).\\
\end{align*}
Similar computations show that the term given by $I' = I$ is equal to
\[
[V\times_{X_\p} (D^\o_{I}\setminus C^\o_I)]_X \int_{\Sk_{\p,I}} \f_\p\, d\n_{\p,I},
\]
and the terms where $I' = \~I\setminus\{j\}$ for $1\leq j\leq t$ are given by
\[
[V\times_{X_\p} C^\o_I]_X \int_{(\Sk_{\p,I})_j}\f_\p(y_1,\dots,y_s)\, e^{-(c-t)y_j- \sum_{i=1}^s a_iy_i}\,dy_1\,\dots\,dy_s.
\]
Adding up all the terms gives us the left hand side of (\ref{eq:invariance-integral}). This completes the proof.
\end{proof}

\begin{example}
Let $X$ be a variety, $\fa \subset \O_X$ an ideal sheaf, and $s \in \R$. 
Let $\p \colon (X_\p,D) \to X$ be a snc model such that $\fa\.\O_{X_\p} = \O_{X_\p}(-\sum b_i D_i)$,
and let $a_i := \^A_X(\val_{D_i})$. Assume that $a_i + sb_i > 0$ for every $i$
(this is automatic, for instance, if we take $s \ge 0$). 
Then 
\[
\int_{X^{\val}} |\fa|^s d\m_X
= \sum_I [D_I^\o]\int_{\R_{> 0}^{|I|}} e^{-\sum_{i \in I}(a_i+sb_i)x_i} dx_1\cdots dx_{|I|} 
= \sum_I \frac{[D_I^\o]}{\prod_{i \in I} (a_i+sb_i)}.
\]
\end{example}

\begin{example}
\label{eg:stringy-class}
Suppose that $X$ is a normal variety
such that the canonical class $K_X$ is $\Q$-Cartier, 
and assume that $X$ has log terminal singularities.
Fix a positive integer $r$ such that $rK_X$ is Cartier, 
and let $\fn_{r,X} \subset \O_X$ be the ideal sheaf defined by 
the image of the natural map 
$(\Om_X^{\dim X})^{\otimes r} \otimes \O_X(-rK_X) \to \O_X$.
One can check that $\^A_X = A_X + \frac 1r\ord(\fn_{r,X})$ where 
$A_X$ is the \emph{log discrepancy function}, which is defined in this
generality in \cite{BdFFU15} (this function takes
value $A_X(\val_E) = \val_E(K_{X_{\p}/X}) + 1$ for every prime divisor $E$ on a
resolution of singularities $X_{\p} \to X$). 
If $\p \colon (X_{\p},D) \to X$ is any local snc model such that $\fn_{r,X}\.\O_Y$
is a locally principal ideal cosupported in $\Supp(D)$, 
and we set $a_i := A_X(\val_{D_i})$, then 
\[
\int_{X^{\val}} |\fn_{r,X}|^{\frac 1r} d\m_X 
= \sum_I [D_I^\o] \int_{\R_{>0}^{|I|}} e^{-\sum_{i \in I} a_ix_i} dx_1\dots dx_{|I|}
= \sum_I \frac{[D_I^\o]}{\prod_{i \in I} a_i}.
\]
This is the \emph{stringy motivic class} of $X$ (cf.\ \cite[Section~4]{Rei02}).
For example, if $X = M/G$ where $M$ is a manifold and $G$ is a finite group, 
then this is equal to $\sum_{[g]} [M^g/C(g)]$,
where the sum runs over conjugacy classes and $C(g) \subset G$ denotes the centralizer of
an element $g \in G$. 
\end{example}

\begin{example}
If $X$ is a smooth variety and $Z$ is a proper closed subscheme of $X$ (e.g., a Cartier divisor), then 
the \emph{log canonical threshold} $\lct(X,Z)$ of the pair $(X,Z)$ is the
supremum of the numbers $t \ge 0$ such that the pair $(X,tZ)$ is log canonical; 
equivalently, $\lct(X,Z)$ is the minimum of the numbers $A_X(\val_E)/\val_E(\I_Z)$.
If the ideal sheaf $\I_Z \subset \O_X$ of $Z$ is locally generated by a single element $h$, 
then the log canonical threshold can be equivalently defined as the supremum of the numbers 
$t\ge 0$ such that $|h|^{-2t}$ is locally integrable on $X$.
It follows by the above computations that $\lct(X,Z)$ is also the supremum value of $t\ge 0$
such that $|\I_Z|^t$ is integrable on $X^{\val}$.
In the singular case, if $X$ and $\fn_{r,X}$ are as in \cref{eg:stringy-class} then 
the log canonical threshold $\lct(X,Z)$ is still defined, and 
is equal to the supremum value of $t\ge 0$ such that $|\fn_{r,X}|^{1/r}|\I_Z|^t$ is integrable on $X^{\val}$.
In general, for an arbitrary variety $X$, one defines the \emph{Mather log canonical threshold}
$\^\lct(X,Z)$ to be the 
supremum of the numbers $t \ge 0$ such that the pair $(X,tZ)$ is Mather log canonical,  
and this condition is equivalent to $|\I_Z|^t$ being integrable on $X^{\val}$.
\end{example}

The motivic integral defined here using Berkovich spaces
is of course closely related to the usual
(geometric) motivic integral defined using arc spaces \cite{Kon95,DL99}. 
For simplicity, assume that $X$ is a smooth variety and $B$ is an effective integral divisor.
Let $X_\infty$ denote the space of formal arcs on $X$.
For the purpose of this discussion, we will denote by $\m^{X_\infty}$ the motivic measure on $X_\infty$. 
By definition the classical motivic integral
$\int_{X_\infty} \L_X^{-\ord(B)} d\m^{X_\infty}$
takes value in a suitable completion of the localization $K_0(\Var_X)[\L_X^{-1}]$.
Using resolution of singularities and the 
change-of-variables formula, 
one shows that the integral can be represented in a natural way by an element in
$K_0(\Var_X)[[\P^a_X]^{-1}]_{a \ge 1}$.
As $[\P_X^a] = a+1$ modulo $\L_X-1$, there is a natural map 
$\Phi \colon K_0(\Var_X)[[\P^a_X]^{-1}]_{a \ge 1} \to (K_0(\Var_X)/(\L_X-1)) \otimes_\Z\Q$, 
and 
\[
\int_{X^{\val}} e^{-\ord(B)} d\m_X = \Psi\Big(\int_{X_\infty} \L_X^{-\ord(B)} d\m^{X_\infty}\Big).
\]
Therefore the motivic integral defined in this paper using Berkovich spaces
recovers, modulo $\L_X-1$, the classical one defined using arc spaces.
The requirement that $B$ be an integral divisor is not essential, and the classical 
definition of motivic integral can be easily extended to deal with $\Q$-divisors. 
This, however, requires to enlarge the motivic ring by introducing 
a symbolic root $\L_X^{1/r}$ of $\L_X$. This step is not needed if the 
integral is defined using Berkovich spaces.  

\begin{remark}
It is possible to extend the theory developed in this section 
using the extended measure defined in \cref{ss:extended-measure}.
For instance, if $\fa \subset \O_X$ is an ideal sheaf and $s > 0$, then one can define
\[
\int_{X^\beth} |\fa|^s d\~\m_X := 
\Big(\int_{V^{\val}} |\fa\.\O_V|^s d\m_V\Big)_{V \subset X}
\]
where $V$ ranges among the closed subvarieties of $X$. 
This defines an element in the extended motivic ring $\~\MM_X$. 
One can look at this \emph{extended motivic integral}
as a way of capturing not only the integral of $\fa$ 
but also of all its restrictions $\fa\.\O_V$.
\end{remark}

\subsection{Change-of-variable formula}
\label{ss:bir-transf-rule}

The approach to motivic integration via Berkovich spaces is 
naturally set up to immediately yield the following key result of the theory. 

\begin{theorem}
\label{t:change-of-vars}
Let $h \colon Y \to X$ be a resolution of singularities, and let $f \in \IMF(X^{\val})$.
Then $(f \o h^{\val})\,|\Jac_h| \in \IMF(Y^{\val})$ and
\[
\int_{X^{\val}} f\,d\m_X = h_* \int_{Y^{\val}} (f\o h^{\val})\,|\Jac_h| \,d\m_Y
\]
in $\MM_X$, where $\Jac_h$ is the Jacobian ideal of $h$.
\end{theorem}

\begin{proof}
As $h$ induces an identification $X^{\val} \simeq Y^{\val}$, 
we can pretend that we are integrating the same function on the same space, 
which is however equipped with two different measures.
The formula follows from the observation that 
since $Y$ is smooth, if $g \colon Y' \to Y$ is any proper birational map
from a smooth variety $Y'$, then $\Jac_{h\o g} = \Jac_h \. \Jac_g$, 
which implies that $\^A_X - \^A_Y = \ord(\Jac_h)$ on the space of valuations.
\end{proof}

\begin{example}
Assume that $h \colon Y \to X$ is a proper birational morphism of smooth varieties, and let
$B$ be an $\R$-divisor on $X$ such that $(X,-B)$ is Kawamata log terminal. Then 
\[
\int_{X^{\val}} e^{-\ord(B)} \,d\m_X = h_*\int_{Y^{\val}} e^{-\ord(K_{Y/X} + h^*B)} \,d\m_Y
\]
in $\MM_X$, where $K_{Y/X}$ is the relative canonical divisor. When $B$ is an integral divisor, 
this recovers, modulo $\L_X-1$, the analogous formula in the usual motivic integration
\[
\int_{X_\infty} \L_X^{-\ord(B)} \,d\m^{X_\infty} = h_*\int_{Y_\infty} \L_X^{-\ord(K_{Y/X} + h^*B)} \,d\m^{Y_\infty}.
\]
\end{example}

\section{Push-forward and functoriality}
\label{s:pushforward}

The purpose of this section is to extend the theory of integration on Berkovich 
spaces introduced in the previous pages into a functorial theory, 
in the spirit of \cite{CL08}.

Throughout this section, we fix a variety $Z$ and a local snc model $\t \colon (Z_\t,F) \to Z$.

\subsection{Category of models and restrictions of quasi-monomial valuations}
\label{ss:category-of-models}

We start by introducing a category of local snc models. We work over the model $\t$ fixed above. 

\begin{definition}
\label{d:Lsnc}
We denote by $\Snc_\t$ the category whose objects are commutative diagrams 
\[
\xymatrix{
(X_\p,D) \ar[r]^-\p \ar[d]_{p_{\p\t}} & X \ar[d]^{p} \\
(Z_\t,F) \ar[r]^-\t & Z
}
\]
where $X$ is a variety, $\p$ is a local snc model,
$p$ is a dominant morphism, and $\Supp(p_{\p\t}^*(F)) \subset \Supp(D)$.
For short, we say that $(X_\p,D)$ is a local snc model (over $X$) \emph{above} $\t$. 
A morphism between two local snc models $(X_\p,D)$ and $(Y_\s,E)$ above $\t$ 
is a commutative diagram
\[
\xymatrix{
(X_\p,D) \ar[r]^-\p \ar[ddr]_{p_{\p\t}} \ar[rd]^{b_{\p\s}} & X \ar[ddr]_(.25){p} |!{[d];[dr]}\hole \ar[rd]^b & \\
& (Y_\s,E) \ar[r]^-\s \ar[d]^{q_{\s\t}} & Y \ar[d]^{q} \\
& (Z_\t,F) \ar[r]^-\t & Z
}
\]
such that $\Supp(b_{\p\s}^*(E)) \subset \Supp(D)$. 
The composition of two morphisms is defined in the obvious way,
by composing the respective diagrams. 
\end{definition}

Every morphism $b_{\p\s} \colon (X_\p,D) \to (Y_\s,E)$ of local snc models above $\t$
induces the map on sets of quasi-monomial valuations
\[
b_{\p\s}^{\Sk} := r_\s \o b^{\val}|_{\Sk_{\p}} \colon \Sk_{\p} \to \Sk_{\s}.
\]

\begin{lemma}
\label{l:p^QM}
$b_{\p\s}^{\Sk} \o r_\p = r_\s \o b^{\val}$. 
\end{lemma}

\begin{proof}
For any $v_x \in X^{\val}$, we need to show that
$b^{\Sk}(r_\p(v_x)) = r_\s(b^{\val}(v_x))$.
Note that, by the definition of $b^{\Sk}$, the left-hand-side of this equation is equal to $r_\s (b^{\val}(r_\p(v_x)))$. 
It is therefore enough to show that, writing $E = \sum_j E_j$ where $E_j$ are the irreducible components, we have
$b^{\val}(r_\p(v_x))(E_j) = b^{\val}(v_x)(E_j)$ for all $j$.
Since $\Supp(b_{\p\s}^*(E)) \subset \Supp(D)$, we can write $b^*E_j = \sum_i a_{ij}D_i$, 
where as usual $D_i$ are the irreducible components of $D$. We have
\[
b^{\val}(r_\p(v_x))(E_j) 
= \sum a_{ij}\. r_\p(v_x)(D_i)
= \sum a_{ij}\. v_x(D_i)
= b^{\val}(v_x)(E_j), 
\]
which gives what we need.
\end{proof}

Next, we define a pull-back homomorphism 
\[
b^* \colon \MF(Y^{\val}) \to \MF(X^{\val})
\]
along a dominant morphism $b \colon X \to Y$.
We proceed as follows.
Given $g \in \MF(Y^{\val})$, we fix a local snc model $(Y_{\s},E)$ over $Y$ so that
$g$ is represented by an element $g_\s \in \MF(Y^{\val})$.
Let then $(X_\p,D)$ be a local snc model over $X$ above $\s$, 
so that $b_{\p\s} \colon (X_\p,D) \to (Y_\s,E)$ is defined. 
Arguing as in the construction of $\a^* \colon \MF_{\p'} \to \MF_\p$ given in \cref{ss:MC-functions}, 
we obtain a commutative diagram
\[
\xymatrix{
\MM_{Y_\s} \ar[d]_{b_{\p\s}^*} & \FF_\s^\o \ar[d] \ar[l] \ar[r] & \FF_\s \ar[d] \\
\MM_{X_\p} & \FF_\p^\o \ar[l] \ar[r] & \FF_\p 
}
\]
where $b_{\p\s}^*$ is the ring homomorphism
defined by mapping a class $[V]_\s$ to the class $[V \times_{Y_\s} X_\p]_\p$
and the other vertical arrows are given by pulling back functions via $b_{\p\s}^{\Sk}$. 
We obtain a ring homomorphism $b_{\p\s}^* \colon \MF_\s \to \MF_\p$.
Using \cref{l:p^QM}, we see that if $\p' \ge \p$ and $\s' \ge \s$ are other models, 
with $\p'$ above $\s'$, and $\a \colon (X_{\p'},D') \to (X_\p,D)$
and $\b \colon (Y_{\s'},E') \to (Y_\s,E)$ are the induced maps, then 
$b_{\p'\s'}^* \o \b^* = \a^* \o b_{\p\s}^*$. 
We can therefore define $b^*(g)$ 
to be the element represented by $b_{\p\s}^*(g_\s)$ in $\MF_\p$. 
The commutativity of the above diagram ensures that the definition of $b^*(g)$
is independent of the choice of models. 

\begin{proposition}
The assignment given on objects by $X \mapsto \MF(X^{\val})$
and on morphisms $b \mapsto b^*$
define functors from the category of varieties and dominant morphisms 
to the category of vector spaces over $\R$. 
\end{proposition}

The proof is straightforward and is omitted.

\subsection{Motivic Functions of finite level}
\label{ss:MotivicFunctions}

The remainder of the section is devoted to the construction of push-forwards. 
The next example is meant to serve as a motivation for the definitions that follow. 

\begin{example}
\label{eg:functoriality}
Let $\a \colon (X_{\p'},D_{\p'}) \to (X_\p,D_\p)$ be a morphism of snc models over
a smooth variety $X$
where $D_\p$ is the sum of two prime divisors $D$ and $E$ with $D \cap E \ne \emptyset$, and 
$\a$ is the blow-up along $D \cap E$.
We take $D_{\p'} = \a^*D_\p = D' + E' + F$ where $F$ is the exceptional divisor.
Let $R \subset \Sk_{\p'}$ be the ray spanned by $\val_F$, and let
$S \subset \Sk_\p$ be the images of $R$ via $r_{\p'\p} \colon \Sk_{\p'} \to \Sk_\p$. 
Note that while $R$ is a face of $\Sk_{\p'}$, $S$ is only a potential face in $\Sk_\p$. 
Given a function $f \in \MF(X^{\val})$ determined by some $f_{\p'} \in \MF_{\p'}$, we wish to realize the integral
$\int_{X^{\val}} f\,d\m_X$ as a push-forward of $f_{\p'}$ via some map $\MF_{\p'} \to \MF_{\id_X} = \MM_X$. 
Such push-forward should be functorial.
Suppose that $f_{\p'}$ is supported on $R$.
Then its push-forward via $\MF_{\p'} \to \MF_\p$, which we denote by $f_\p$, must be supported on $S$. 
This means that $f_\p$ is almost-everywhere zero for the measure defined in $\Sk_\p$, hence its push-forward 
via $\MF_\p \to \MM_X$ can only be zero. However, $f_{\p'}$ may contribute nontrivially
to the integral of $f$, since $R$ has nonzero measure. 
The issue here is that even though $R = S$ as subsets in $X^{\val}$, the measures this set
inherits from $\Sk_{\p'}$ and $\Sk_\p$ are different.
The solution to this impasse is to remember the measure the potential face $S$ had when it was realized 
as an actual face $R$. 
\end{example}

Let $X$ be a variety of dimension $n$.
For every $0 \le i \le n-1$, we let $\Func^i(\Sk_\p,\R)$
denote the set of functions that are each
supported within a finite union of potential faces of relative codimension $i$.
Note that $\Func^0(\Sk_\p,\R) = \Func(\Sk_\p,\R)$
and, for $i > 0$, $\Func^i(\Sk_\p,\R)$ is a non-unital associative $\R$-algebra.
Let then $\Func^i_0(\Sk_\p,\R) \subset \Func^i(\Sk_\p,\R)$ be the ideal of functions
whose restriction to any potential face $R$ of relative codimenision $i$
is almost-everywhere zero with respect to the measure $\n_R$.

We define
\[
\GG_\p^i := \Func^i(\Sk_\p,\R)/\Func^i_0(\Sk_\p,\R)
\]
and
\[
\GG_\p := \bigoplus_{i=0}^{n-1} \GG_\p^i.
\]
Note that there is a natural surjective map $\FF_\p \surj \GG_\p^0$. 
The kernel consists of those functions on $\Sk_\p$ 
that are almost-everywhere zero on the faces $\Sk_{\p,I}$
but may fail to be almost-everywhere zero on some potential faces 
of positive relative codimension. 

\begin{definition}
An element in $\GG_\p$ is denoted by $\bm{\g} = (\g^0,\dots,\g^{n-1})$ 
and called a \emph{Function of level $\p$} (with a capital \emph{F} 
to remind us that this element is a vector).
We think of the components $\g^i$ as almost-everywhere defined functions
(on unions of potential faces of relative codimension $i$), 
and simply refer to them as \emph{functions}.
We say that an element $\bm{\g} = (\g^0,\dots,\g^{n-1}) \in \GG_\p$
is \emph{concentrated in relative codimension $0$} if $\g^i = 0$ for $i > 0$. 
\end{definition}

\begin{remark}
\label{r:honest-function-on-refinement}
It might be reassuring to keep in mind that any element $\bm{\g} = (\g^0,\dots,\g^{n-1}) \in \GG_\p$
can be viewed as a Function concentrated in relative codimension $0$
on some sufficiently high refinement $\p'$ of $\p$.
Indeed, for every $i$ let $R_{i,j}$ be the potential faces of relative codimension $i$
where $\g^i$ is supported. Let $\p' \ge \p$ be a refinement such that each $R_{i,j}$
is an actual face of $\Sk_{\p'}$. Note that $\Sk_{\p'} = \Sk_\p$ as subsets of $X^{\val}$, 
and $\FF_{\p'} = \FF_\p$.
We can then find an element $\f \in \FF_\p$ to represent $\bm{\g}$, in the following sense. 
For every index set $I$, let $i(I)$ be the relative codimension of $\Sk_{\p',I}$
as a potential face of $\Sk_\p$, and define $\f_i \in \FF_\p$ to be the function
that agrees with $\g^i$ on each face $\Sk_{\p',I}$ with $i(I) = i$ and is zero elsewhere. 
We then take $\f := \sum_i \f_i \in \FF_\p$. 
By construction, the restriction of $\f$ to any $R_{i,j}$ agrees (almost-everywhere) with
the corresponding $\g^i$,
and the function is zero on the complement of $\bigcup_{i,j} R_{i,j}$. 
We can regard $\f$ as defining an element in $\GG_{\p'}$ concentrated in relative codimension $0$.
\end{remark}

Each $\GG_\p^i$ has a natural $\FF_\p^\o$-module structure
via multiplication of functions, and this induces an $\FF_\p^\o$-module structure on $\GG_\p$. 
We define 
\[
\MG_\p := \MM_{X_\p} \otimes_{\FF_\p^\o} \GG_\p  = \bigoplus_{i=0}^{n-1} \MG_\p^i
\]
where $\MG_\p^i := \MM_{X_\p} \otimes_{\FF_\p^\o} \GG^i_\p$.
Note that $\MG_\p^0$ is a quotient of $\MF_\p$.

\begin{definition}
An element of $\MG_\p$ is denoted by $\bm{g} = (g^0,\dots,g^{n-1})$, 
where $g^i \in \MG_\p^i$, and is called a 
\emph{motivic Function of level $\p$} (with a capital \emph{F}).
We say that an element $\bm{g} = (g^0,\dots,g^{n-1}) \in \MG_\p$
is \emph{concentrated in relative codimension $0$} if $g^i = 0$ for $i > 0$. 
\end{definition}

Returning to the setting fixed at the beginning of the section, 
where we work over a variety $Z$ and a local snc model $\t \colon (Z_\t,F) \to Z$, 
let 
\[
b_{\p\s} \colon (X_\p,D) \to (Y_\s,E)
\]
be a morphism of local snc models above $\t$ (i.e., a morphism in the category $\Snc_\t$).
Let $R \subset \Sk_\p$ be a potential face, and
let $S \subset \Sk_\s$ be the image of $R$
via the map $b_{\p\s}^{\Sk} \colon \Sk_\p \to \Sk_\s$ defined in \cref{ss:category-of-models}.
For short, we denote by $b_{R} \colon R \to S$ the induced map on these cones.

\begin{lemma}
There is a decomposition
$S = \big(\bigsqcup_{\a=1}^m S_\a\big) \sqcup S'$ 
where $S_\a$ are potential faces of $\Sk_\s$
of the same dimension $s$ and $S' \subset S$ is a finite union of potential faces of dimension $< s$. 
\end{lemma}

\begin{proof}
Let $I$ be the index set such that $R \subset \Sk_{\p,I}$ and $|I| = \dim R$ (i.e., $R$ is an open 
subset of $\Sk_{\p,I}$). The condition that $\Supp(b_{\p\s}^*(E)) \subset \Supp(D)$ implies that
$b_{\p\s}(D_I)$ is contained in some stratum $E_J$. We pick $J$ such that $E_J$ is minimal
with this property. Then $S \subset \Sk_{\s,J}$. Furthermore, the closure of $S$
in $\ov{\Sk_{\s,J}} \cong \R_{\ge 0}^t$ is a convex rational polyhedral cone, 
spanned over $\R_{\ge 0}$ by the divisorial valuations 
$b_{\s\p}(\val_{D_i})$ for $i \in I$. 
As a convex polyhedral cone, the relative interior of $S$ is equidimensional and is dense in $S$. 
We let $s$ be the relative dimension of $S$.

The fact that $S$ is a finite union of potential faces
can be checked for instance using toric geometry and functorial resolution 
of singularities, as follows. Denoting by 
$\D$ the standard fan in $\ov{\Sk_{\s,J}}$, let $X(\D) = \A^k$ be the toroidal chart of $Y_\s$ centered
at the generic point of $E_J$, and let $\S$ be any rational polyhedral fan 
with $|\S| = |\D|$ such that $S$ is a union of faces of $\S$. 
Since $\S$ is a refinement of $\D$, we have a proper birational toric 
morphism $\f \colon X(\S) \to X(\D)$. By \cite[Theorem~3.26]{Kol07}, 
we can find a resolution of indeterminacies of $\f^{-1}$
given by a sequence of toric blow-ups with smooth irreducible centers. This corresponds
to a refinement $\D'$ of $\D$, and
the induced map $\f' \colon X(\D') \to X(\S)$ expresses $\D'$ as a refinement of $\S$, 
showing that $S$ is a union of faces of $\D'$. 
The given sequence of smooth toric blow-ups giving $X(\D') \to X(\D)$
determines a sequence of blow-ups of strata, starting from $Y_\s$, which produces
a model $Y_{\s'}$, and by construction $S$ becomes a union of faces on $\Sk_{\s'}$. 
To conclude, just notice that the union of faces of $\Sk_{\s'}$ of dimension $s$ contained in $S$ is dense in $S$,
and all other faces contained in $S$ have dimension $< s$.
\end{proof}

In the setting of the lemma, we write $S \aeeq \bigsqcup_{\a=1}^m S_\a$.
We stress that with this notation we tacitly assume that all $S_\a$ are potential faces of the same dimension.
The sum of the forms $\om_{S_\a}$ on $S_\a$ defines almost-everywhere, 
via the embedding $\bigsqcup_{\a=1}^m S_\a \subset S$, an $s$-form on $S$
which we denote by $\om_S$. 
Letting $r = \dim R$, we can pick an $(r-s)$-form $\om_{R/S}$ on $R$ such that
\[
\om_R = \om_{R/S} \wedge b_R^*(\om_S).
\]
We remark that the choice of $\om_{R/S}$ is not unique, but its restriction to the fibers
of $b_R$ is independent of the choice. 
Note that if $R' \subset R$ is another potential face of $\Sk_\p$ of the same dimension $r$
then its image $S' \subset \Sk_\s$ has the same dimension $s$ of $S$ and, up to sign, 
$\om_{R/S}|_{R'} = \om_{R'/S'}$. 
For every $y \in S$, the fiber $b_R^{-1}(y)$ has dimension $r-s$
and the restricted form $\om_{R/S}|_{b_R^{-1}(y)}$ is a top form on the fiber.

\begin{definition}
An element $\g \in \GG_\p^i$ is \emph{relatively integrable over $\s$}
if for every potential face $R \subset \Sk_\p$ of relative codimension $i$
we have
\[
\int_{b_R^{-1}(y)} |\g\,\om_{R/S}| < \infty 
\]
for almost every $y \in S$. We denote by $\relIG{\s}_\p^i \subset \GG_\p^i$
the subspace of relatively integrable functions over $\s$. 
\end{definition}

\begin{remark}
Since an element $\g\in \GG_\p^i$ is uniquely defined almost-everywhere on a potential face $R$
of codimension $i$, 
its restriction to $b_R^{-1}(y)$ is uniquely defined almost-everywhere for almost every $y \in S$.
Therefore the condition of integrability is well-posed.
\end{remark}

\begin{definition}
A Function of level $\p$ on $X^{\val}$ is \emph{relatively integrable over $\s$} if it belongs to 
\[
\relIG{\s}_\p := \bigoplus_{i=0}^{n-1} \relIG{\s}_\p^i.
\]
Similarly, a motivic Function of level $\p$ on $X^{\val}$
is \emph{relatively integrable over $\s$} if it belongs to the module
\[
\relIMG{\s}_\p := \MM_{X_\p} \otimes_{\FF_\p^\o} \relIG{\s}_\p.
\]
If $\s = \id_Y$, then we write $\relIG{Y}_\p$ and $\relIMG{Y}_\p$.
If moreover $Y = \Spec k$, then we write $\relIG{k}_\p$ and $\relIMG{k}_\p$.
\end{definition}

A special case of particular interest is when we take $Y = X$ and $\s = \id_X$. 

\begin{definition}
A Function of level $\p$ on $X^{\val}$ is \emph{integrable} if it belongs to 
$\IG_\p := \relIG{X}_\p$.
Similarly, a motivic Function of level $\p$ on $X^{\val}$
is \emph{integrable} if it belongs to the module
$\IMG_\p := \relIMG{X}_\p$.
\end{definition}

\subsection{Push-forward}
\label{ss:PushForward}

Given a morphism 
$b_{\p\s} \colon (X_\p,D) \to (Y_\s,E)$ 
in the category $\Snc_\t$, the first step is to define a push-forward map 
$b_{\p\s!} \colon \relIG{\t}_\p \to \relIG{\t}_\s$.

Every element in $\relIG{\t}_\p$ can be written as a finite sum of elements of the 
form $\bm{\g} = (\g^0,\dots,\g^{n-1})$ where each $\g^i$ is supported on a 
unique face $R_i$ of relative codimension $i$.
For simplicity, we first define the push-forward on elements of this form, and then
extend by linearity. 

So, let $\bm{\g} = (\g^0,\dots,\g^{n-1}) \in \relIG{\t}_\p$ be an element as above.
Let $\g := \g^i$ be one of the components and $R = R_i$ the corresponding face of 
relative codimension $i$.
Let $S \subset \Sk_\s$ and $T \subset \Sk_\t$ be the images of $R$, 
and let $b_R \colon R \to S$, $p_R \colon R \to T$, and $q_S \colon S \to T$ be the induced maps.
Note that there are index sets $I$, $J$, and $K$ such that $R \subset \Sk_{\p,I}$, 
$S \subset \Sk_{\s,J}$ and $T \subset \Sk_{\t,K}$. 
In particular, $R$ has relative codimension $i$
and, writing $S \aeeq \bigsqcup S_\a$ and $T \aeeq \bigsqcup T_\b$ (with the above convention), 
we have that each $S_\a$ is a potential face of the same relative codimension $j$ and
each $T_\b$ is a potential face of the same relative codimension $k$. 
As explained before, this induces forms $\om_{R/S}$ and $\om_{R/T}$ defined almost-everywhere on $R$
and a form $\om_{S/T}$ defined almost-everywhere on $S$. 
Since $\g$ is relative integrable over $\t$, we have
\[
\int_{p_R^{-1}(z)} |\g\,\om_{R/T}| < \infty
\]
for almost every $z \in T$. Along the fibers $p_R^{-1}(z)$, we have 
$|\om_{R/T}| = |\om_{R/S} \wedge b_R^*(\om_{S/T})|$ for almost every $z \in T$. 
Applying Fubini's theorem, we deduce that the quantity
\[
\ff(y) := \int_{b_R^{-1}(y)} |\g\,\om_{R/S}|
\]
is well-defined (and finite) for almost every $y \in S$, and defines almost-everywhere
a function on $S$ that is measurable and integrable over $\t$. Moreover, we have
\[
\int_{p_R^{-1}(z)} |\g\,\om_{R/T}| = \int_{q_S^{-1}(z)} |\ff\,\om_{S/T}| 
\]
for almost every $z \in T$. Clearly, $\ff$, 
as a function defined almost-everywhere on $S$, 
is independent of the choice of representative of $\g$. 
Extending $\ff$ by zero outside of $S$, we obtain an element 
in $\relIG{\t}_\s^{(j)}$, which we denote by $b_{\p\s!}(\g)$.

For every $\bm{\g} \in \relIG{\t}_\p$ as above, 
we define the push-forward $b_{\p\s!}(\bm{\g})$ of $\bm{\g}$ by $p_{\p\s}$ 
to be the element of $\relIG{\t}_\s$ whose component $b_{\p\s!}(\bm{\g})^j \in \relIG{\t}_\s^{(j)}$
is the sum of all $b_{\p\s!}(\g^i)$ that are supported, according to the above 
construction, on a potential face $S_i$ of relative codimension $j$. 
We define the push-forward map
\[
b_{\p,\s!} 
\colon \relIG{\t}_\p \to \relIG{\t}_\s
\]
by extending the definition by linearity.

With the notation as in \cref{d:Lsnc}, it follows from what observed above that 
$p_{\p\t!} = q_{\s\t!} \o b_{\p\s!}$. 
A similar argument shows functoriality, which is stated next.

\begin{proposition}
\label{t:relIG-functorial}
The assignment given on objects by $(X_\p,D) \mapsto \relIG{\t}_\p$
and on morphisms by $b_{\p\s} \mapsto b_{\p\s!}$
defines a functor from $\Snc_\t$ to the category of vector spaces over $\R$. 
\end{proposition}

A more careful application of Fubini--Tonelli's theorem yields the following property. 

\begin{proposition}
\label{p:Fubini-Tonelli-for-G}
Let $b_{\p\s} \colon (X_\p,D) \to (Y_\s,E)$ be a morphism of locally snc models over $\t$
and $\bm{\g} = (\g^0,\dots,\g^{n-1}) \in \GG_\p$ be any element.
\begin{enumerate}
\item
If $\bm{\g} \in \relIG{\t}_\p$ then $\bm{\g} \in \relIG{\s}_\p$ 
and $b_{\p\s!}(\bm{\g}) \in \relIG{\t}_\s$. 
\item
Assuming $\g^i \ge 0$ for every $i$, the converse holds too, hence we have that
$\bm{\g} \in \relIG{\t}_\p$ if and only if $\bm{\g} \in \relIG{\s}_\p$ 
and $b_{\p\s!}(\bm{\g}) \in \relIG{\t}_\s$. 
\end{enumerate}
\end{proposition}

\begin{remark}
In the setting considered in~(2), the condition that $\bm{\g} \in \relIG{\s}_\p$ 
ensures that the push-forward $b_{\p\s!}(\bm{\g})$ is well-defined. 
\end{remark}

\begin{proof}[Proof of \cref{p:Fubini-Tonelli-for-G}]
We prove the proposition component by component. Let $\g = \g^i$ be any component of $\bm{\g}$.
By linearity, we can assume that $\g$ is supported on just one potential face $R$ of codimension $i$. 
Let $b_R \colon R \to S$, $p_R \colon R \to T$, and $q_S \colon S \to T$
be the induced maps and $\om_{R/S}$, $\om_{R/T}$ and $\om_{S/T}$ the corresponding forms. 
Let $j$ denote the relative codimension of $S$. 

Assume first $\g \in \relIG{\t}_\p^i$. 
As $|\om_{R/T}| = |\om_{R/S} \wedge b_R^*(\om_{S/T})|$,
Fubini's theorem implies that
\[
\int_{b_R^{-1}(z)} |\g\,\om_{R/S}| < \infty
\]
for almost every $z \in T$, and therefore $b_{\p\s!}(\g)$ is well-defined.
Furthermore, we have
\[
\int_{q_S^{-1}(z)} |b_{\p\s!}(\g)\,\om_{R/T}| 
= \int_{q_S^{-1}(z)} \Big(\int_{b_R^{-1}(z)} |\g\,\om_{R/S}| \Big)|\om_{S/T}| 
= \int_{p_R^{-1}(z)} |\g\,\om_{R/T}| < \infty
\]
for almost every $z \in T$.
This means that $\g \in \relIG{\s}_\p^i$ 
and $b_{\p\s!}(\g) \in \relIG{\t}_\s^{(j)}$, 
which gives~(1).

Conversely, assume now that $\g \ge 0$, and that $\g \in \relIG{\s}_\p^i$ 
and $b_{\p\s!}(\g) \in \relIG{\t}_\s^{(j)}$. 
Then Tonelli's theorem implies that $\g \in \relIG{\t}_\p^{(j)}$, and this gives~(2).
\end{proof}

The next step is to extend the definition of push-forward to motivic Functions of level $\p$.
As before, let $p_{\p\s} \colon (X_\p,D) \to (Y_\s,E)$
be a morphism in $\Snc_\t$.
Recall that $\relIMG{\t}_\p = \MM_{X_\p} \otimes_{\FF_\p^\o} \relIG{\t}_\p$.
Since the inverse image of a face $\Sk_{\s,J}$ under 
$b_{\p\s}^{\Sk} \colon \Sk_\p \to \Sk_\s$ is a union of faces, pull-back along this map
defines a natural ring homomorphism $\FF_\s^\o \to \FF_\p^\o$.

\begin{lemma}
Regarding $\MM_{X_\p}$ and $\relIG{\t}_\p$ as $\FF_\s^\o$-modules
via the ring homomorphism $\FF_\s^\o \to \FF_\p^\o$, the push-forward maps 
$b_{\p\s*} \colon \MM_{X_\p} \to \MM_{Y_\s}$
and
$b_{\p\s!} \colon \relIG{\t}_\p \to \relIG{\t}_\s$
are $\FF_\s^\o$-module homomorphisms.
\end{lemma}

\begin{proof}
The assertion is clear for $b_{\p\s!}$. If we write $I \succeq J$
whenever $b_{\p\s}(D_I^\o) \subset E_J^\o$, then we have
\begin{multline*}
b_{\p\s*}\big(b_{\p\s}^*(1_{\Sk_{\s,J}}) \.[V]_\p \big)
= b_{\p\s*}\Big( \sum_{I\succeq J} 1_{\Sk_{\p,I}}\.[V]_\p\Big) 
= b_{\p\s*}\Big( \sum_{I\succeq J} [D_I^\o]_\p \.[V]_\p\Big) = \\ 
= b_{\p\s*}\big( [E_J^\o \times_{Y_\s} X_\p]_\p \.[V]_\p\big) 
= [E_J^\o]_\s \. [V]_\s 
= 1_{\Sk_{\s,J}} \. [V]_\s. 
\end{multline*}
By linearity, this proves the assertion for $b_{\p\s*}$.
\end{proof}

\begin{lemma}
\label{l:rho}
There is a natural map 
\[
\r_{\p\s} \colon
\MM_{X_\p} \otimes_{\FF_\p^\o} \relIG{\t}_\p \to \MM_{X_\p} \otimes_{\FF_\s^\o} \relIG{\t}_\p.
\]
\end{lemma}

\begin{proof}
Note that given $r\in {\FF_\p^\o}$, it is not necessarily true that multiplication by $r$ on the left and on the right on the group $\MM_{X_\p} \otimes_{\FF_\s^\o} \relIG{\t}_\p$ give the same result, 
so more care is needed to define the map.
We first claim that there are decompositions 
\[
\MM_{X_\p} \otimes_{\FF_\p^\o} \relIG{\t}_\p\cong
 \bigoplus_I \big([D_I^\o]\cdot\MM_{X_\p}\big) \otimes_\R \big(1_{\Sk_{\p,I}}\cdot\relIG{\t}_\p\big)
\]
and 
\[
\MM_{X_\p} \otimes_{\FF_\s^\o} \relIG{\t}_\p\cong
 \bigoplus_J\Big(\bigoplus_{I\succeq J}[D_I^\o]\.\MM_{X_\p}\Big)\otimes_\R \Big(\bigoplus_{I\succeq J} 1_{\Sk_{\p,I}}\.\relIG{\t}_\p\Big).
\]
Indeed, multiplication by $1_{\Sk_{\p,I}}\in \FF_\p^\o$ gives the projection 
\[
\MM_{X_\p} \otimes_{\FF_\p^\o} \relIG{\t}_\p \to 
\big([D_I^\o]\.\MM_{X_\p}\big) \otimes_\R \big(1_{\Sk_{\p,I}}\.\relIG{\t}_\p\big), 
\]
and multiplication by
$1_{\Sk_{\s,J}}\in \FF_\s^\o$ gives the projection 
\[
\MM_{X_\p} \otimes_{\FF_\s^\o} \relIG{\t}_\p \to 
\Big(\bigoplus_{I\succeq J}[D_I^\o]\.\MM_{X_\p}\Big)\otimes_\R \Big(\bigoplus_{I\succeq J} 1_{\Sk_{\p,I}}\.\relIG{\t}_\p\Big).
\] 
Then the obvious inclusion 
\[
\bigoplus_I \big([D_I^\o]\.\MM_{X_\p}\big) \otimes_\R \big(1_{\Sk_{\p,I}}\.\relIG{\t}_\p \big) 
\subset
\bigoplus_J\Big(\bigoplus_{I\succeq J}[D_I^\o]\.\MM_{X_\p}\Big)\otimes_\R \big(\bigoplus_{I\succeq J} 1_{\Sk_{\p,I}}\.\relIG{\t}_\p\Big)
\]
defined using the isomorphism 
\[
\Big(\bigoplus_{I\succeq J}[D_I^\o]\.\MM_{X_\p}\Big)\otimes_\R \Big(\bigoplus_{I\succeq J} 1_{\Sk_{\p,I}}\.\relIG{\t}_\p\Big)
\cong
\bigoplus_{I,I'\succeq J} \big([D_I^\o]\.\MM_{X_\p}\big)\otimes_\R \big(1_{\Sk_{\p,I'}}\.\relIG{\t}_\p\big)
\]
gives us the desired map.
\end{proof}

The push-forward
\[
b_{\p\s!} 
\colon \relIMG{\t}_\p \to \relIMG{\t}_\s
\]
is defined by the composition
\[
\MM_{X_\p} \otimes_{\FF_\p^\o} \relIG{\t}_\p
\xrightarrow{\;\r_{\p\s}\;}
\MM_{X_\p} \otimes_{\FF_\s^\o} \relIG{\t}_\p
\xrightarrow{\;b_{\p\s*} \otimes b_{\p\s!}\;}
\MM_{Y_\s} \otimes_{\FF_\s^\o} \relIG{\t}_\s
\]
This is a linear map of vector spaces over $\R$. \cref{t:relIG-functorial} and the obvious functoriality of
$\r_{\p\s}$ and $b_{\p\s*}$ yields the following property. 

\begin{proposition}
\label{t:relIMG-functorial}
The assignment given on objects by $(X_\p,D) \mapsto \relIMG{\t}_\p$
and on morphisms by $b_{\p\s} \mapsto b_{\p\s!}$
defines a functor from $\Snc_\t$ to the category of vector spaces over $\R$. 
\end{proposition}

\subsection{Integrable motivic Functions}

Still working above $\t$,
any morphism $\a \colon (X_{\p'},D') \to (X_\p,D)$ of local snc models over the same variety $X$ 
induces a push-forward homomorphism $\a_! \colon \relIMG{\t}_{\p'} \to \relIMG{\t}_\p$.
We can therefore define
\[
\relIMG{\t}(X^{\val}) := \invlim_\p \relIMG{\t}_\p.
\]
This inverse limit 
has a natural module structure over $\FF_\p^\o$ for all local snc models $\p$ above $\t$, and hence
a natural module structure over $\FF^\o(X^{\val})$. 

\begin{definition}
We call $\relIMG{\t}(X^{\val})$ the module of 
\emph{relatively integrable motivic Functions} on $X^{\val}$ \emph{over $\t$}.
When $\t = \id_Z$, we write $\relIMG{Z}(X^{\val})$. 
\end{definition}

If $\t' \ge \t$ is a higher local snc model over $Z$ then by \cref{p:Fubini-Tonelli-for-G}
we have a natural inclusion
$\relIMG{\t}(X^{\val}) \subset \relIMG{\t'}(X^{\val})$. We can therefore define
\[
\relIMG{Z^{\val}}(X^{\val}) := \dirlim_\t \relIMG{\t}(X^{\val}).
\]

\begin{definition}
We call $\relIMG{Z^{\val}}(X^{\val})$ the module of 
\emph{relatively integrable motivic Functions} on $X^{\val}$ \emph{over $Z^{\val}$}.
\end{definition}

For any dominant morphism $b \colon X \to Y$ over $Z$
and any local snc model $\t$ over $Z$, we obtain by functoriality a push-forward homomorphism
\[ 
b_! \colon \relIMG{\t}(X^{\val}) \to \relIMG{\t}(Y^{\val})
\]
defined by $b_!(\bm{g}) := (b_{\p\s!}(\bm{g}_\p))_\s$
where $\p$ is any local snc model above $\t$ such that $\bm{g}$ is determined by $\bm{g}_\p$. 
Letting now $\t$ vary, this induces a push-forward map (denoted by the same symbol)
\[
b_! \colon \relIMG{Z^{\val}}(X^{\val}) \to \relIMG{Z^{\val}}(Y^{\val}).
\]
By \cref{t:relIMG-functorial}, we obtain the following property.

\begin{theorem}
\label{t:functoriality}
The assignments $X \mapsto \relIMG{\t}(X^{\val})$ and $X \mapsto \relIMG{Z^{\val}}(X^{\val})$,
paired with the corresponding assignment $b \mapsto b_!$,
define functors from the category of varieties endowed with a 
dominant morphism to $Z$, whose morphisms are dominant morphisms defined over $Z$, 
to the category of vector spaces over $\R$. 
\end{theorem}

By interpreting the push-forward $b_! \colon \relIMG{Z^{\val}}(X^{\val}) \to \relIMG{Z^{\val}}(Y^{\val})$ 
as integration along the fibers of $b^{\val} \colon X^{\val} \to Y^{\val}$, 
this result can be regarded as a Fubini theorem.

An element $\bm{g} \in \relIMG{\t}(X^{\val})$ is given by a net of motivic Functions $\bm{g} = (\bm{g}_\p)_\p$
where $\bm{g}_\p = (g_\p^0,\dots,g_\p^{n-1}) \in \relIMG{\t}_\p$ 
and $\a_!(\bm{g}_{\p'}) = \bm{g}_\p$ for $\p' \ge \p$ (here, as usual, $\a \colon X_{\p'} \to X_\p$
is the induced map). The next proposition says that, as long as we restrict to a smaller collection $\cC$
of models, we can view any element $\bm{g}$ as being given by a net $(\bm{g}_\p)_{\p \in \cC}$
of motivic Functions that are concentrated in relative codimension 0. 

\begin{proposition}
\label{p:selection-of-snc-models-C}
For every $\bm{g} \in \relIMG{\t}(X^{\val})$ there exists a collection $\cC$ of local snc models
over $X$ such that:
\begin{enumerate}
\item
for every local snc model $\p$ over $X$ there exists $\p' \in \cC$ that is a refinement of $\p$;
\item
if $\p \in \cC$ and $\p'$ is a refinement of $\p$, then $\p' \in \cC$;
\item
for every $\p \in \cC$, $\bm{g}_\p$ is concentrated in relative codimension $0$.
\end{enumerate}
\end{proposition}

\begin{proof}
For every $\p$, we let $\cC_\p$ be the collection of refinements $\p'$ of $\p$
such that every potential face of $\Sk_{\p}$ on which $\bm{g}_\p$
is not almost-everywhere zero becomes a face on $\Sk_{\p'}$. 
Then $\bm{g}_{\p'}$ is concentrated in relative codimension $0$. 
If this were not the case, then there would exist a potential face $R \subset \Sk_{\p'}$
such that $\bm{g}_{\p'}$ is not almost-everywhere zero on $R$. 
By the definition of push-forward, it would follow that 
$\bm{g}_\p$ is not almost-everywhere zero on $R$, viewed as a potential face of $\Sk_\p$, 
but this would contradict our choice of $\p'$. 
The proposition follows by taking $\cC := \bigcup_\p \cC_\p$. 
\end{proof}

For every $\p$ there is a natural push-forward map
\[
t_\p
\colon \relIMG{\t}(X^{\val}) \to \relIMG{\t}_\p.
\]
When $\p = \id_X$, we write $t_X$ for $t_{\id_X}$. 

A special case occurs when we take $Z = X$ and $\t = \id_X$.
Note that $\relIMG{X}_{\id_X} = \MM_X$, hence we have a push-forward map
\[
t_X 
\colon \relIMG{X}(X^{\val}) \to \MM_X.
\]

\begin{definition}
We call 
\[
\IMG(X^{\val}) := \relIMG{X}(X^{\val})
\]
the module of \emph{integrable motivic Functions} on $X^{\val}$.
For every element $\bm{g} \in \IMG(X^{\val})$, we define
the \emph{integral} of $\bm{g}$ to be 
\[
\int_{X^{\val}} \bm{g}\, d\m_X := t_X(\bm{g}) \in \MM_X.
\]
\end{definition}

The next result relates the above definition to the definition of motivic integration
given in \cref{s:functions}. 

\begin{theorem}
\label{t:Lambda}
There is a canonical injective linear map
$\Lambda \colon \IMF(X^{\val}) \inj \IMG(X^{\val})$
such that 
\[
\int_{X^{\val}} f\,d\m_X = \int_{X^{\val}} \Lambda(f) \,d\m_X
\]
for every $f \in \IMF(X^{\val})$.
\end{theorem}

\begin{proof}
For every local snc model $\p$, we consider the linear map
$\Lambda_\p \colon \IMF_\p \to \relIMG{X}_\p$ given by $f_\p \mapsto \bm{g}_\p = (g_\p,0,\dots,0)$
where $g_\p$ is the image of $f_\p$ via the quotient map $\MF_\p \to \MG_\p^0$.
We claim that for every pair of snc models $\p' \ge \p$, if 
$\a \colon X_{\p'} \to X_\p$ is the induced map then the diagram 
\[
\xymatrix{
\IMF_{\p'} \ar[r]^-{\Lambda_{\p'}} & \relIMG{X}_{\p'} \ar[d]^{\a_!} \\
\IMF_\p \ar[r]^-{\Lambda_\p} \ar[u]_{\a^*} & \relIMG{X}_\p
}
\]
is commutative. Note that this implies that the maps $\Lambda_\p$ glue together to give a linear
map $\Lambda \colon \IMF(X^{\val}) \to \IMG(X^{\val})$. The stated equality between integrals
can then be easily checked using $\Lambda_\p$ for any $\p \gg 1_X$. 

To prove the claim, we follow a strategy similar to the proofs of 
\cref{t:existence-measure,t:integral-well-defined}. 
The first step is to check that the diagram commutes
whenever $\a$ is a smooth transversal blow-up. 
By linearly, we can assume that $f$ is represented by an element 
$f_\p = [V]_\p \otimes \f_\p$ where $\f_\p$ is supported on a face $\Sk_{\p,I}$. 
Let $C$ be the center of blow-up. 

Suppose first that $C \not\subset \Supp(D)$. 
This corresponds to \cref{case:t=0} in the proof of \cref{l:existence-measure-one-blowup}.
Note that $C \not\subset D_i$ for any $i \in I$.
As in the proof of \cref{l:existence-measure-one-blowup}, let 
$c = \codim(C,X_\p)$ and set $C^\o := C \cap D_I^\o$, 
which we assume not to be empty. Note that $\codim(C_I^\o,D_I^\o) = c$. 
We denote by $D_0'$ the exceptional divisor of $\a$ 
and, for $i \in I$, let $D_i'$ be the proper transform of $D_i$ on $X_\p'$.
We denote $\~I := I \cup \{0\}$.
Note that $\a$ induces an isomorphism $D'^\o_I \simeq D_I^\o\setminus C^\o_I$
and a $\P^{c-1}$-bundle fibration $D'^\o_{\~I} \to C^\o_I$. 
Setting $\f_{\p'} := \f_\p \o r_{\p'\p}$, we can write
\[
f_\p = [V \times_{X_\p}(D_I^\o \setminus C_I^\o)]_\p \otimes \f_\p + [V \times_{X_\p} C_I^\o]_\p \otimes \f_\p.
\]
Then
\[
\a^*(f_\p) = [V \times_{X_\p}D'^\o_I]_{\p'} \otimes \f_{\p'} + [V \times_{X_\p} D_{\~I}^\o]_{\p'} \otimes \f_{\p'},
\]
and we have $\Lambda_{\p'}(\a^*(f_\p))^0 = \a^*(f_\p)$ in $\MG_\p^0$, 
all other entries of $\Lambda_{\p'}(\a^*(f_\p))$ being zero.
We compute the push-forward term by term. We have
\[
\a_!([V \times_{X_\p}D'^\o_I]_{\p'} \otimes \f_{\p'}) = [V \times_{X_\p}(D_I^\o \setminus C_I^\o)]_\p \otimes \f_\p
\]
and
\[
\a_!([V \times_{X_\p} D_{\~I}^\o]_{\p'} \otimes \f_{\p'})
= c[V \times_{X_\p} C_I^\o] \otimes \int_0^\infty \f_{\p'}\, e^{-cx}dx
= [V \times_{X_\p} C_I^\o]_\p \otimes \f_\p
\]
in $\MG_\p^0$.
This shows that $\a_!(\Lambda_{\p'}(\a^*(f_\p))) = \Lambda_\p(f_\p)$, as desired.

Next, we consider the case where $C \subset \Supp(D)$, 
which corresponds to \cref{case:t-neq-0} in the proof of \cref{l:existence-measure-one-blowup}.
The term $[V \times_{X_\p}(D_I^\o \setminus C_I^\o)]_\p$
behaves as before, so we can focus on the term $[V \times_{X_\p} C_I^\o]_\p \otimes \f_\p$. 
Letting $J \subset I$ and denoting $\~J := J \cup \{0\}$, we have 
\[
\a^*([V \times_{X_\p} C_I^\o]_\p \otimes \f_\p) = 
\sum_{\~J \succeq I} [V \times_{X_\p} D_{\~J}^\o]_{\p'} \otimes \f_{\p'}
\]
in $\MG_\p^0$. 
Similar computations as above, using now \eqref{stra1}--\eqref{stra3} from the proof of 
\cref{l:existence-measure-one-blowup}, show that
\[
\a_!\Big(\sum_{\~J \succeq I,\,|J|\ge |I|-1} [V \times_{X_\p} D_{\~J}^\o]_{\p'} \otimes \f_{\p'}\Big) 
= [V \times_{X_\p} C_I^\o]_\p \otimes \f_\p
\]
in $\MG_\p^0$. 
In general, the terms in the sum $\sum_{\~J \succeq I,\,|J|\le |I|-2} [V \times_{X_\p} D_{\~J}^\o]_{\p'} \otimes \f_{\p'}$
can be nonzero and can be supported on unions of faces of $\Sk_{\p'}$ that come from potential 
faces of $\Sk_\p$ of positive relative codimension, so in principle
they could produce nonzero entries in $\a_!(\Lambda_{\p'}(\a^*(f_\p))$
beyond the first vector component. This however does not occur, as
\eqref{stra4} from the proof of \cref{l:existence-measure-one-blowup}
implies that all the terms in this sum belong to the kernel of $\a_!$.
This verifies the commutativity of the diagram when $C \subset \Supp(D)$. 

Now that we know that the diagram commutes whenever $\a$ is a smooth transversal blow-up, 
it follows by functoriality that the diagram commutes whenever $\a$ is a composition 
of such blow-ups. It therefore suffices to reduce to this case.
Using Hironaka's resolution of marked ideals, we first reduce to the situation where $\p$
and $\p'$ are snc models and $\a$ induces an isomorphism 
$X_{\p'} \setminus D' \cong X_\p \setminus D$. Then, using the weak factorization theorem, 
we obtain a diagram
\[
\xymatrix@C=15pt@R=18pt{
& X_{\p_1'} \ar[dl]_{p_1}\ar[dr]^{q_1} && X_{\p_2'}  \ar[dl]_{p_2}\ar[dr]^{q_2} 
&&&& X_{\p_n'} \ar[dl]_{p_n}\ar[dr]^{q_n} & \\
X_{\p'} && X_{\p_1} && X_{\p_2} & \dots & X_{\p_{n-1}} && X_\p
}
\]
where $p_i$ and $q_i$ are compositions of smooth transversal blow-ups 
and the induced rational maps $\a_i \colon X_{\p_i} \rat X_\p$ are morphisms. 
By the above discussion, the diagram commutes 
with either $p_i$ or $q_i$ in place of $\a$.
If $n=1$, then we have $\a \o p_1 = q_1$, 
and we can use functoriality and the commutativity of the diagram for $p_1$ and $q_1$
to conclude that the diagram commutes for $\a$. In more details, we have
\[
\Lambda_\p 
= q_{1!} \o \Lambda_{\p_1'} \o q_1^* 
= \a_! \o p_{1!} \o \Lambda_{\p_1'} \o p_1^* \o \a^* 
= \a_! \o \Lambda_{\p'} \o \a^*,
\]
which proves the assertion.
We can therefore use induction on $n \ge 1$ and assume that the diagram commutes for $\a_1$. 
By functoriality, this implies that it commutes for $\a_1 \o q_1$. 
Observing that $\a \o p_1 = \a_1 \o q_1$, we conclude by the same argument just used in the 
case $n=1$ that
the diagram commutes for $\a$. This proves the claim. 

It remains to check that $\Lambda$ is injective. 
Let $f \in \IMF(X^{\val})$ be any nonzero element,
and let $\p$ be a model such that $f$ is determined by $f_\p \in \IMF_\p$. 
Then $f_\p \ne 0$, so there is a potential face $R$ of $\Sk_\p$
such that $f_\p|_R$ is not almost everywhere zero with respect to the measure $\om_R$. 
After replacing $\p$ with a higher model, we may assume without loss
of generality that $R$ is an actual face of $\Sk_\p$. 
Then $f_\p$ is non-zero in $\MG_\p^0$, hence $\Lambda(f) \ne 0$.
\end{proof}

\begin{remark}
We do not know if $\IMF(X^{\val})$ is stable under push-forwards. 
More precisely, given a dominant morphism of varieties $b \colon X \to Y$ over $Z$
we do not know if the push-forward 
$b_! \colon \relIMG{Z^{\val}}(X^{\val}) \to \relIMG{Z^{\val}}(Y^{\val})$
sends $\Lambda(\IMF(X^{\val})) \cap \relIMG{Z^{\val}}(X^{\val})$ into 
$\Lambda(\IMF(Y^{\val}))$.
\end{remark}

It might be helpful at this point to work out a simple example.

\begin{example}
\label{e:ramified-cover}
Let $b \colon X = \A^1 \to Y = \A^1$ be the ramified cover given by $v = u^2$, 
with ramification point $P \in X$ and branch point $Q \in Y$.
Consider the models $\p$ given by $(X,P)$ and $\s$ given by $(Y,Q)$.
Notice that even though these models are identities on the underlying varieties, 
they are different from $\id_X$ and $\id_Y$ as they have non-empty
boundaries.
We consider the motivic function $f$ determined 
at level $\p$ by $f_\p = [P]\otimes 1$, 
and view it as a motivic Function via $\Lambda$.
Note that $f$ is supported on the inverse image in $X^{\val}$ of
the 1-dimensional face $R$ of $\Sk_\p$.
Let $x$ be the coordinate on $R$ such that $\val_P$ has coordinate $x=1$.
Similarly, let $S$ be the 1-dimensional face of $\Sk_\s$
with coordinate $y$ such that $\val_Q$ has coordinate $y=1$.
The function $b_R \colon R \to S$ is given by $y=2x$.
The volume forms on these faces are $\om_R = e^{-x}dx$ and $\om_S = e^{-y}dy$, 
hence $\om_{R/S} = \tfrac 12 e^x$. Then $b_!(f)$ is represented at level $\s$ by
\[
b_{\p\s!}(f_\p) = b_{\p\s*}[P] \otimes b_{\p\s!}(1) = [Q] \otimes \tfrac 12 e^{\frac y2}.
\]
We have
\[
t_Y(b_!(f)) = \int_{Y^{\val}} b_!(f)\, d\m_Y 
= [Q] \int_0^\infty \tfrac 12 e^{\frac y2 - y}dy = [Q].
\]
Note that we also have
\[
b_*(t_X(f)) = b_*\Big( \int_{X^{\val}} f\, d\m_X\Big) 
= b_*\Big([P] \int_0^\infty e^{-x}dx\Big) = [Q],
\]
as expected by functoriality.
\end{example}

\subsection{Projection formula}
\label{ss:ProjectionFormula}

Recall that, given a local snc model $(X_\p,D)$, 
there is a natural $\FF_\p$-module structure on $\GG_\p$ given by multiplication.
It is easy to see that the following proposition holds.

\begin{proposition}
Let $b_{\p\s} \colon (X_\p,D) \to (Y_\s,E)$ be a morphism in the category $\Snc_\t$,
and let $\f \in \FF_\s$ and $\bm{\g} \in \relIG{\t}_\p$.
If $b_{\p\s}^*(\f)\.\bm{\g} \in \relIG{\t}_\p$, then 
$\f\.b_{\p\s!}(\bm{\g}) \in \relIG{\t}_\s$ and
\[
b_{\p\s!}(b_{\p\s}^*(\f)\.\bm{\g}) = \f\.b_{\p\s!}(\bm{\g}).
\]
\end{proposition}

There is also a natural $\MF_\p$-module structure on $\MG_\p$.
In concrete terms, the module structure 
is given by setting, for any $f = \sum_i[V_i]_\p \otimes \f_i \in \MF_\p$
and $\bm{g} = \sum_j[W_j]_\p \otimes \bm{\g}_j \in \MG_\p$, 
\[
f\.\bm{g} := \sum_{i,j}[V_i \times_{X_\p} W_j]_\p \otimes (\f_i\bm{\g}_j)
\]
where $\f_i \bm{\g}_j := (\f_i\g_j^{0},\dots, \f_i\g_j^{n-1})$.
As a consequence of the previous proposition, we obtain the following formula.

\begin{proposition}
\label{l:ProjectionFormula}
Let $b_{\p\s} \colon (X_\p,D) \to (Y_\s,E)$ be a morphism in the category $\Snc_\t$,
and let $f \in \MF_\s$ and $\bm{g} \in \relIMG{\t}_\p$.
If $b_{\p\s}^*(f)\.\bm{g} \in \relIMG{\t}_\p$, then 
$f\.b_{\p\s!}(\bm{g}) \in \relIMG{\t}_\s$ and
\[
b_{\p\s!}(b_{\p\s}^*(f)\.\bm{g}) = f\.b_{\p\s!}(\bm{g}).
\]
\end{proposition}

\begin{remark}
When $Z = Y$ and $\t = \s$, the hypothesis that 
$b_{\p\s}^*(f)\.\bm{g} \in \relIMG{\s}_\p$ is automatically satisfied.
\end{remark}

Let now 
$b \colon X \to Y$ 
be a dominant morphism of varieties.
The proposition allows us to define a natural $\MF(Y^{\val})$-module structure on $\relIMG{Y^{\val}}(X^{\val})$.
Indeed, given any $f \in \MF(Y^{\val})$ and $\bm{g} = (\bm{g}_\p)_\p \in \relIMG{Y^{\val}}(X^{\val})$, 
we can fix a high enough local snc model $\s$ over $Y$ such that
$f$ is represented by $f_\s \in \MF_\s(Y^{\val})$ and $\bm{g} = (\bm{g}_\p)_\p \in \relIMG{\s}(X^{\val})$. 
Then we define the action
\[
f\.\bm{g} 
:= (b_{\p\s}^*(f_\p)\.\bm{g}_\p)_\p
\]
The representatives $b_{\p\s}^*(f_\p)\.\bm{g}_\p$ in the right-hand-side are well-defined
for all high enough local snc models $\p$. It is convenient here to remember in the notation of the action
the pull-back map $b^*$, and hence denote the action of $f$ on $\bm{g}$ by $b^*(f)\.\bm{g}$.
\cref{l:ProjectionFormula} yields the following projection formula.

\begin{proposition}
Let $b \colon X \to Y$ be a dominant morphism of varieties.
Then for every $f \in \MF(Y^{\val})$ and $\bm{g} \in \relIMG{Y^{\val}}(X^{\val})$ we have
\[
b_!(b^*(f)\.\bm{g}) = f\.b_!(\bm{g}).
\]
\end{proposition}

\section{Atomic approach}
\label{s:atomic}

An atomic measure on the Berkovich space of a variety 
is introduced in this last section as a parallel way of doing motivic
integration on Berkovich spaces that resembles more closely the classical
motivic integration on arc spaces and the theory of
motivic constructible functions introduced in \cite{CL08}. 
Contrary to the theory developed in the previous sections, 
this approach does not allow to integrate real valued functions, and 
ultimately produces a more restrictive algebra of measurable sets.
The trade off is that it does not require to work modulo $\L-1$ 
and as a result it fully recovers the usual motivic integral. 
This part, which still relies on geometric properties such as resolution of singularity
and weak factorization, also uses some preliminary results from \cite{CL08}.

\subsection{Atomic motivic measure}

As before, let $X$ be a variety over
an algebraically closed field $k$ of characteristic zero equipped with the trivial norm. 
We consider the subspace $X^{\val}(\Z) \subset X^{\val}$ consisting of valuations 
with center in $X$ and values in $\Z$. 
Elements of this space are considered as maps $v \colon k(X)^\times \to \Z$
which are not assumed to be surjective. 
In particular, we distinguish between multiples $mv$ of the same valuation $v$. 

For every local snc model $\p \colon (X_\p,D) \to X$, 
we let $\Sk_\p^\Z := \Sk_\p \cap X^{\val}(\Z)$. This is the set of $\Z$-valued quasi-monomial
valuations determined at level $\p$. 
We denote by $r_\p^\Z \colon X^{\val}(\Z) \to \Sk_\p^\Z$ the restriction of $r_\p$.  
We have $X^{\val}(\Z) = \invlim_\p \Sk_\p^\Z$. 
Writing $D = \sum_{i=1}^rD_i$, we have a decomposition $\Sk_\p = \bigsqcup_I \Sk_{\p,I}$
parameterized by subsets $I \subset \{1,\dots,r\}$. 
On each face $\Sk_{\p,I} \subset \Sk_\p$, 
the intersection $\Sk_{\p,I}^\Z := \Sk_\p^\Z \cap \Sk_{\p,I}$
coincides with the inverse image of the lattice $\Z^s \subset \R^s$
via the isomorphism $\ff_{\p,I} \colon \Sk_{\p,I} \xrightarrow\sim \R^s_{>0}$,
where $s = |I|$.
In particular, this map gives an embedding 
$\ff_{\p,I}^\Z \colon \Sk_{\p,I}^\Z \inj \Z^s$.
This allows us to transfer some of the notions introduced in \cite{CL08} to this setting 
by carrying over the Presburger's language $\bL_\PR$ from 
$\Z^s$ to $\Sk_{\p,I}^\Z$. The theory of Presburger arithmetic has language
\[
\bL_\PR=\{+,-,0,1,\leq\}\cup\{\equiv_n | n\in\mathbb{N}\},
\]
with $\equiv_n$ the equivalence relation modulo $n$ (for more details, see \cite[Chapter~3]{Mar02}). Here we only consider the model $\mathbb{Z}$ for the theory. Note that the multiplication operation is not included in the language. Elimination of quantifiers holds in this setting, and in fact we will be relying on a cell decomposition theorem due to \cite{Clu03}.
 
For short, we set $\L := \L_X$. 
As in \cite{CL08}, we consider the ring
\[
\A_X := \Z[\L,\L^{-1},(1-\L^{-i})^{-1}]_{i\ge 1},
\]
and work with the following version of motivic ring:
\[
\MA_X := K_0(\Var_X)\otimes_{\Z[\L]} \A_X.
\]
Note that $\MA_X \cong K_0(\Var_X)[\L,\L^{-1},(1-\L^{-i})^{-1}]_{i\ge 1}$.

We consider the collection $\S_\PR(X^{\val}(\Z))$ of all subsets of the form
$(r_\p^\Z)^{-1}(S)$ where $\p$ is a local snc model over $X$
and $S \subset \Sk_\p^\Z$ an $\bL_\PR$-definable subset, 
by which we mean that $S_I := S \cap \Sk_{\p,I}^\Z$
is $\bL_\PR$-definable in $\Sk_{\p,I}^\Z$ for all $I$.
This will be the collection of measurable sets. 
It is clear that $\S_\PR(X^{\val}(\Z))$ is closed under finite unions, finite intersections, and complements.
Note that the same element of $\S_\PR(X^{\val}(\Z))$ can be written in different ways, by choosing different models $\p$.

\begin{theorem}
\label{t:atomic-measure}
There is a well-defined function $\m_X^\Z \colon \S_\PR(X^{\val}(\Z)) \to \MA_X$ given by setting
\[
\m_X^\Z\big((r_\p^\Z)^{-1}(S)\big) := \sum_I[D_I^\o] \big(\L-1\big)^{|I|} 
\sum_{S_I} \L^{-\^A_X}, 
\]
where the sums $\sum_{S_I} \L^{-\^A_X}$ are considered as being taken in $\A_X$. 
\end{theorem}

\begin{proof}
The series $\sum_{S_I} \L^{-\^A_X}$ converge in $\Z(\hskip-1pt(\L^{-1})\hskip-1pt)$.
The fact that the sums define elements in $\A_X$ follows from
\cite[Theorem--Definition~4.5.1]{CL08} (see also the discussion at the beginning of
\cref{ss:constr-motivi-fn}).

Checking that the definition of $\m_X^\Z$ is independent of the choice of $\p$ requires a proof. 
Using resolution of singularities and the weak factorization theorem as 
in the proof of \cref{t:existence-measure}, we reduce the proof
of \cref{t:atomic-measure} to the following lemma. 
\end{proof}

\begin{lemma}
Let $\a \colon X_{\p'} \to X_{\p}$ be the blow-up of a smooth variety $C \subset X_{\p}$
with normal crossings with $D$, and let $D' = \sum_{i=0}^rD_i'$
where $D_0'$ is the exceptional divisor of $\a$ and $D_i' = f^{-1}_*D_i$ for $i > 0$. 
Let $S \subset \Sk_\p^\Z$ be an $\bL_\PR$-definable subset, and let
$S' := (r_{\p'\p}^\Z)^{-1}(S)$ where $r_{\p'\p}^\Z \colon \Sk_{\p'}^\Z \to \Sk_\p^\Z$
is the restriction of $r_{\p'\p}$. 
Then $S$ is $\bL_\PR$-definable if and only if $S'$ is $\bL_\PR$-definable, and 
\[
\sum_I[D_I^\o] \big(\L-1\big)^{|I|} \, \sum_{S_I} \L^{-\^A_X}
= 
\sum_I[D'^\o_{I'}] \big(\L-1\big)^{|I'|} \, \sum_{S'_{I'}} \L^{-\^A_X}.
\]
\end{lemma}

\begin{proof}
The proof follows the same steps of the proof of \cref{l:existence-measure-one-blowup},
but the computations need to be adapted to the current definitions. 

As in the proof of \cref{l:existence-measure-one-blowup}, 
it suffices to consider the case where $S \subset \Sk_{\p,I}^\Z$ for some $I$, 
and we can assume that $I = \{1,\dots,s\}$.
For any subset $J \subset I$, we let $\~J := \{0\} \cup J$. 
Let $c = \codim(C,X_{\p})$ and 
set $C_I^\o := C \cap D_I^\o$. 
If $C_I^\o = \emptyset$, then $(r_{\p'\p}^\Z)^{-1}(\Sk_{\p,I}^\Z) = \Sk_{\p',I}^\Z$
and the equality is clear.
Assume therefore that $C_I^\o \ne\emptyset$. We can assume that 
$C \subset D_i$ if and only if $1 \le i \le t$ for some $0 \le t \le s$. 
We have the decomposition $S' = \bigsqcup_{I' \succeq I} S'_{I'}$.
Recall that $I' \succeq I$ if and only if either $I' = I$ or
$I' =\~J$ where $\{t+1,\dots,s\} \subset J \subset I$.
\begin{enumerate}
\item
\label{it:strat1} 
First, note that $[D'^\o_I]=[D^\o_I\setminus C^\o_I]$ in $K_0(\Var_X)$.
\end{enumerate}
We now look at the strata parameterized by $\~J$ where $\{t+1,\dots,s\} \subset J \subset I$. 
\begin{enumerate}
\setcounter{enumi}{1}
\item
\label{it:strat2}
Taking $J = I$, we see that
$D'^\o_{\~I} \to C^\o_I$ is a piecewise locally trivial fibration with fiber $\P_k^{c-t-1}$, hence
$[D'^\o_{\~I}] = [C^\o_I][\P_X^{c-t-1}]$
in $K_0(\Var_X)$.
\item
\label{it:strat3}
If $J\subsetneq I$,
then $D'^\o_{\~{J}} \to C^\o_I$ is a piecewise locally trivial fibration with fiber 
$\mathbb{G}_{m,k}^{s-1-|J|} \times \A_k^{c-t}$, hence
$[D'^\o_{\~J}]=[C^\o_I](\L-1)^{s-1-|J|}\L^{c-t}$
in $K_0(\Var_X)$.
\end{enumerate}
Setting $a_i := \^A_X(\val_{D'_i})$, we have
$a_i = \^A_X(\val_{D_i})$ for $i > 0$ and 
$a_0 = c - t + \sum_{i=1}^t a_i$.

\setcounter{case}{0}
\begin{case}
Suppose $C\subsetneq \Supp(D)$, i.e., $t=0$.
\end{case}

Using the natural inclusions 
$\Sk_{\p,I}^\Z \inj \Z^{s}$
and
$\Sk_{\p',I}^\Z \sqcup \Sk_{\p',\~I}^\Z \inj \Z^{s+1}$,
we fix coordinates $y_1,\dots,y_s$ on $\Sk_{\p,I}^\Z$
and $x_0,\dots,x_s$ on $\Sk_{\p',I}^\Z \sqcup \Sk_{\p',\~I}^\Z$.
The retraction map $r_{\p'\p}^\Z$ restricted to $\Sk_{\p',I}^\Z \sqcup \Sk_{\p',\~I}^\Z$ 
is given by $y_i = x_i$ for $i=1,\dots,s$. By definition, we have
$\^A_X(y_1,\dots,y_s) = \sum_{i=1}^s a_iy_i$ and 
$\^A_X(x_0,\dots,x_s) = \sum_{i=0}^s a_ix_i=cx_0+\sum_{i=1}^{s} a_iy_i$.
We can then deduce in this case the formula stated in the lemma, as follows: 
\begin{align*}
[D'^\o_I] &(\L-1)^s \sum_{S'_I}\L^{-\^A_X}
+
[D'^\o_{\~I}](\L-1)^{s+1}\sum_{S'_{\~I}}\L^{-\^A_X} = \\
&=
[D^\o_I\setminus C^\o_I](\L-1)^s\sum_{S'_I}\L^{-\sum_{i=1}^s a_i x_i}
+
[C^\o_I] [\P_X^{c-1}](\L-1)^{s+1}\sum_{S'_{\~I}}\L^{-\sum_{i=0}^s a_i x_i}\\
&=
[D^\o_I\setminus C^\o_I](\L-1)^s\sum_S\L^{-\sum_{i=1}^s a_i y_i}
+
[C^\o_I](\L^c-1)(\L-1)^s\sum_{S'_{\~I}}\L^{-cx_0-\sum_{i=1}^s a_i y_i}\\
&=
[D^\o_I\setminus C^\o_I](\L-1)^s\sum_S\,\L^{-\^A_X}
+
[C^\o_I](\L-1)^s\sum_S\L^{-\^A_X}.
\end{align*}
The first equality follows from \eqref{it:strat1} and \eqref{it:strat2} above,
the second from the fact that 
$\Sk_{\p',I}^\Z$ maps bijectively to $\Sk_{\p,I}^\Z$, 
and the third from the fact that the fiber of $S'_{\~I}\to S$ over a point $(y_1,\dots,y_s )\in S$ is the set 
$\{(x_0,y_1,\dots,y_s)\mid x_0\in \Z_{\ge 1}\}$.

\begin{case}
Suppose $C\subset\Supp(D)$, i.e., $t\neq 0$.
\end{case}

We have inclusions 
$\Sk_{\p,I}^\Z \inj \Z^{s}$
and
$\Sk_{\p',I}^\Z \sqcup \big( \bigsqcup_J \Sk_{\p',\~J}^\Z \big) \inj \Z^{s+1}$,
where $J$ ranges among subsets of $I$ containing $\{t+1,\dots,s\}$.
Using these inclusions, we fix coordinates $y_1,\dots,y_s$ on $\Sk_{\p,I}^\Z$
and $x_0,\dots,x_s$ on $\Sk_{\p',I}^\Z \sqcup \big( \bigsqcup_J \Sk_{\p',\~J}^\Z \big)$.
The retraction map $r_{\p'\p}^\Z$ restricted to $\Sk_{\p',I}^\Z \sqcup (\bigsqcup_J \Sk_{\p',\~J}^\Z \big)$
is given by $y_i = x_0 + x_i$ for $1 \le i \le t$ and $y_i = x_i$ for $t+1 \le i \le s$. 
By definition, we have
$\^A_X(y_1,\dots,y_s) = \sum_{i=1}^s a_iy_i$ and 
$\^A_X(x_0,\dots,x_s) = \sum_{i=0}^s a_ix_i=(c-t)x_0+\sum_{i=1}^s a_iy_i$.

We now look at the formula stated in the lemma. 
The term involving the stratum $D'^\o_I$ on the right-hand-side of the formula is equal to
\[
[D^\o_I\setminus C^\o_I] (\L-1)^s\sum_S\L^{-\^A_X}.
\]

In order to compute the term that involves the stratum $D'^\o_{\~I}$, we first observe that
the fiber of the restriction map $S'_{\~I}\to S$ over a point $(y_1,\dots,y_s)\in S$ 
consists of the points $(x_0,x_1,\dots,x_s)\in \Z^{s+1}$ subject to the conditions
$x_i\ge 1$ for $1 \le i \le s$, 
$x_i+x_0=y_i$ for $i \le i \le t$, and
$x_i=y_i$ for $t+1 \le i \le s$. 
Thus the term is equal to
\begin{align*}
[D'^\o_{\~I}](\L-1)^{s+1}\sum_{S'_{\~I}}\L^{-\^A_X}
&=[C^\o_I][\P_X^{c-t-1}](\L-1)^{s+1}
\sum_{S'_{\~I}}\,\L^{-(c-t)x_0-\sum_{i=1}^s a_i y_i}\\
&=[C^\o_I](\L^{c-t}-1) (\L-1)^s
\sum_S \sum_{x_0=1}^{\min_{1\le i\le t}{y_i-1}}\L^{-(c-t)x_0-\sum_{i=1}^s a_i y_i}\\
&=[C^\o_I](\L-1)^s\sum_S\L^{-\^A_X}(1-\L^{-(c-t)(\min_{1\le i\le t}{y_i}-1)}).
\end{align*}

We now look at the terms involving $D'^\o_{\~J}$ 
for $J \subsetneq I$. Recall that we always have $\{t+1,\dots,s\} \subset J$.
There is a decomposition
$S=\bigsqcup_{\{t+1,\dots,s\} \subset J\subsetneq I}S^J$
where
\[
S^J := S\cap \big\{(y_1,\dots,y_s)\mid \text{for $1 \le j \le t$, 
$y_j= \min_{1\le i\le t} y_i \,\Leftrightarrow\, j\not\in J$} \big\}.
\]
The retraction map induces bijections $S'_{\~J}\to S^J$,
where the preimage of a point $(y_1,\dots,y_s)\in S^J$
has value $\min_{1\leq i \leq t}y_i$ in its $0$-th coordinate,
$y_i-\min_{1\leq i \leq t}y_i$ in its $i$-th coordinate for $1 \le i \le t$, and
$y_i$ in its $i$-th coordinate for $t+1 \le i \le s$.
The term involving the stratum $D'^\o_{\tilde{J}}$ (for $J \subsetneq I$) is then equal to
\begin{align*}
[D'^\o_{\~J}](\L-1)^{|J|+1}\sum_{S'_{\~J}}\L^{-\^A_X}
&= [C^\o_I] (\L-1)^s \L^{c-t} \sum_{S^J}\L^{-(c-t)\min_{1\leq i\leq t}{y_i}-\sum_{i=1}^s a_iy_i}\\
&=[C^\o_I](\L-1)^s
\sum_{S^J}\,\L^{-\^A_X}\L^{-(c-t)(\min_{1\leq i \leq t}y_i-1)}.
\end{align*}
Summing up the terms above, it is now clear that the formula stated in the lemma holds
in this case as well. This completes the proof of the lemma.
\end{proof}

\subsection{Constructible motivic functions}
\label{ss:constr-motivi-fn}

For every $\p$ and $I$, we consider the ring of functions
\[
\PP_{\p,I} \subset 
\Func(\Sk_{\p,I}^\Z,\A_X)
\]
generated by constant functions with values in $\A_X$, 
$\bL_\PR$-definable functions $\Sk_{\p,I}^\Z \to \Z$, and functions of the form
$\L^\b$ where $\b$ is an $\bL_\PR$-definable function $\Sk_{\p,I}^\Z \to \Z$.

For every $q\in\R_{>1}$, there is a unique ring homomorphism 
$\theta_q\colon \A_X \to \R$ given by $\L\mapsto q$.
We let $\IP_{\p,I} \subset \PP_{\p,I}$ 
be the set of functions $\f$ such that the series $\sum_{v \in \Sk_{\p,I}^\Z} \theta_q(\f(v))$
converges for all $q > 1$. We interpret this as an integrability condition. 
By \cite[Theorem-Definition~4.5.1]{CL08}, there is a unique $\A_X$-module homomorphism
$\sum_{\QM_{\p,I}^\Z} \colon \IP_{\p,I} \to \A_X$ 
satisfying
\[
\theta_q\Big(\sum_{\QM_{\p,I}^\Z}\f\Big) = \sum_{v \in \QM_{\p,I}^\Z} \theta_q(\f(v))
\]
for all $q > 1$. 

Let $\PP_\p \subset \Func(\Sk_\p^\Z,\A_X)$ 
be the set of functions $\f$
that restrict to elements in $\PP_{\p,I}$ for every $I$, 
and let $\PP_\p^\o \subset \PP_\p$ be the subring generated by the constant
function $\L$ and characteristic functions of the form $1_{\Sk_{\p,I}^\Z}$. 
There are natural inclusions $\PP_\p \inj \PP_{\p'}$ for $\p' \ge \p$, given by pull-back, 
and we define $\PP(X^{\val}(\Z)) := \dirlim_\p \PP_\p$ and $\PP^\o(X^{\val}(\Z)) := \dirlim_\p \PP_\p^\o$. 

For every $\p$, there is a unique ring homomorphism 
$\PP_\p^\o \to K_0(\Var_{X_\p})$ sending $\L \mapsto \L_{X_\p}$ and
$r_\p^\Z \o 1_{\Sk_{\p,I}^\Z} \mapsto [D_I^\o]_\p$, 
and $\IP_\p$ is a $\PP^\o_\p$-module under the action given by multiplication. Let
\[
\MP_\p := K_0(\Var_{X_\p}) \otimes_{\PP_\p^\o} \PP_\p
\quad\text{and}\quad
\IMP_\p := K_0(\Var_{X_\p}) \otimes_{\PP_\p^\o} \IP_\p.
\]
For every $\p' \ge \p$, there is a natural pull-back map $\MP_\p \to \MP_\p'$
which restricts to a map $\IMP_\p \to \IMP_{\p'}$. We can therefore define
\[
\MP(X^{\val}(\Z)) := \dirlim_\p \MP_\p
\quad\text{and}\quad
\IMP(X^{\val}(\Z)) := \dirlim_\p \IMP_\p.
\]

\begin{remark}
There is a natural isomorphism
$\MP_\p \cong \bigoplus_I [D_I^\o]_\p\,  K_0(\Var_{X_\p}) \otimes_{\Z[\L]} \PP_{\p,I}$
given by $[V]_\p \otimes \f \mapsto ([V]_\p [D_I^\o]_\p \otimes \f|_{\Sk_{\p,I}^\Z})_I$.
In Cluckers--Loeser's notation \cite{CL08}, 
$D_I^\o \times \Sk_{\p,I}^\Z$ is a globally definable sub-assignment,
and $\MP_\p = \bigoplus_I \cC(D_I^\o \times \Sk_{\p,I}^\Z)$.
\end{remark}

\begin{definition}
An element $f \in \MP(X^{\val}(\Z))$ is called a \emph{constructible motivic function} on $X^{\val}(\Z)$. 
Such element is said to be \emph{integrable} if it belongs to $\IMP(X^{\val}(\Z))$.
If $f$ is an integrable constructible function that is represented by 
an element $f_\p \sum_j [V_j]_\p \otimes \f_j \in \IMP_\p$, then we define the \emph{integral}
of $f$ to be the element of $\MA_X$ given by
\[
\int_{X^{\val}(\Z)} f \, d\m_X^\Z := \sum_I [V_j\times_{X_\p}D_I^\o] \big(\L-1\big)^{|I|} \, 
\sum_{\Sk_{\p,I}^\Z} \f_j\L^{-\^A_X}.
\]
\end{definition}

The same proof of \cref{t:integral-well-defined}, with the obvious adaptation in the
computations, shows that the definition of the integral is independent of any choice. 

Starting from here, one can prove a change-of-variables formula and
develop a functorial theory similarly as done in the previous sections.
The change-of-variables formula is stated next.

\begin{theorem}
\label{t:change-of-vars}
Let $h \colon Y \to X$ be a resolution of singularities, and let $f \in \IMP(X^{\val}(\Z))$.
Then $(f \o h^{\val}(\Z)) \L^{-\ord(\Jac_h)} \in \IMP(Y^{\val}(\Z))$ and
\[
\int_{X^{\val}(\Z)} f\,d\m_X^\Z = h_* \int_{Y^{\val}(\Z)} (f\o h^{\val}(\Z)) \L^{-\ord(\Jac_h)} \,d\m_Y^\Z.
\]
\end{theorem}

\begin{example}
If $X$ is smooth and $D = \sum d_iD_i$ is an effective simple normal crossing divisor, then 
\[
\int_{X^{\val}(\Z)} \L^{-\ord(D)} \, d\m_X^\Z 
= \sum_{I} [D_I^\o] \big(\L-1\big)^{|I|}
\prod_{i \in I} \sum_{m_i = 1}^\infty\L^{-m_i(d_i+1)} 
= \sum_{I} \frac{[D_I^\o]}{\prod_{i \in I} \big[\P_X^{d_i}\big]}.
\]
This agrees with the usual motivic integral defined using arc spaces. 
It follows by \cref{t:change-of-vars} that integration over $X^{\val}(\Z)$ agrees with 
integration over $X_\infty$ for every variety $X$, in the following sense: if $Z$ 
is any proper closed subscheme of a variety $X$ then 
\[
\int_{X^{\val}(\Z)} \L^{-\ord(Z)} \, d\m_X^\Z = \int_{X_\infty} \L^{-\ord(Z)} \, d\m^{X_\infty}.
\]
Note that both integrals take values in $\MA_X$. 
This also means that integration over $X^{\val}(\Z)$ agrees, in the sense discussed in \cite[Section~16.3]{CL08}, 
with the one constructed by Cluckers and Loeser. 
\end{example}

As for the functorial theory, it turns out that it is
actually simpler to define push-forwards using the atomic measure since in this setting
there is no longer need to keep track of potential faces and working with vector functions
as we did in \cref{s:pushforward}. 
Just to give an idea of how the functorial approach can be developed using the
atomic measure, we overview some of the results
one can prove in this setting.

As in \cref{s:pushforward}, we fix a variety $Z$ and a local snc model $\t \colon (Z_\t,F) \to Z$ over it. 
Given a variety $X$ which dominates $Z$, and a local snc model $\p \colon (X_\p,D) \to X$ above $\t$, 
we are going to define the subspace $\relIMP{\t}_\p \subset \MP_\p$ of constructible motivic functions 
that are relatively integrable over $\t$. 
If $b_{\p\s} \colon (X_\p,D) \to (Y_\s,E)$ is a morphism in the category 
$\Snc_\t$, then we will define a push-forward morphism 
\[
b_{\p\s!}^\Z \colon \relIMP{\t}_\p \to \relIMP{\t}_\s,
\]
similarly as it was done in \cref{ss:PushForward}. 

The first step is to define the subring $\relIP{\s}_\p \subset \PP_\p$ of functions that are
relatively integrable over $\s$. For every $I$, let $b_{\p,I}^\Z \colon \Sk_{\p,I}^\Z \to \Sk_\s^\Z$
be the restriction of $b_{\p\s}^{\Sk} \colon \Sk_\p \to \Sk_\s$. 

\begin{definition}
We say that a function $\f \in \PP_\p$ is \emph{relatively integrable over $\s$} 
if for every $I$ and $w \in \Sk_\s^\Z$ the series 
\[
\sum_{v \in (b_{\p,I}^\Z)^{-1}(w)} \theta_q(\f(v))q^{ - \^A_X(v) + \^A_Y(w)}
\]
converges for all $q > 1$. According to \cite[Theorem--Definition~4.5.1]{CL08}, 
this gives an element of $\A_X$
which we denote by $\sum_{(b_{\p,I}^\Z)^{-1}(w)} \f\,\L^{ - \^A_X + \^A_Y}$.
We denote by $\relIP{\s}_\p \subset \PP_\p$ the subring
of relatively integrable functions over $\s$, and define
$\relIMP{\s}_\p := K_0(\Var_X) \otimes_{\PP_\p^\o} \relIP{\s}_\p$.
\end{definition}

The next step is to define push-forward at the level of functions. 
By \cite[Lemma~4.5.7]{CL08}, we have an inclusion $\relIP{\t}_\p \subset \relIP{\s}_\p$, and 
for every $\f \in \relIP{\t}_\p$, the assignment
\[
b_{\p\s!}^\Z(\f)(w) := \sum_I \sum_{(b_{\p,I}^\Z)^{-1}(w)} \f\,\L^{ - \^A_X + \^A_Y}.
\]
for every $w \in \Sk_\s^\Z$ defines an element $b_{\p\s!}(\f) \in \Func(\Sk_{\p,I}^\Z,\A_X)$. 
By \cite[Theorem--Definition~4.5.1]{CL08}, this element belongs to $\PP_\p$, 
and in fact to $\relIP{\t}_\s$ by \cite[Lemma~4.5.7]{CL08}. 
After proving the analogue of \cref{l:rho} in the present setting, we obtain 
the desired push-forward map $b_{\p\s!}^\Z \colon \relIMP{\t}_\p \to \relIMP{\t}_\s$. 

\begin{remark}
As explained in \cite[Section~4.2]{CL08}, one can define a partial ordering in $\A_X$, 
and hence the semiring $\PP_\p^+ \subset \PP_\p$ of functions $\f \ge 0$. 
Restricting then to such functions, one can use part~(2) of \cite[Lemma~4.5.7]{CL08}
to prove a Tonelli type statement analogous to part~(2) of \cref{p:Fubini-Tonelli-for-G}.
\end{remark}

If $\p' \ge \p$ is a higher local snc model over $X$ and 
$\a \colon X_\p' \to X_\p$ is the corresponding map, then we have a
push-forward map $\a_! \colon \relIMP{\t}_{\p'} \to \relIMP{\t}_\p$, hence we can define
\[
\relIMQ{\t}(X^{\val}(\Z)) := \invlim_\p \relIMP{\t}_\p
\quad\text{and}\quad
\relIMQ{Z^{\val}(\Z)}(X^{\val}(\Z)) = \dirlim_\t \relIMQ{\t}(X^{\val}(\Z)),
\]
and $b$ induces push-forward maps 
\[
b_!^\Z \colon \relIMQ{\t}(X^{\val}(\Z)) \to \relIMQ{\t}(Y^{\val}(\Z))
\quad\text{and}\quad
b_!^\Z \colon \relIMQ{Z^{\val}(\Z)}(X^{\val}(\Z)) \to \relIMQ{Z^{\val}(\Z)}(Y^{\val}(\Z)).
\] 
Just like in \cref{t:functoriality}, the assignments $X \mapsto \relIMQ{\t}(X^{\val}(\Z))$
and $X \mapsto \relIMQ{Z^{\val}(\Z)}(X^{\val}(\Z))$ are functorial. Here we use again \cite[Lemma~4.5.7]{CL08}. 

When $Z = X$ and $\t = \id_X$, we denote $\relIMQ{\id_X}(X^{\val}(\Z))$ simply by $\IMQ(X^{\val}(\Z))$. 
Note that $\relIMP{\id_X}_{\id_X} = \MA_X$, hence we have a map
\[
t^\Z_X \colon \relIMQ{X}(X^{\val}(\Z)) \to \MA_X.
\]

\begin{definition}
An element $g \in \IMQ(X^{\val}(\Z))$ is called an \emph{integrable constructible motivic function}. 
Its \emph{integral} is defined to be
\[
\int_{X^{\val}(\Z)} g\, d\m_X^\Z := t^\Z_X(g).
\]
\end{definition}

We have the following analogue of \cref{t:Lambda}. 

\begin{theorem}
There is a natural inclusion $\IMP(X^{\val}(\Z)) \inj \IMQ(X^{\val}(\Z))$ that is compatible
with respect to the corresponding definitions of integral.
\end{theorem}

We finish this section by revisiting \cref{e:ramified-cover} from the viewpoint of atomic measures.
While the answer is of course the same, comparing the computations 
gives a good sense of the difference between the two approaches. 

\begin{example}
Keeping the same notation as in \cref{e:ramified-cover}, consider now $f_\p = [P] \otimes 1$
as an element of $\MP_\p$ and $f$ as an element of $\MP(X^{\val}(\Z))$.
We let $b_R^\Z \colon R^\Z \to S^\Z$
be the restriction of $b_R$, where $R^\Z \subset R$ and $S^\Z \subset S$ are the 
subsets of integral points in the respective coordinates $x$ and $y$. 
We denote by $m$ and $n$ the points of these sets and identify them with their positive integer values. 
Note that $b_R^\Z$ is given by $n=2m$, hence $b_R^\Z(R^\Z) = S^{2\Z}$, the subset of $S^\Z$ of even integral points.
Then $b_!^\Z(f)$ is represented at level $\s$ by
\[
b_{\p\s!}^\Z(f_\p) = b_{\p\s*}[P] \otimes b_{\p\s!}^\Z(1) =
\begin{cases}
[Q] \otimes \L_Y^{-\frac n2 + n} &\text{if $n$ is even}, \\
0 &\text{if $n$ is odd},
\end{cases}
\]
hence
\[
t_Y^\Z(b_!^\Z(f)) = \int_{Y^{\val}(\Z)} b_!^\Z(f)\, d\m_Y^\Z 
= [Q] (\L_Y-1) \sum_{m \in 2\Z_{>0}} \L_Y^{(- \frac n2+n) - n} = [Q].
\]
On the other hand, 
\[
b_*(t_X^\Z(f)) = b_*\Big( \int_{X^{\val}(\Z)} f\, d\m_X^\Z\Big) 
= b_*\Big([P] (\L_X-1) \sum_{n = 1}^\infty \L_X^{-n}\Big) = [Q]
\]
where $b_* \colon \MA_X \to \MA_Y$ is the map induced by push-forward on Grothendieck rings.
\end{example}

\section{Integration over nontrivially valued fields}
\label{s:valued-fields}

In the previous sections, we have focused on the case where $X$ is a variety over a field $k$ with trivial norm. 
Here we outline an analogous theory when $X$ is defined over a 
nontrivially valued field $K$. We will not discuss all aspects of the theory, but only 
give some directions on how it can be developed in this setting.
While there are clear analogies with the case of constant fields, 
there are also some meaningful adjustments that need to be made.
The main result in this section is a formula, stated in \cref{t:comparison},
which relates motivic integrals over a non-trivially valued field
to pairs of motivic integrals over trivially valued fields.

Throughout the section, we consider the following setting.
We let $R = k[[t]]$ and $K = k(\hskip-1pt(t)\hskip-1pt)$ where $k$ is an
algebraically closed field of characteristic zero,
and let $X$ be a proper variety over $K$ with invertible canonical sheaf $\om_X$. We denote $\D := \Spec R$. 
We assume that there exist a variety $Y$ with invertible canonical sheaf, 
a proper flat morphism $Y \to C$ over $k$ where
$C$ is a smooth curve, and a non-constant morphism $\g \colon \D \to C$ over $k$, such that 
$X \cong Y \times_C \Spec K$. 
Let $p \in C$ be the image of $\Spec k \to C$. 

Most of what is discussed in this section can be extended to more general settings
where $R$ is a discrete valuation ring with fraction field $K$ and residue field $k$, 
provided the necessary assumptions on existence of resolution of singularities
and weak factorizations are made. The more restrictive setting considered here
becomes relevant in the last subsection, when we formulate of \cref{t:comparison}.

\subsection{Analytifications over valued fields}

Continuing with the above notation, we regard $K$ as a valued field, 
with valuation $v := \ord_t\colon K^\times \to \G_v \cong \Z$. 
To regard $K$ as a Banach algebra requires fixing 
an embedding of $\G_v$ in $\R$ and viewing $v$ as a real valuation. 
A standard choice is to normalize $v$ so that $v(t) = 1 \in \R$.
i.e., to embed $\G_v \inj \R$ as the set of integers. 
Let $\|\bla\| := e^{-v}$ be the corresponding norm, and
let $X^\an$ be the analytification of $X$ over $(K,\|\bla\|)$. 
We will also denote such analytification by $X^\an_1$
when we want to remember that we fixed the normalization $v(t) = 1 \in \R$.

One can vary the normalization of $v$ by multiplying by any positive real number $b$,
so that the valuation takes value $b$ on $t$. This
corresponds to regarding $K$ with the norm given by $\|\bla\|^b$.  
For every choice of $b$, we consider the Berkovich analytification of $X$ over $(K,\|\bla\|^b)$, 
which we denote by $X^\an_b$. 
By definition, a point $x \in X^\an_b$ corresponds to a multiplicative seminorm $|\bla|_x$ on $X$
that restricts to $\|\bla\|^b$ on $K$; we set $v_x := -\log |\bla|_x$. 

All the analytifications $X^\an_b$ are canonically 
isomorphic to each other, but the choice of $b$ may in principle affect the measure we wish to define. 
In fact, it is unclear how to define a measure directly on any of the spaces $X^\an_b$.
It turns out that is more natural, from the point of view of this paper, 
to consider all analytifications $X^\an_b$ at once. 
This can be done as follows (the construction does not require properness
and will be also applied to open subsets of $X$).

The set of norms of the form $\|\bla\|^b$ form an interval $(0,\infty)$, with each 
point $b \in (0,\infty)$ identified with the norm $\|\bla\|^b$. 
We denote by $X^\an_{(0,\infty)}$ the set of multiplicative seminorms
on $X$ that restrict to a norm of the form $\|\bla\|^b$ on $K$ for some $b > 0$. 
Set theoretically, we have
\[
X^\an_{(0,\infty)} = \bigcup_{b > 0} X^\an_b.
\]
We equip this space with the weakest topology such that for every open set $U \subset X$
and every function $h \in \O_X(U)$, the valuation map $x \mapsto v_x(h)$ is continuous
on $U^\an_{(0,\infty)} \subset X^\an_{(0,\infty)}$. 
The map
$\lambda \colon X^\an_{(0,\infty)} \to (0,\infty)$ given by $x \mapsto v_x(t)$
is continuous, with fiber over a point $b \in (0,\infty)$ equal to the analytification $X^\an_b$,
and there is a naturally defined map
$\theta \colon X^\an_{(0,\infty)} \to X^\an_1 = X^\an$ given by $x \mapsto y$
where $y$ is characterized by the condition $|h|_y = |h|_x^{1/\lambda(x)}$ for all local functions $h \in \O_X(U)$
with $y \in U^\an_{(0,\infty)}$.
These maps yield a canonical homeomorphism
\[
X^\an_{(0,\infty)} \xrightarrow{\sim} X^\an \times (0,\infty), \quad x \mapsto (\theta(x),\lambda(x)).
\]

By assigning to the interval $(0,\infty)$ measure 1 using the volume form $e^{-u}du$, we will interpret the measure 
we are going to define on $X^\an_{(0,\infty)}$ as an `average of measures' on the fibers $X^\an_b$. 
This will lead to our definition of measure on $X^\an$.

\subsection{Quasi-monomial valuations}

Continuing with the same setting, 
we adopt the following notion of model (and snc model) over $R$. 

\begin{definition}
An \emph{$R$-model for $X$} is a proper flat scheme $\cX_\p$ over $R$
endowed with a proper birational morphism $\p \colon (\cX_\p)_K \to X$.
\end{definition}

\begin{definition}
\label{d:snc-R-model}
A \emph{snc $R$-model for $X$} consists of an $R$-model $\cX_\p$ for $X$
equipped with a divisor $D$ on $\cX$, for which
there exists a snc model $(Y_\s,E)$ over $Y$, with
$E$ containing in its support the fiber of $Y_\s$ over $p \in C$, 
such that $\cX_\p = Y_\s \times_C \D$ and $D = E \times_C \D$.
We refer to $D$ as the \emph{boundary divisor} of the snc $R$-model.
\end{definition}

For ease of notation, we may refer to a snc $R$-model by writing $(\cX_\p,D)$ 
(or just $\p$, if no risk of confusion is likely to arise), but we stress
that all the information listed above is part of the datum of the model. 

By definition, a snc $R$-model $(\cX_\p,D)$ has the following properties:
\begin{enumerate}
\item
$\cX_\p$ is a regular scheme and $\p \colon (\cX_\p)_K \to X$ is a resolution of singularities;
\item
$\p \colon (\cX_\p)_K \to X$ factors through the Nash blow-up of $X$; 
\item
$D$ contains in its support the exceptional locus of $\p$ and the reduced fiber $((\cX_\p)_k)_\red$;
\item
$D$ is a reduced snc divisor and every stratum of $D$ is irreducible. 
\end{enumerate}

\begin{remark}
The above definition of $R$-model is different from what is usually given in the literature, 
as we do not require $\p$ to be an isomorphism and furthermore we allow the boundary divisor to
have components away from the fiber over $k$. While working with the usual definition 
(i.e., requiring $\p$ to be an isomorphism and the boundary divisor to be supported
on the fiber over $k$) 
may suffice in order to define a reasonable measure on the Berkovich space, by enlarging the class
of $R$-models we obtain a larger class of measurable sets and integrable functions. 
This larger generality also allow us to work with singular varieties over $K$.
\end{remark}

Given two snc $R$-models $(\cX_\p,D)$ and $(\cX_{\p'},D')$, we write
$\p' \ge \p$ whenever the birational map $\a \colon \cX_{\p'} \rat \cX_\p$ is a morphism
and $\Supp(\a^*D) \subset \Supp(D)$. 
Note that the condition that $\a$ is a morphism is stronger than just requiring that $\p' \colon (\cX_{\p'})_K \to X$
factors through $\p \colon (\cX_\p)_K \to X$.
Without loss of generality, we may assume that whenever we have $\p' \ge \p$, the corresponding
snc models $(Y_\s,E)$ and $(Y_{\s'},E')$ are chosen so that $\s' \ge \s$.

\begin{definition}
Given a snc $R$-model $(\cX_\p,D)$ and writing $D = \sum_{i=1}^r D_i$, 
we say that an index set $I \subset \{1,\dots,r\}$ is \emph{over $k$}, 
and write $I/k$, if the stratum $D_I = \bigcap_{i \in I} D_i$ is contained in the central fiber $\cX_k$.
Similarly, we say that $I$ is \emph{over $K$}, 
and write $I/K$, if the stratum $D_I = \bigcap_{i \in I} D_i$ intersects the generic fiber $(\cX_\p)_K$.
\end{definition}

Given any snc $R$-model $(\cX_\p,D)$, we associate to it the set of 
quasi-monomial valuations $\Sk_{\p/k} \subset X^\an_{(0,\infty)}$ 
determined by the boundary divisor $D = \sum D_i$. 
More precisely, $\Sk_{\p/k}$ is the set of valuations on $X$ that are monomials in the toroidal coordinates
on $\cX_\p$ determined by $D$ at the generic point of strata $D_I$ with $I/k$. 
We set $X^\qm_{(0,\infty)} := \bigcup_\p \Sk_{\p/k}$.
For every $\p$, there is a decomposition $\Sk_{\p/k} = \bigsqcup_{I/k} \Sk_{\p,I}$. 
A subset $P \subset \Sk_\p$ is a \emph{potential face} if there is a snc model $\p' \ge \p$
such that $P$ is a face in $\Sk_{\p'/k}$.

Assuming for a moment that $\p \colon (\cX_\p)_K \to X$ is an isomorphism and $D$ is supported in the fiber 
$(\cX_\p)_k$, for every $b \in (0,\infty)$ the intersection $\Sk_{\p/k} \cap X^\an_b$ 
is the skeleton $\Sk_\p(X^\an_b)$ of $X^\an_b$ as constructed in \cite{MN15},
and $\Sk_{\p/k}$ can be viewed as the fan over this skeleton with the origin removed. 
Just like in the case of constant fields discussed in \cref{s:measure},
for every $c$ there are continuous retractions $X^\an_b \to \Sk_\p(X^\an_b)$,
and these retractions glue together to continuous retractions 
$r_\p \colon X^\an_{(0,\infty)} \to \Sk_{\p/k}$ \cite[Section~(3.1.5)]{MN15}.

If $D$ is not supported in $(\cX_\p)_k$, then 
$\Sk_{\p/k}$ is not a fan over a simplicial complex, as it does not contain all
the boundary faces of its simplexes.
For instance, if $I$ is an index set over $k$ which contains an index $i \in I$
such that the component $D_i$ is not contained in $(\cX_\p)_k$, then 
the ray corresponding to $\{i\}$, which we would expect to see in the boundary
of $\Sk_{\p,I}$, is not in $\Sk_{\p/k}$. 

In general, we view $\Sk_{\p/k}$ both as a subset of $X^\an_{(0,\infty)}$ and of $((\cX_\p)_K)^\an_{(0,\infty)}$. 
There is a continuous retraction $r_\p \colon ((\cX_\p)_K)^\an_{(0,\infty)} \to \Sk_{\p/k}$, and for any set
$S \subset \Sk_{\p/k}$ we consider the set $\p^\an_{(0,\infty)}(r_\p^{-1}(S)) \subset X^\an_{(0,\infty)}$
where $\p^\an_{(0,\infty)} \colon ((\cX_\p)_K)^\an_{(0,\infty)} \to X^\an_{(0,\infty)}$
is the morphism induced by $\p$.

\subsection{Motivic  measure}

In order to define a measure on $X^\an_{(0,\infty)}$
(and hence on $X^\an$), we need a replacement of the 
Mather log discrepancy function. This is given by the \emph{weight function}. This function
was introduced in \cite[Section~4]{MN15} assuming that $X$ is smooth.
Recall that we allow $X$ to be singular, but assume that the canonical sheaf $\om_X$
is an invertible sheaf. Given that we allow, in our definition of snc model, to blow up
the generic fiber and hence pass to a resolution of $X$, assuming that the canonical sheaf
is an invertible is really all we need. 
In a nutshell, one fixes a nonzero rational canonical form $\om$ on $X$ 
(i.e., a nonzero rational section of $\om_X$) and defines the weight function to be
the unique continuous function $\wt_\om \colon X^\qm_{(0,\infty)} \to \R$ such that 
for every divisorial valuation $q\val_E \in X^\an_{(0,\infty)}$, 
and every snc $\R$-model $\cX_\p$ on which the center of $\val_E$ has of codimension 1
(i.e., is an irreducible component of the central fiber $(\cX_\p)_k$), the function 
takes value
\[
\wt_\om(q\val_E) = q \val_E\big(\div_{\cX}(\om) + ((\cX_\p)_k)_\red\big)
\]
Note that the given rational section of $\om_X$ defines a canonical divisor $K_X$, 
which is a Cartier divisor on $X$.

Let $(\cX_\p,D)$ be a snc model for $X$. 
For every nonempty index set $I = \{i_1,\dots,1_s\}$ over $k$ we have an isomorphism
$\ff_{\p,I} \colon \Sk_{\p,I} \xrightarrow\sim \R^s_{>0}$ 
given by $v \mapsto (v(D_{i_1}),\dots,v(D_{i_s}))$.
We consider the measure on $\Sk_{\p,I}$ given by
\[
\n_{\p,I,\om} := e^{-\wt_\om} \. \ff_{\p,I}^*(\n), 
\]
where $\n$ is the Lebesgue measure on $\R^s_{>0}$. 
If $P \subset \Sk_{\p/k}$ is a potential face, and we have $P = \Sk_{\p',I'}$, 
then we set $\n_P := \n_{\p',I',\om}$. 
We denote by $\S_\om(\Sk_{\p/k})$ the collection of subsets $S \subset \Sk_{\p/k}$ 
such that $S \cap P$ is $\n_P$-measurable for every potential face $P$,
and let $\S_{\p,\om}(X^\an_{(0,\infty)})$ be the collection of subsets of $X^\an_{(0,\infty)}$ of the form
$r_\p^{-1}(S)$ where $S \in \S_\om(\Sk_{\p/k})$. For any such set, 
we define 
\[
\m_{\p,\om}\big(\p^\an_{(0,\infty)}(r_\p^{-1}(S))\big) := \sum_{I \ne \emptyset} [D_I^\o] \,\n_{\p,I,\om}(S \cap \Sk_{\p,I})
\]
where the class $[D_I^\o]$ is considered in the motivic ring $\MM(k)$.
Essentially the same proof of \cref{t:existence-measure} gives the following result.

\begin{theorem}
Denoting $\S_\om(X^\an_{(0,\infty)}) := \bigcup_\p \S_{\p,\om}(X^\an_{(0,\infty)})$
where the union is taken over all snc models $(\cX,D,\p)$,
the measures $\m_{\p,\om}$ glue together to give a measure
\[
\m_\om \colon \S_\om(X^\an_{(0,\infty)}) \to \MM(k).
\]
\end{theorem}

\begin{definition}
We let
\[
\ov\S_\om(X^\an) := \{ T \subset X^\an \mid \theta^{-1}(T) \in \S_\om(X^\an_{(0,\infty)})\}
\]
and define $\ov\m_\om \colon \ov\S_\om(X^\an) \to \MM(k)$ by setting
$\ov\m_\om(T) := \m_\om(\theta^{-1}(T))$.
We call $\ov\m_\om$ the \emph{motivic  measure} on $X^\an$. 
\end{definition}

\begin{remark}
The measure $\ov\m_\om(T)$ can be thought of as an `average measure' of the sets $\theta^{-1}(T) \cap X^\an_b$
as $b$ varies in $(0,\infty)$. 
\end{remark}

\subsection{Integration}

The notion of integrability defined in \cref{s:functions}
can be adapted to the current setting, as follows.

Let $\Func(\Sk_{\p/k}\cap X^\an,\R)$ is the ring of real valued functions
on $\Sk_\p \cap X^\an$. For every element $\f$ in this ring, we consider the function
$\f\o\theta_\p \colon \Sk_{\p/k} \to \R$
where $\theta_\p$ is the restriction of $\theta$ to $\Sk_{\p/k}$;
note that $\theta_\p(\Sk_{\p/k}) = \Sk_{\p/k} \cap X^\an$. 
Let then $\Func_0(\Sk_{\p/k}\cap X^\an,\R)$
be the ideal of functions $\f \in \Func(\Sk_{\p/k}\cap X^\an,\R)$
such that $\f\o\theta_\p$ almost-everywhere zero on every potential face $P$
with respect to the corresponding measure $\n_P$. 
For any snc $R$-model $(\cX_\p,D)$, we define
\[
\FF_\p := \Func(\Sk_{\p/k} \cap X^\an,\R)/\Func_0(\Sk_{\p/k} \cap X^\an,\R),
\]
and denote by $\FF_\p^\o \subset \FF_\p$ the subring
generated by the characteristic functions $1_{\Sk_{\p,I}\cap X^\an}$ with $I/k$.

\begin{definition}
An element $\f \in \FF_\p$ called a \emph{function of level $\p$}. 
We say that $\f$ is \emph{integrable} if the restriction of $\f \o \theta_\p$ to any potential face $P$
of $\Sk_{\p,I}$ is integrable with respect to the corresponding measure $\n_P$. 
We let $\IF_{\p,\om} \subset \FF_\p$
be the subspace of integrable functions of level $\p$.
Using the natural injective maps $\IF_{\p,\om} \inj \IF_{\p',\om}$ defined for $\p' \ge \p$, 
we define the \emph{space of integrable functions} on $X^\an$ to be
$\IF_\om(X^\an) := \dirlim_\p \IF_{\p,\om}$.
\end{definition}

Setting $\FF^\o(X^\an) := \dirlim_\p \FF_\p^\o$
there is a natural $\FF^\o(X^\an)$-module structure on $\MM(k)$ that is defined 
similarly as in \cref{s:functions}.

\begin{definition}
We define the \emph{module of integrable motivic functions} on $X^\an$ to be
\[
\IMF_\om(X^\an) := \MM(k) \otimes_{\FF^\o(X^\an)}\IF_\om(X^\an).
\]
\end{definition}

Every element $f \in \IMF_\om(X^\an)$ can be written as a finite sum 
$\sum_j[V_j] \otimes \f_j$ where $[V_j]$ are the classes in $\MM(k)$ 
determined by $k$-varieties $V_j$ and
$\f_j \in \IF_{\p,\om}$ for a suitable choice of snc $R$-model $(\cX_\p,D)$. 

\begin{definition}
Using the above notation, 
the \emph{integral} of $f\in \IMF_\om(X^\an)$ is the element in $\MM(k)$ given by
\[
\int_{X^\an} f\,d\ov\m_\om := \sum_j \sum_{I/k} [V_j \times D_I^\o] \int_{\Sk_{\p,I}} (\f_j \o \theta_\p) \, d\n_{\p,I,\om}.
\]
\end{definition}

The same proof of \cref{t:integral-well-defined} shows that the integral is well-defined. 

\begin{remark}
The integral defined above should be interpreted as an `average integral' 
on the analytifications $X^\an_b$ where $b$ ranges in $(0,\infty)$.
\end{remark}

Fixing a divisor or an ideal sheaf on $X$ does not determine
a function on $X^\an$. This is because units in $\O_X$ may take nontrivial values on this space. 
Working locally on $X$, one needs to fix generators (or equivalently, to work with functions).
Globally, one can work with sections (and rational sections) of line bundles on $X$.  
Another way of constructing functions on $X^\an$ is to fix an $R$-model 
and look at divisors and ideals sheaves on it.

\begin{example}
Let $(\cX_\p,D)$ be a snc $R$-model, and $\cB = \sum b_i D_i$ an $\R$-divisor on $\cX_\p$
supported within the support of $D$ and with coefficients $b_i > - w_i$ where $w_i := \wt_\om(\val_{D_i})$. 
Then $e^{-\ord(\cB)}$ is an integrable function on $X^\an$, and
\[
\int_{X^\an} e^{-\ord(\cB)}\,d\ov\m_\om 
= \sum_{I/k} [D_I^\o]\int_{\R_{> 0}^{|I|}} e^{-\sum_{i \in I}(b_i+w_i)x_i} dx_1\cdots dx_{|I|}
= \sum_{I/k} \frac{[D_I^\o]}{\prod_{i \in I} (b_i+w_i)}.
\]
\end{example}

\begin{example}
\label{eg:subscheme-spreading}
Given our setting, where $X = Y \times_C \Spec K$, we can consider the $R$-model 
$\cX := Y \times_C \D$.
Let $\cZ \subset \cX$ be a subscheme that spreads over $C$, that is, 
of the form $\cZ = Z \times_C \D$ where $Z \subset Y$ is a closed subscheme. 
Let $(Y_\s,E) \to Y$ be a log resolution of $(Y,Z)$, and let
$(\cX_\p,D)$ be the snc $R$-model for $X$ obtained by base change from $(Y_\s,E)$. 
Write $\I_\cZ \.\O_{\cX} = \O_\cX(-b_iD_i)$.
We then reduce to the previous example to define integrals $\int_{X^\an}|\I_\cZ|^s\,d\ov\m_\om$
for $s \in \R$ (under some condition on $s$). 
\end{example}

\begin{remark}
Similarly to what we did in \cref{s:atomic}, one can develop a parallel theory
where the Lebesgue measure on the faces of the fan is replaced by a motivic atomic
measure concentrated on the $\Z$-valued valuations. 
The resulting integral takes value in $\MA_k$. 
We leave the details to the interested reader.
It is useful to remark that the integral valuations supporting the measure lie on
different fibers of $X^\an_{(0,1)} \to (0,1)$. In particular, it is not clear to us
whether a similar theory can be developed by just working on the Berkovich space $X^\an$
over $(K,\|\bla\|)$ without allowing any rescaling of the valuation on $K$. 
\end{remark}

\subsection{Comparison theorem}

We continue with the setting fixed throughout the section. 
Our goal is to compare motivic integration on $X^\an$ to motivic integration on the 
analytification $X^\an_0$ of $X$ over $K$ regarded as a field with trivial norm.

We enlarge the space $X^\an_{(0,\infty)}$ to include $X^\an_0$.
This gives rise to an hybrid space $X^\an_{[0,\infty)}$
with a continuous map
$\~\lambda \colon X^\an_{[0,\infty)} \to [0,\infty)$.
The point $0 \in [0,\infty)$ corresponds to the trivial norm on $K$, 
and $X^\an_0$ is the fiber over this point. 
We stress that this fiber is not homeomorphic to the other fibers $X^\an_b$. 
The construction of $X^\an_{[0,\infty)}$ is similar to the hybrid space constructed in \cite{Ber09}
when $X$ is a variety defined over the complex numbers. 

Recall that $X = Y \times_C \Spec K$ where $Y$ is a variety 
and $C$ is a smooth curve over $k$, and 
that we are assuming that the canonical sheaves $\om_X$ and $\om_Y$ are invertible. 
As in \cref{eg:subscheme-spreading}, we consider the $R$-model
$\cX := Y \times_C \D$.
Note that for this model the map $\cX_K \to X$ is an isomorphism.
Let $Y^\an$ be the analytification of $Y$ over $k$ with the trivial norm, 
and $Y^{\val} \subset Y^\an$ the space of real valuations on $Y$. 
Similarly, let $X^{\val}_0 \subset X^\an_0$ denote the space of real valuations on $X$; 
we keep the subsript $0$ to remember that we here are regarding $K$ with the trivial norm.
We can assume that the rational canonical form $\om$ we fixed on $X$
is the restriction of a rational canonical form $\~\om$ on $Y$.
Let $K_X$ and $K_Y$ be the canonical divisors defined by this form on $X$ and $Y$, respectively. 
Let $B$ be a Cartier divisor on $Y$, and let $\cB := B \times_C \D$
and $B_K := B \times_C \Spec K$. 

The following theorem shows that the motivic integral over the valued field $K$ can be 
computed as a difference of two motivic integrals over constant fields ($K$ with the trivial norm and $k$, 
respectively). 

\begin{theorem}
\label{t:comparison}
With the above notation, for every $s \in \R$ we have
\[
\int_{X^\an} e^{-\ord(s\cB)}\,d\ov\m_\om = 
\int_{Y^{\val}} e^{-\ord(K_Y+sB)} \,d\m_Y - 
\int_{X^{\val}_0} e^{-\ord(K_X+sB_K)} \,d\m_X
\]
in $\MM(\D)$. Here, the integral are viewed as taking values in $\MM(\D)$ via the natural maps
$\MM(k) \to \MM(\D)$, $\MM(Y) \to \MM(C) \to \MM(\D)$, and $\MM(X_K) \to \MM(K) \to \MM(\D)$, 
respectively. 
\end{theorem}

\begin{proof}
Let $(\cX_\p,D)$ and $(Y_\s,E)$ be snc models as in \cref{d:snc-R-model}, 
with the latter gives a log resolution of $(Y,B)$. 
We write $D = \sum_{i=1}^rD_i$ and $E = \sum_{j=1}^m E_j$, and let 
$\b \colon \{1,\dots,r\} \to \{1,\dots,m\}$ be the function defined by the property that
$D_i$ is an irreducible component of $E_{\b(i)} \times_C\D$. 
For every $J \subset \{1,\dots,m \}$, we denote by $\I(J)$
the collection of subsets $I \subset \{1,\dots,r\}$ such that 
$\b$ restricts to a bijection $\b|_I \colon I \to J$. 

Recall that the snc $R$-model $(\cX_\p,D)$ determines the set of 
quasi-monomial valuations 
$\Sk_{\p/k} = \bigsqcup_{I/k} \Sk_{\p,I}$
in $X^\an_{(0,\infty)}$.
This snc $R$-model also gives a snc model over $X$, namely, $((\cX_K)_\p,D_K)$, and associated to this model
we have the skeleton
$\Sk_{\p/K} := \bigsqcup_{I/K} \Sk_{\p,I}$
in $X_0^{\val}$.
In a similar fashion, the snc model $(Y_\s,E)$ over $Y$ determines a skeleton 
$\Sk_\s \subset Y^{\val}$. 
For any $J \subset \{1,\dots,m\}$, we write $J/\D$ if the image of the stratum $E_J$ 
in $C$ intersects the image of $\D$. For any such $J$, we write $J/k$ if $E_J$ maps to the image
of the closed point of $\D$, and $J/K$ is $E_J$ maps to the image of the generic point of $\D$
(i.e., if $E_J$ dominates $C$).
Setting $\Sk_{\s/\D} = \bigsqcup_{J/\D} \Sk_{\s,I}$, $\Sk_{\s/k} = \bigsqcup_{J/k} \Sk_{\s,I}$ and 
$\Sk_{\s'/K} = \bigsqcup_{J/K} \Sk_{\s',I}$, we have
\[
\Sk_\s \supset \Sk_{\s/\D} = \Sk_{\s/k} \sqcup \Sk_{\s/K}.
\]
We can disregard everything in $\Sk_\s \setminus \Sk_{\s/\D}$
since for any index set $J$ that is not over $\D$, the class
$[E_J]_Y$ is in the kernel of the map $\MM(Y) \to \MM(\D)$. 
In the following, all motivic classes will be regarded in $\MM(\D)$. 

Setting $a_j := \^A_Y(\val_{E_j})$, $k_j := \ord(K_Y)(\val_{E_j})$, and $b_j := \ord(B)(\val_{E_j})$, 
the equality
\[
\int_{Y^{\val}} e^{-\ord(K_Y+sB)} \,d\m_Y = 
\sum_{J/\D}[E_J^\o\times_C\D] \int_{\R_{>0}^{|J|}} e^{-\sum_{j\in J}(a_j+k_j+b_j)x_j}d\n
\]
holds in $\MM(\D)$. 
The formula stated in the \lcnamecref{t:comparison} will follow by breaking 
the sum into two sums according to whether $J$ is over $k$ or over $K$.  

If $J/k$, then $\I(J)$ consists of only one index set $I$, and we have $[E_J^\o\times_C\D] = [D_I^\o]$. 
Note that all index sets $I/k$ with $D_I \ne \emptyset$ are realized in this way. 
Furthermore, the bijection $\b|_I \colon I \to J$ induces a linear isomorphism 
$\Sk_{\p,I} \simeq \Sk_{\s,J}$ sending $\val_{D_i}$ to $\val_{E_\b(i)}$.
For every $i \in I$, we have
$\wt_\om(\val_{D_i}) = \^A_Y(\val_{E_{\b(i)}}) - \ord(K_Y)(\val_{E_{\b(i)}})$
and $\ord(\cB)(\val_{D_i}) = \ord(B)(\val_{E_{\b(i)}})$. 
Therefore we have
\[
\int_{X^\an} e^{-\ord(s\cB)}\,d\ov\m_\om =
\sum_{J/k}[E_J^\o\times_C\D] \int_{\R_{>0}^{|J|}} e^{-\sum_{j\in J}(a_j+k_j+b_j)x_j}d\n
\]
in $\MM(\D)$.

If $J/K$, then $\I(J)$ may consist of several index sets, each one over $K$, 
and we have $[E_J^\o\times_C\D] = \sum_{I \in \I(J)}[D_I^\o]$. 
Every index set $I/K$ with $D_I \ne \emptyset$ appear in $\I(J)$ for some $J/K$.
For every $I \in \I(J)$, the bijection $\b|_I \colon I \to J$ induces a linear isomorphism 
$\Sk_{\p,I} \simeq \Sk_{\s,J}$ sending $\val_{D_i}$ to $\val_{E_\b(i)}$, and for every $i \in I$ we have
$\^A_X(\val_{D_i}) = \^A_Y(\val_{E_{\b(i)}})$, 
$\ord(K_X)(\val_{D_i}) = \ord(K_Y)(\val_{E_{\b(i)}})$,
and $\ord(B_K)(\val_{D_i}) = \ord(B)(\val_{E_{\b(i)}})$.
Therefore we see that
\[
\int_{X_0^{\val}} e^{-\ord(K_X+B_K)}\,d\ov\m_\om =
\sum_{J/K}[E_J^\o\times_C\D] \int_{\R_{>0}^{|J|}} e^{-\sum_{j\in J}(a_j+k_j+b_j)x_j}d\n
\]
in $\MM(\D)$.

Combining everything, we get the formula stated in the \lcnamecref{t:comparison}.
\end{proof}

\end{document}